\documentclass[reqno]{amsart}
\usepackage{amsmath}
\usepackage{amssymb}
\usepackage{amscd}
\usepackage{graphics}
\usepackage{latexsym}
\usepackage{stmaryrd}
\usepackage{xy}
\usepackage{hyperref}

\theoremstyle{plain}
\newtheorem{theorem}{Theorem}[section]
\newtheorem{thm}[theorem]{Theorem}
\newtheorem{cor}[theorem]{Corollary}
\newtheorem{lem}[theorem]{Lemma}
\newtheorem{prop}[theorem]{Proposition}
\newtheorem{claim}[theorem]{Claim}

\theoremstyle{definition}
\newtheorem{defn}[theorem]{Definition}

\newtheorem{rmk}[theorem]{Remark}

\newtheorem{notat}[theorem]{Notation}

\newtheorem{hyp}[theorem]{Hypothesis}

\theoremstyle{remark}

\input xy
\xyoption{all}

\newcommand{\marpar}[1]{}
\newcommand{\mni}{\medskip\noindent}
\newcommand{\ZZ}{\mathbb{Z}}

\newcommand{\AAA}{\mathbb{A}}
\newcommand{\PP}{\mathbb{P}}

\newcommand{\mc}{\mathcal}

\newcommand{\OO}{\mc{O}}

\newcommand{\ol}{\overline}
\newcommand{\ul}{\underline}

\newcommand{\SP}{\text{Spec }}

\newcommand{\M}{\overline{\mc{M}}}

\newcommand{\m}{\overline{\text{M}}}

\newcommand{\kgnb}[1]{\m_{#1}}
\newcommand{\Kgnb}[1]{\M_{#1}}

\newcommand{\Pine}[1]{\text{Porcupine}^{#1}}
\newcommand{\Ch}{\text{Chain}}
\newcommand{\FCh}{\text{FreeChain}}

\newcommand{\PCh}{\text{PeaceChain}}
\newcommand{\Sec}{\text{Sections}}
\newcommand{\Co}{\text{Comb}}
\newcommand{\peaceful}{\text{peaceful}}
\newcommand{\Peaceful}{\text{Peaceful}}
\newcommand{\pax}{\text{pax}}

\newsavebox{\sembox}
\newlength{\semwidth}
\newlength{\boxwidth}

\newsavebox{\semrbox}
\newlength{\semrwidth}
\newlength{\boxrwidth}

\def\Hom{{\rm Hom}\,}

\def\boxtensor{\mathop{\vbox{\hrule\hbox{\vrule\kern5pt%
\vbox{\kern5pt}\vrule}\hrule}%
\kern-7pt\raise .5pt\hbox{$\times$}}\nolimits}

\title
{Families of rationally simply connected varieties
  over surfaces and torsors for semisimple groups}

\author[de Jong]{A. J. de Jong}
\address{Department of Mathematics \\
Columbia University \\
New York, NY 10027}
\email{dejong@math.columbia.edu}

\author[He]{Xuhua He}
\address{Department of Mathematics \\
  Stony Brook University \\ Stony Brook, NY 11794}
\email{hugo@math.sunysb.edu} 

\author[Starr]{Jason Michael Starr}
\address{Department of Mathematics \\
  Stony Brook University \\ Stony Brook, NY 11794}
\email{jstarr@math.sunysb.edu} 

\date{\today}

\begin{document}

%%%%%%%%%%%%%%%%%%%%%%%%%%%%%%%%%%%%%%%%%%%%%%%%%%%%%%%%%%%%%%%%%%%%
%%
%% Abstract
%%
%%%%%%%%%%%%%%%%%%%%%%%%%%%%%%%%%%%%%%%%%%%%%%%%%%%%%%%%%%%%%%%%%%%%

\begin{abstract}
Under suitable hypotheses, we prove that a form of a projective
homogeneous variety $G/P$ defined over the function field of a surface
over an algebraically closed field has a rational point. The method
uses an algebro-geometric analogue of simple connectedness replacing
the unit interval by the projective line. As a consequence, we
complete the proof of Serre's Conjecture II in Galois cohomology for
function fields over an algebraically closed field.
\end{abstract}

%%%%%%%%%%%%%%%%%%%%%%%%%%%%%%%%%%%%%%%%%%%%%%%%%%%%%%%%%%%%%%%%%%%%%%
%%
%% Body
%%
%%%%%%%%%%%%%%%%%%%%%%%%%%%%%%%%%%%%%%%%%%%%%%%%%%%%%%%%%%%%%%%%%%%%%%

\maketitle

\tableofcontents

%%%%%%%%%%%%%%%%%%%%%%%%%%%%%%%%%%%%%%%%%%%%%%%%%%%%%%%%%%%%%%%
%%
%% Section 1 : Introduction
%% 
%%%%%%%%%%%%%%%%%%%%%%%%%%%%%%%%%%%%%%%%%%%%%%%%%%%%%%%%%%%%%%%

\section{Introduction}
\label{sec-intro}
\marpar{sec-intro}

\noindent
In this introduction we work with varieties over an algebraically
closed field $k$ of characteristic $0$, unless mentioned otherwise.
In the paper \cite{GHS}
it was shown that a family of birationally rationally connected
varieties over a curve has a rational section.
A variety $X$ is called {\it birationally
rationally connected} if a general pair of points $(x,y)\in X\times X$
can be connected by a rational curve, i.e., if there exists an open
$U \subset \PP^1$ and a morphism $f : U \to X$ such that $x,y\in f(U)$.
This is related to Tsen's theorem, see \cite{Lang52}, which says
that a family of hypersurfaces of degree $d$ in $\PP^n$ over
an $r$-dimensional base variety has a section if $d^r \leq n$.
The relation is that a smooth hypersurface in $\PP^n$ is birationally
rationally connected if and only if $d \leq n$.

\medskip\noindent
How to generalize the result of \cite{GHS} to families of varieties
over a surface? By analogy with topology one could hope for a
statement of the following type: A family of (birationally) rationally
simply connected varieties over a surface has a rational section.
There are two good reasons why this is too simplistic: (1) it seems
hard to come up with a good notion of ``rationally simply connected''
varieties, and (2) Brauer-Severi schemes give counter examples.

\medskip\noindent
Let $S$ be a smooth projective surface. A Brauer-Severi scheme over $S$
is a smooth projective morphism $X \to S$ all of whose geometric fibres are
isomorphic to $\PP^n$. There are plenty of examples of such Brauer-Severi
schemes such that $X \to S$ does not have a rational section, because
most surfaces have infinite Brauer groups, see \cite{BrauerI}, and
\cite[Corollaire 3.4]{BrauerII}. Since the authors would like to think
of $\PP^n$ as a prime example of a rationally simply connected variety
this seems to dash all hopes.

\medskip\noindent
Not so! It turns out that in addition to a condition on the fibres of the
family we need to impose additional global conditions of a cohomological
nature. The goal is to find example theorems providing evidence for the
principle which we have stated elsewhere:
\begin{list}{(P)}
\item ``The only obstruction
to having a rational section of a family of rationally simply connected
varieties over a surface should be a Brauer class on the base''.
\end{list}
\noindent
This should not be taken too literally, e.g., the Brauer class
could be an element of $H^2_{et}(S, T)$ for some torus $T$ over
$S$ (or an open part of $S$); or perhaps something more general,
cf.\ \cite{Hamel}.

\medskip\noindent
Let $Y$ be a smooth projective rationally connected variety.
When do we call $Y$ rationally simply connected? One of the main
obstacles is that the space of rational curves on $Y$ has many
irreducible and even connected components; simply because
there are many curve classes $\beta \in H_2(Y, \ZZ)$ which
can be represented by rational curves. Here is a partial list of
conditions that we have considered using as a definition:
\begin{enumerate}
\item
\label{Ind-A1}
The $\text{Ind}$-scheme parametrizing
maps $\AAA^1 \to Y$ is birationally rationally connected.
\item
\label{curves-in-surface}
For any pair of rational curves $C_1, C_2 \subset Y$ there
exists a rational surface $S \subset Y$ containing $C_1$ and $C_2$.
\item
\label{moduli-curves}
For all sufficiently positive curve classes $\beta$ there
exists an irreducible component $Z_\beta$ of $\Kgnb{0,0}(Y, \beta)$
which is rationally connected.
\item
\label{moduli-curves-pointed}
A variant of (\ref{moduli-curves}) which considers for every $n \geq 0$
the moduli spaces of rational curves passing through $n$ general points
on $X$. See the notion labeled ``strongly rationally simply connected''
in \cite{dJS10}.
\item
The variety $Y$ is \emph{rationally simply connected by chains of free
lines}, i.e., $Y \to \text{Spec}(k)$ satsifies Hypothesis \ref{hyp-peace}.
This implies that for a general pair of points $x,y\in Y(k)$, and
$n \gg 0$ the moduli space of length $n$ chains of free lines
connecting $x$ to $y$ is itself birationally rationally connected
(see Proposition \ref{prop-peace}).
\end{enumerate}
There is a subtle interplay between the formulation of the property,
the problem of proving a suitable variety $Y$ has the property and the
problem of proving the principle (P) holds for families of varieties
having the property. For example, trying to use a property such
as (\ref{Ind-A1}) above to prove (P) one runs into the problem of not
having a result like \cite{GHS} for families of ind-schemes.

\medskip\noindent
In the study of rationally connected varieties a key technical tool
is the notion of a \emph{very free rational curve}. Such a curve is
given by $f : \PP^1 \to Y$ such that $f^* T_Y$ is an ample vectorbundle.
Similarly, the key technical tool in the settting of rational
simple connectedness is the notion of a \emph{very twisting surface}.
This notion was first introduced in the paper \cite{HS2}. In the
paper you are reading now, a very twisting surface is given by a
morphism $f : \PP^1 \to \Kgnb{0,1}(Y, 1)$ such that $f^* T_\Phi$,
$f^* T_{\text{ev}}$ are sufficiently ample and such that
$\text{ev} \circ f$ is a free curve in $Y$, see
Lemma \ref{lem-twistexist} and Definition \ref{defn-twistabs}.
See \cite{HS2}, \cite{Sr1c}, and \cite{dJS10} for example usages of this
notion; allthough we warn the reader that each time the definition
is slightly different. A smooth projective rationally connected variety 
which contains a very twisting surface is likely to satisfy a number
of the conditions listed above. On the other hand, it turns out
condition (\ref{moduli-curves}) above is not enough to guarantee the
existence of a very twisting surface.

\medskip\noindent
Let us get to the point and formulate one of the main
results of this long paper.

\begin{cor}
\label{cor-A}
\marpar{cor-A}
See Corollary \ref{cor-main}.
Let $f : X \to S$ be a morphism of nonsingular projective
varieties over $k$ with $S$ a surface. Assume
\begin{enumerate}
\item there exists a Zariski open subset $U$ of $S$ whose complement
has codimension $2$ such that $X_u$ is irreducible for $u\in U(k)$,
\item there exists an invertible sheaf $\mc{L}$ on $f^{-1}(U)$ which
is $f$-relatively ample, and
\item the geometric generic fibre $X_{\ol{\eta}}$ of $f$ together
with the base change $\mc{L}_{\ol{\eta}}$ is rationally simply connected
by chains of free lines (see above) and has a very twisting surface,
\end{enumerate}
then there exists a rational
section of $f$.
\end{cor}

\noindent
The assumption on the irreducibility of fibres is likely superfluous
and can be removed in most instances, see Remark \ref{rmk-improve}.
The existence of the invertible sheaf $\mc{L}$ \emph{globally} is
what ensures the Brauer class is zero in case the Corollary is applied to a
Brauer-Severi scheme over $S$. The third condition we discussed above.

\medskip\noindent
Let us discuss the method of proof of the Corollary above.
The first remark is that there is a natural inductive structure
on the overall problem of finding sections of families of varieties.
Namely, we may fibre the surface $S$ by curves over a
base curve $B$. Lemma \ref{lem-ratsec} implies that it is
enough to find a rational section of $B$ into the (coarse)
moduli space $\Sec(X/S/B)$ of sections of $X/S/B$. Briefly,
a $k$-point of $\Sec(X/S/B)$ is a pair $(b, \sigma)$,
where $b\in B(k)$, and $\sigma : S_b \to X_b$ is a section
of the restriction of $X \to S$ to $S_b$. Applying \cite{GHS}
to the morphism $\Sec(X/S/B) \to B$ we see that it would suffice
to produce an irreducible subvariety $Z \subset \Sec(X/S/B)$
such that the general fibre of $Z \to B$ is rationally connected.

\medskip\noindent
It turns out that the irreducible components of $\Sec(X/S/B)$
usually do not have rationally connected fibres over $B$.
One obstruction is the Abel map
$$
\alpha_{\mc{L}} : \Sec^e(X/S/B) \longrightarrow \underline{\text{Pic}}^e_{S/B}.
$$
The image of the pair $(b, \sigma)$ is the invertible sheaf
$\sigma^*\mc{L} \in \text{Pic}(S_b)$. It turns out that it is enough
if the fibres of this Abel map are rationally connected (see proof
of Corollary \ref{cor-main}). Thus our corollary follows from
the following which is the main theorem of this article.

\begin{thm}
\label{thm-A}
\marpar{thm-A}
See Theorem \ref{thm-main}.
Let $X \to C$ be a morphism of nonsingular projective varieties
over $k$. Let $\mc{L}$ be an invertible sheaf on $X$. Assume $\mc{L}$
is ample on each fibre of $X \to C$. Assume
the geometric generic fibre of $X \to C$ is rationally simply
connected by chains of free lines (see above), and contains a
very twisting surface. In this case, for $e \gg 0$ there exists
a canonically defined irreducible component $Z_e \subset \Sec^e(X/S/B)$
so that the restriction of $\alpha_{\mc{L}}$ to $Z_e$ has
rationally connected fibres.
\end{thm}

\noindent
The proof of this theorem is long and takes up most of this paper.
We have tried to split up the argument in two stages. In the first
stage, Sections \ref{app-A} -- \ref{sec-famvarlines},
we avoid using the existence of a very twisting surface
in $X_{\ol{\eta}}$, using only the assumption that $X_{\ol{\eta}}$
is rationally simply connected by
chains of free lines. This condition still implies there are many
free lines in fibres of $X \to C$. The main technical result 
in these sections is Proposition \ref{prop-pentogether}. Let us
rephrase it here. Let $\Sigma^e(X/C/k)$ denote the
Kontsevich moduli space of stable sections of $X \to C$ of degree $e$,
see Definition \ref{defn-relstablemap}. A free section, see
Definition \ref{defn-classicalfree}, is a section $s : C \to X$
such that $s^*T_{X/C}$ is a globally generated locally free sheaf
with trivial $H^1$. A \emph{porcupine}, see Definition \ref{defn-ppine},
is roughly speaking a stable map obtained by taking a free section
and glueing on a bunch of free lines (its quills). With this terminology
Proposition \ref{prop-pentogether} roughly states that given any two free
sections $s_1$, $s_2$ there exist extensions of them to porcupines
$p_1$, $p_2$ whose moduli points $[p_i] \in \Sigma^e(X/C/k)$
can be connected by a chain of rational curves whose nodes 
correspond to unobstructed, non-stacky points of $\Sigma^e(X/C/k)$.

\medskip\noindent
This in some sense means some huge Ind-scheme constructed out of
spaces of sections of $X/C$ is rationally connected. But, as we mentioned
above, we do not know a theorem a la \cite{GHS} for families of
Ind-schemes. And it is at this point that the existence of a
very twisting surface comes in. Namely, in the second stage, 
Sections \ref{sec-perfectpens} and \ref{sec-proofs}, we show how
the existence of very twisting surfaces in fibres of
$X \to C$ implies a the moduli point corresponding to a
general porcupine is connected by a chain of
rational curves in $\Sigma^e(X/C/k)$ to the moduli point 
corresponding to a free section. Combining this with the previous
result we get the theorem above.

\medskip\noindent
{\bf Application I.} In an email dated Jul 8, 2005 Phillipe Gille
sketched out how a theorem as above for families of homogeneous
varities (or Borel varieties) over surfaces might lead to a proof
of Serre's conjecture II for function fields of surfaces.
Our method sketched above reduces this
to studying sections of families of Borel varieties over curves.
In the case of curves over finite fields \emph{the same method}
was used by G.\ Harder to prove Serre's conjecture II for function
fields of curves over finite fields, see for example \cite{Harder},
the references therein and \cite{HarderIII}. Of course the actual
details differ substantially.

\medskip\noindent
In Sections \ref{sec-He} and \ref{sec-GPR1C} we show that the theory
above applies to families of varieties with fibres
of the form $Z = G/P$ with $P$
maximal parabolic. The main new results are Theorem \ref{thm-He} and
Proposition \ref{prop-diagram}. We struggled with this material for
a long time, partly because we did not realize that for a general
variety of type $G/P$ with $P$ maximal parabolic, it is \emph{not}
the case that tangent vectors along lines at a fixed point $p$
span the tangent space $T_pZ$. Nonetheless these varieties are 
rationally simply connected by chains of free lines and contain
very twisting surfaces, see Corollary \ref{cor-GPverytwisting} and
Lemmas \ref{lem-GPevRC} and \ref{lem-GPchainR1C}. Therefore
using the Corollary \ref{cor-A} stated above we immediately obtain
the following corollary.

\begin{cor}
\label{cor-B}
\marpar{cor-B}
Let $G$ be a linear algebraic group over $k$ and let
$P \subset G$ be a standard maximal parabolic subgroup.
Let $S$ be a nonsingular projective surface over $k$.
Let $X \to S$ be a smooth projective morphism
all of whose geometric fibres are isomorphic to $G/P$.
Assume there exists an invertible sheaf $\mc{L}$
on $X$ which restricts to an ample generator of
$\text{Pic}(G/P)$ on a fibre. Then $X \to S$ has a rational
section.
\end{cor}

\noindent
It turns out that we can reduce a more general case of a family
of projective homogeneous spaces to the simple case above.
Here is the most general result we obtain of this kind.

\begin{thm}
\label{thm-B}
\marpar{thm-B}
See Theorem \ref{thm-main2}.
Let $k$ be an algebraically closed field of any characteristic.
Let $S$ be a quasi-projective surface over $k$. Let
$X \to S$ be a projective morphism whose geometric generic
fibre is of the form $G/P$ for some linear algebraic group
$G$ and parabolic subgroup $P$. Assume furthermore that
$\text{Pic}(X) \to \text{Pic}(X_{\ol{\eta}})$ is surjective.
Then $X \to S$ has a rational section.
\end{thm}

\noindent
We would like to explain how this leads to a proof
of Serre's conjecture II in the special case of function fields.
Namely, let $G$ be connected, semi-simple, and simply connected
over $k$ algebraically closed of any characteristic.
Suppose $K$ is the function field of a surface $S$ over
$k$. Suppose $\xi \in H^1(K, G)$ is a Galois cohomology
class. After shrinking $S$ we may assume that $\xi$ comes from
an element in $H^1(S, G)$, i.e., is given by a right $G$-torsor
$\mathcal{T}$ over $S$. Consider the family
$$
X = \mathcal{T} \times^{G} G/B \to S
$$
Since $G$ is simply connected any invertible sheaf
$\mc{L}_0$ on $G/B$ has a $G$-linearization and hence
we can descend $\OO_{\mc{T}} \boxtensor \mc{L}_0$ to
$X$. In other words, the assumptions of Theorem \ref{thm-B}
are satisfied and we see that $X$ has a $K$-rational point.
Because the fibres of $\mc{T} \to X$ are $B$-torsors,
this means that $\xi$ comes from an element of $H^1(K, B)$.
But as $B$ is a connected solvable group over $k$ we have
$H^1(K, B) = 0$ for any field $K$. We conclude $\xi$
is trivial. Thus we obtain the following special
case of a conjecture of Serre, cf. \cite[p. 137]{GalCoh}.

\begin{cor}[Serre's Conjecture II over function fields for split groups]
\label{cor-C}
\marpar{cor-C}
Let $k$ be an algebraically closed field of any characteristic
and let $K/k$ be the function field of a surface.  Let $G$ be a
connected, semisimple, simply connected algebraic group over $k$.
Every $G$-torsor over $K$ is trivial.
\end{cor}

\noindent
In particular, suppose $G_K$ is a simple algebraic group over
$K$ of type $E_8$. Note that $G_K$ is simply connected. Let $G$
be a split form of $G_K$. Note that $\text{Aut}(G) \cong \text{Inn}(G)
\cong G$. We see that $G_K$ corresponds to an element of $H^1(K, G)$
and hence by Corollary \ref{cor-C} it is itself split.
Combined with \cite[Theorem 1.2(v)]{CTGP}, which deals with all
other types this completes
the proof of Serre's Conjecture II for function fields.

\begin{thm}[Serre's Conjecture II for function fields]
\label{thm-C}
\marpar{thm-C}
For every connected, simply connected, semisimple algebraic
group $G_K$ over $K$, every $G_K$-torsor over $K$ is trivial.
\end{thm}

\noindent
We would like to point out here that lots of work has been done by
lots of authors on Serre's conjecture II. Our approach is tailored
to function fields of surfaces. Perhaps some of the motivation
of Serre's Conjecture II comes from studying the relationship
between cohomological
dimensions of fields, the arithmetic of fields, and linear algebraic groups. 
We have found the paper \cite{CTGP} and its references a rich source of
information regarding this material.

\medskip\noindent
{\bf Application II.} We would like to point out that
Corollary \ref{cor-A} and Theorem \ref{thm-A} can be used
to reprove Tsen's theorem for families of hypersurfaces
over surfaces. Whereas the proof of Tsen's theorem can fit on
a napkin, this paper has 60+ pages. However, the aim of such an
investigation is to stimulate research into different generalizations
of Tsen's theorem. For example, one possible avenue for research is
to study families of low degree hypersurfaces in homogeneous
spaces. In any case, we briefly explain in Section
\ref{sec-hypersurface} below how to reprove Tsen's theorem
using the main results of this paper.

%%%%%%%%%%%%%%%%%%%%%%%%%%%%%%%%%%%%%%%%%%%%%%%%%%%%%%%%%%%%%%%
%%
%% Appendix A : Stacks of curves and maps from curves to a target
%% 
%%%%%%%%%%%%%%%%%%%%%%%%%%%%%%%%%%%%%%%%%%%%%%%%%%%%%%%%%%%%%%%

\section{Stacks of curves and maps from curves to a target}
\label{app-A}
\marpar{app-A}

\noindent
The correct notion when doing moduli of algebraic varieties
is to use families of varieties where the total space is 
an algebraic space -- not necessarily a scheme. Even in the case
of families of curves it can happen that the total space is not
locally a scheme over the base. This leads to the following
definitions.

\begin{defn}
\label{defn-curves}
\marpar{defn-curves}
Let $S$ be a scheme.

\mni
\textbf{1.}
A \emph{flat family of proper curves over $S$} is
a morphism of algebraic spaces $\pi:C\rightarrow S$
which is proper, locally finitely presented,
and flat of relative dimension $1$.\footnote{All fibres pure dimension 1.}

\mni
\textbf{2.}
A \emph{flat family of polarized, proper curves over $S$} is a
pair $(\pi:C\rightarrow S,\mc{L})$ consisting of a flat family
of proper curves over $S$ and a $\pi$-ample invertible sheaf
$\mc{L}$ on $C$.   

\mni
\textbf{3.}
For every morphism of schemes $u:T \rightarrow S$
and $\pi : C \rightarrow S$ as above the
\emph{pullback family} is the projection
$$
u^*\pi = \text{pr}_T: T\times_{u,S,\pi} C \rightarrow T.
$$
If the family is polarized by $\mc{L}$ then  
the pullback family is polarized by the $u^*\pi$-ample
invertible sheaf $(u,\text{Id}_C)^*\mc{L}$.

\mni
\textbf{4.}
For two flat families of proper curves 
$\pi:C\rightarrow S$ and $\pi':C' \rightarrow S'$,
a \emph{morphism} $f$ from the first family to
the second is given by a cartesian diagram
$$
\xymatrix{
C \ar[r]^f \ar[d] & C' \ar[d] \\
S \ar[r] & S'.
}
$$
If $\pi:C\rightarrow S$ is
polarized by $\mc{L}$ and $\pi':C' \rightarrow S'$
is polarized by $\mc{L}'$ then a \emph{morphism}
between the polarized families is a pair $(f, \phi)$
with $f$ as above and $\phi:\mc{L} \rightarrow
f^*\mc{L}'$ is an isomorphism.
\end{defn}

\begin{defn}
\label{defn-stackcurves}
\marpar{defn-stackcurves}
The \emph{stack of proper curves} is the fibred category
$$
\textit{Curves}
\longrightarrow
\textit{Schemes}
$$
whose objects are families of flat proper curves
$\pi:C\rightarrow S$ and morphisms are as above.
The functor $\textit{Curves} \to \textit{Schemes}$
is the forgetful functor.
Similarly the \emph{stack of proper, polarized curves} is the
fibred category
$$
\textit{Curves}_{\text{pol}}
\longrightarrow
\textit{Schemes}
$$
whose objects are flat families of proper, polarized curves
$(\pi:C\rightarrow S,\mc{L})$ as in the definition above.
\end{defn}

\noindent
Denote by $(\text{Aff})$ the category of affine schemes.
A technical remark is that in \cite{LM-B} stacks are defined
as fibred categories over $(\text{Aff})$. So in the following
propositions we state our results in a manner that is 
compatible with their notation. In particular when we speak
of $\textit{Curves}$ over $(\text{Aff})$ we mean the
restriction of the fibred category $\textit{Curves}$
to $(\text{Aff}) \subset \textit{Schemes}$.

\mni
There is functor of fibred categories
$$
F:\textit{Curves}_{\text{pol}} \rightarrow \textit{Curves}
$$ 
forgetting the invertible sheaves.

\begin{prop}
\label{prop-algebraic}
\marpar{prop-algebraic}
The categories $\textit{Curves}_{\text{pol}}$ and $\textit{Curves}$ are
limit preserving algebraic stacks over $(\text{Aff})$ (with the fppf
topology) with
finitely presented, separated diagonals.  Moreover $F$
is a smooth, surjective morphism of algebraic stacks. 
\end{prop}

\begin{proof}
This is folklore; here is one proof.
The category $\textit{Curves}_{\text{pol}}$ is a limit preserving
algebraic stack with finitely presented, separated diagonal by
\cite[Proposition 4.2]{Smoduli}.  Moreover, by \cite[Propositions 3.2
and 3.3]{Smoduli}, $\textit{Curves}$ is a limit preserving stack with
representable, separated, locally finitely presented diagonal.  To
prove that $\textit{Curves}$ is an algebraic stack, it only remains to
find a smooth cover of $\textit{Curves}$, i.e., a
smooth, essentially surjective $1$-morphism from an algebraic space
to $\textit{Curves}$.  Let $u : X \to \textit{Curves}_{\text{pol}}$
be such a smooth cover for the stack of polarized curves.
We claim that $F \circ u : X \to \textit{Curves}$ is a 
smooth cover.

\mni
Let $(S,\pi:C\rightarrow S)$ be an object of $\textit{Curves}$. The
$2$-fibreed product of $F$ with the associated $1$-morphism
$S\rightarrow \textit{Curves}$ is the stack $\mc{A}$ of $\pi$-ample
invertible sheaves on $C$. Consider the diagram
$$
\xymatrix{
\mc{A}\times_{\textit{Curves}_{\text{pol}}\ar[d]}X \ar[r]\ar[d] & X \ar[d] \\
\mc{A} \ar[r] \ar[d] & \textit{Curves}_{\text{pol}}\ar[d] \\
S \ar[r] & \textit{Curves}.
}
$$
Once we have shown $\mc{A}$ is an algebraic stack and
$\mc{A} \to S$ is smooth surjective, then the claim above
follows, as well as the smoothness and surjectivity of $F$
in the proposition.

\mni
Before we continue we remark that $\mc{A}$ is an open
substack of the Picard stack of $C \to S$ by
\cite[Th\'eor\`eme 4.7.1]{EGA3}. Hence, by
\cite[Proposition 4.1]{Smoduli} (which is a
variant of \cite[The\'eor\`eme 4.6.2.1]{LM-B})
the stack $\mc{A}$ is a limit
preserving algebraic stack with quasi-compact, separated diagonal.
In particular, for every smooth cover $A\rightarrow \mc{A}$, the induced
morphism $A\rightarrow S$ is locally finitely presented.  Thus to
prove $A\rightarrow S$ is smooth, and thus that $\mc{A}$ is smooth
over $S$, it suffices to prove that $A\rightarrow S$ is formally
smooth.   Let $s$ be a point of $S$ and denote by $C_s$ the fibre
$\pi^{-1}(s)$.  
By \cite[Proposition 3.1.5]{Illusie1}, the obstructions to
infinitesimal extensions of invertible sheaves on $C_s$
live in $\text{Ext}^2_{\OO_{C_s}}(\OO_{C_s},\OO_{C_s}) =
H^2(C_s,\OO_{C_s})$.  Since $C_s$ is a Noetherian, $1$-dimensional
scheme, $H^2(C_s,\OO_{C_s})$ is zero.  Thus there are no obstructions
to infinitesimal extensions of $\pi$-ample invertible sheaves on
$C_s$. Thus $\mc{A} \rightarrow S$ is smooth as desired.

\mni
Next, we prove that $\mc{A} \rightarrow S$ is
surjective.  Thus, let $s$ be a point of
$S$.  It suffices to prove there is an ample invertible sheaf on $C_s$.
This is essentially \cite[Exercise III.5.8]{H}.  Therefore
$\textit{Curves}$ is an algebraic stack.

\mni
It only remains to prove that the diagonal of $\textit{Curves}$
is quasi-compact.  Let $(\pi_i:C_i\rightarrow S)$, $i=1,2$ be 
two families of proper flat curves over an affine scheme $S$.
We have to show that the algebraic space
$$
\text{Isom}_S(C_1, C_2)
$$
is quasi-compact. Using that $F$ is smooth and surjective, after replacing
$S$ by a quasi-compact, smooth cover, we may assume there
exist a $\pi_i$-ample invertible sheaves $\mc{L}_i$ on $C_i$.
We may assume the Hilbert polynomial of the fibres
of $C_i \to S$ with respect to $\mc{L}_i$ are constant, say $P_i$.
For every $S$-scheme $T$ and every $T$-isomorphism
$\phi :C_{1,T} \rightarrow C_{2,T}$, the graph
$\Gamma_\phi : C_{1,T} \rightarrow (C_1\times_S C_2)_T$
is a closed immersion.  In the usual way this identifies
$\text{Isom}_S(C_1,C_2)$ as a locally closed subscheme of
the Hilbert scheme of $C_1\times_SC_2$ over $S$.
We will show that whenever $T = \text{Spec}(k)$ is the spectrum of a
field, the Hilbert polynomial of $\Gamma_\phi$ with respect to the
$S$-ample invertible sheaf $\text{pr}_1^*\mc{L}_1\otimes
\text{pr}_2^*\mc{L}_2$
is equal to $P_1 - P_2 - P_1(0)$.
Quasi-compactness follows from the projectivity of the
Hilbert scheme parametrizing closed subschemes with given
Hilbert polynomial. By Snapper's
theorem (see \cite{Numerical}) the values of the
Hilbert polynomial
$\chi(C_{1,T}, \mc{L}_{1,T}^{\otimes n}
\otimes \phi^* \mc{L}_{2,T}^{\otimes n})
$
equal
$
\chi(C_{1,T}, \mc{L}_{1,T}^{\otimes n})) +
\chi(C_{1,T}, \phi^*\mc{L}_{2,T}^{\otimes n}) -
\chi(C_{1,T},\OO)$.
And of course $\chi(C_{1,T}, \phi^*\mc{L}_{2,T}^{\otimes n})
= \chi(C_{2,T}, \mc{L}_{2,T}^{\otimes n})$.
The claim follows.
\end{proof}

\noindent
At this point we need some definitions related to moduli of
maps from curves into varieties. But since we are going to study
the rational connectivity of moduli stacks of rational curves
on varieties, we need to deal with moduli of morphisms from
curves to stacks.

\begin{defn}
\label{defn-curvemaps}
\marpar{defn-curvemaps}
Let $\mc{X}$ be an algebraic stack, let $S$ be a scheme 
and let $\mc{X} \to S$ be a morphism.
A \emph{family of maps of proper curves to $\mc{X}$ over $S$}
is a triple $(T \to S, \pi : C \rightarrow T, \zeta : C \rightarrow \mc{X})$
consisting of a flat family of proper curves over the $S$-scheme
$T$ and a $1$-morphism $\zeta$ over $S$.
\end{defn}

\noindent
We leave it to the reader to define morphisms between
families of maps of proper curves to $\mc{X}$ over $S$.

\begin{defn}
\label{defn-stackcurvemaps}
\marpar{defn-stackcurvemaps}
The \emph{stack of maps of proper curves to $\mc{X}$}
is the fibred category
$$
\textit{CurveMaps}(\mc{X}/S)
\longrightarrow
\textit{Schemes}/S
$$
whose objects are families of maps of proper curves to
$\mc{X}$ over $S$, and morphisms as above.
\end{defn}

\noindent
We again have the technical remark that according to the
conventions in \cite{LM-B} we should really be working with
the restriction of $\textit{CurveMaps}(\mc{X}/S)$ to the
category $(\text{Aff}_S)$ of affine schemes over $S$.
Denote by $\textit{Curves}_S$ the restriction of
$\textit{Curves}$ over the category $\textit{Schemes}/S$
of $S$-schemes. There is an obvious functor
$$
G :
\textit{CurveMaps}(\mc{X}/S)
\longrightarrow
\textit{Curves}_S
$$
forgetting the $1$-morphisms to $\mc{X}$.

\begin{prop}
\label{prop-Lieblich}
\marpar{prop-Lieblich}
See \cite[Proposition 2.11]{Lieblich}. Assume $S$ excellent.
Suppose that $\mc{X} = [Z/G]$ where $Z$ is an algebraic space,
$Z \to S$ is separated and of finite presentation, and where
$G$ is an $S$-flat linear algebraic group scheme.
The stack $\textit{CurveMaps}(\mc{X}/S)$ is a limit preserving algebraic
stack over $(\text{Aff}_S)$ with locally finitely presented, separated
diagonal.  And the $1$-morphism $G$ is representable by limit
preserving algebraic stacks.
\end{prop}

\begin{proof}
The second assertion is precisely proved in \cite[Proposition
2.11]{Lieblich} and implies the first assertion because of
Proposition ~\ref{prop-algebraic}.
\end{proof}

\begin{defn}
\label{defn-LCI}
\marpar{defn-LCI}
Let $S$ be a scheme.  A flat family of proper curves over $S$,
$\pi:C\rightarrow S$, is \emph{LCI} if it is a local complete
intersection morphism in the sense of
\cite[D\'efinition 19.3.6]{EGA4}.
\end{defn}

\noindent
It is true that this is equivalent to \cite[D\'efinition VIII.1.1]{SGA6}.
See \cite[Proposition VIII.1.4]{SGA6}.

\begin{prop}
\label{prop-LCI}
\marpar{prop-LCI}
The subcategory $\textit{Curves}_{\text{LCI}}$ of $\textit{Curves}$ of
flat families which are LCI is an open substack of $\textit{Curves}$.
\end{prop}

\begin{proof}
Let $\pi : C \to S$ be a flat family of proper curves,
and assume for the moment that $C$ is a scheme.
According to \cite[Corollaire 19.3.8]{EGA4} there exists
an open subscheme $U \subset S$ such that $C_s$ is
LCI if and only if $s \in U$. By definition \cite[D\'efinition 19.3.6]{EGA4}
the induced morphism $C_U \to U$ is LCI. It follows from
\cite[Proposition 19.3.9]{EGA4} that given a morphism
of schemes $g : T \to S$ then inverse image $g^{-1}(U) \subset
T$ has the same property for the family over $T$.

\mni
In the general case (where we do not assume that $C$ is a scheme)
we use that there is a smooth cover $S' \to S$ such that
the pullback family over $S'$ has a total space which is a
scheme (for example by using \ref{prop-algebraic}). Hence we
obtain an open $U' \subset S'$ characterized as above.
Clearly $U' \times_S S' = S'\times_S U'$ and we conclude
that $U'$ is the inverse image of an open $U$ of $S$.
Again by definition the restriction $C_U \to U$ is LCI,
and all fibres of $C \to S$ at points of $S \setminus U$
are not LCI. Clearly this shows that any base change
by $T \to S$ is LCI if and only if $T\to S$ factors through $U$.
\end{proof}

\begin{defn}
\label{defn-stackcurvemapsLCI}
\marpar{defn-stackcurvemapsLCI}
Let $\mc{X}$ be an algebraic stack, let $S$ be a scheme and
let $\mc{X} \to S$ be a morphism.
The \emph{stack of maps of proper, LCI curves to $\mc{X}$} is the open
substack $\textit{CurveMaps}_{\text{LCI}}(\mc{X}/S)$ of
$\textit{CurveMaps}(\mc{X}/S)$ obtained as the $2$-fibreed product of
$\textit{CurveMaps}(\mc{X}/S) \rightarrow \textit{Curves}$ and
$\textit{Curves}_{\text{LCI}} \rightarrow \textit{Curves}$.
\end{defn}

\noindent
In order to define the determinant pushforward below we need
the following lemma.

\begin{lem}
\label{lem-ci}
\marpar{lem-ci}
Suppose we are given a diagram of morphisms of algebraic spaces
$$
\xymatrix{
C \ar[r] \ar[d] & Z \ar[d] \\
T \ar[r] & S
}
$$
where $C \to T$ is an LCI, flat proper family of curves,
and $Z \to S$ is a smooth morphism. Then the induced morphism
$C \rightarrow T\times_S Z$ is a local complete intersection
morphism in the sense of \cite[D\'efinition VIII.1.1]{SGA6}.
In particular, it is a \emph{perfect morphism} in the sense
of \cite[D\'efinition III.4.1]{SGA6}.
\end{lem}

\begin{proof}
Note that $T\times_S Z$ is smooth over $T$, and hence it
suffices to prove the lemma in the case that $S=T$. 
By \cite[Proposition VIII.1.6]{SGA6} it suffices
to prove the lemma after a faithfully flat base change.
Hence it is also sufficent to prove the proposition in
the case that $C \to S$ is projective, by \ref{prop-algebraic}.
Choose a closed immersion $C \to \mathbf{P}_S^N$ over $S$. This is
regular by \cite[Proposition VIII.1.2]{SGA6}. 
According to \cite[Corollaire VIII.1.3]{SGA6}
this implies that $C \to \mathbf{P}_S^N \times_S Z$
is regular, which in turn by definition implies that
$C \to Z$ is a local complete intersection morphism.
\end{proof}

\noindent
We are going to use the stack $\textit{CurveMaps}(\mc{X}/S)$
in the case where $\mc{X} = Z\times B\mathbf{GL}_n$ for some
algebraic space $Z$ separated and of finite presentation over $S$.
Note that by Proposition \ref{prop-Lieblich} this is a limit
preserving algebraic stack. An object of this stack is given
by a datum $(T\to S, \pi : C \to T, \zeta : C \to Z, \mc{E})$
where $(T\to S, \pi : C \to T, \zeta : C \to Z)$ is a family
of maps of proper curves to $Z$ over $S$ and
where $\mc{E}$ is a locally free sheaf of rank $n$ over $C$.
It is straightforward to spell out what the morphisms are 
in this stack (left to the reader). The object
$(T\to S, \pi : C \to T, \zeta : C \to Z, \mc{E})$
belongs to $\textit{CurveMaps}_{\text{LCI}}(\mc{X}/S)$
if and only if $\pi$ is LCI.

\mni
Assuming that $S$ is excellent, $Z$ is a scheme,
$Z \to S$ is quasi-compact and smooth
and $\mc{X} = Z\times B\mathbf{GL}_n$
we define a functor
$$
\text{det}(R(-)_*):\textit{CurveMaps}_{\text{LCI}}(\mc{X}/S)
\rightarrow \text{Hom}_S(Z,B\mathbb{G}_{m}) = \textit{Pic}_{Z/S}
$$
to the Picard stack of $Z$ over $S$. First we note that, by our
discussion above both sides are limit preserving algebraic stacks
over $S$. Thus we need only define the functor on fibre categories
over schemes of finite type over $S$. Consider an object 
$
(T\to S, \pi : C \to T, \zeta : C \to Z, \mc{E})
$
of the left hand side with $T/S$ finite type.
The morphism $\zeta : C\rightarrow T\times_S Z$
is a perfect morphism by Lemma ~\ref{lem-ci}.
Since $\zeta$ is perfect $R\zeta_* \mc{E}$ is a perfect complex of
bounded amplitude on $T\times_S Z$ by \cite[Corollaire III.4.8.1]{SGA6}.
(We leave it to the reader to check that we've put in enough finiteness
assumptions so the Corollary applies.) The ``det'' construction of
\cite{detdiv} associates an invertible sheaf $\text{det}(R\zeta_* \mc{E})$
to the perfect complex $R\zeta_* \mc{E}$ on the scheme
$T\times_SZ$. Also, by \cite[Definition 4(iii)]{detdiv} and
\cite[Proposition IV.3.1.0]{SGA6}, formation of
$\text{det}(R\zeta_*\mc{E})$ is compatible with morphisms of objects
of $\textit{CurveMaps}_{\text{LCI}}(\mc{X}/S)$.

\begin{defn}
\label{defn-detzeta}
\marpar{defn-detzeta}
Given $S$, $Z \to S$ and $n$ as above.
The \emph{determinant pushforward $1$-morphism}
is the functor
$$
\text{det}(R(-)_*)
:
\textit{CurveMaps}_{\text{LCI}}( Z\times B\mathbf{GL}_n / S)
\rightarrow
\textit{Pic}_{Z/S}
$$
defined above. If the Picard stack $\textit{Pic}_{Z/S}$
has a course moduli space $\underline{\text{Pic}}_{Z/S}$ then the
composite morphism
$$
\textit{CurveMaps}_{\text{LCI}}(Z\times B\mathbf{GL}_n /S)
\rightarrow
\underline{\text{Pic}}_{Z/S}
$$
will also be called the \emph{determinant pushforward $1$-morphism}.  
\end{defn}

\noindent
The following is an important special case for our paper.
Namely, suppose that $S = \text{Spec}(\kappa)$ is the spectrum
of a field $\kappa$, and suppose that $Z = C$ is a smooth,
projective, geometrically connected $\kappa$-curve of genus $g$.
In this case we can ``compute'' the value of the determinant
$1$-morphism on some special points. Assume that we have a
$\text{Spec}(\kappa)$-valued point of
$\textit{CurveMaps}_{\text{LCI}}( C\times B\mathbf{GL}_n / S)$
given by a datum
$(\text{Spec}(\kappa)\to \text{Spec}(\kappa),
\pi : C' \to \text{Spec}(\kappa), \zeta : C' \to C, \mc{E})$.
Assume furthermore that $C'$ is proper, at worst nodal,
geometrically connected of genus $g$. Finally, assume that
$\zeta$ is an isomorphism over a dense open of $C'$. 
For every such map, there exists a unique section 
$s : C \to C'$ whose image is the unique irreducible
component of $C'$ mapping dominantly to $C$.
The scheme $\overline{C'-s(C)}$ is a disjoint union 
of trees of rational curves $C_1, \ldots, C_\delta$.
Each $C_i$ meets $s(C)$ in a single point $t_i$.

\begin{lem}
\label{lem-newAbel}
\marpar{lem-newAbel}
Assumptions and notation as above. There is an isomorphism
$$
\text{det}(R\zeta_*(\mc{E}))
=
\text{det}(s^*\mc{E})(d_1\cdot t_1 + \dots + d_\delta\cdot t_\delta)
$$
where $d_i$ is the degree of the basechange of $\mc{E}$ to the
connected nodal curve $C_i$ over the field $\kappa(t_i)$.  
\end{lem}

\begin{proof}
For a comb $C' = s(C)\cup C_1 \cup \dots \cup C_\delta$ as above,
there is a surjection of $\OO_{C}$-modules $\mc{E} \to \mc{E}|_{s(C)}$
whose kernel we will write as $\oplus \mc{E}|_{C_i}(-t_i)$.
We will use later on that $\chi(C_i, \mc{E}|_{C_i}(-t_i))
= d_i$ by Riemann-Roch on $C_i$ over $\kappa(t_i)$.
Pushing this forward by $\zeta$ gives a triangle of perfect complexes:
$$
\xymatrix{
\ar[r]
&
R\zeta_*(\oplus \mc{E}|_{C_i}(-t_i))
\ar[r]
&
R\zeta_*(\mc{E})
\ar[r]^-{\lambda}
&
R\zeta_*(\mc{E}|_{s(C)})
\ar[r]
&
}
$$
Note that $R\zeta_*(\mc{E}|_{s(C)}) = s^*(\mc{E})[0]$,
and that the term $R\zeta_*(\oplus \mc{E}|_{C_i}(-t_i))$
is supported in the points $t_i$.
Hence $\lambda$ is ``good'' and we may
apply the ``Div'' construction of \cite{detdiv}. 
It folows that
$\text{det}(R\zeta_*(\mc{E}))
\cong
\text{det}(s^*\mc{E})(-\text{Div}(\lambda))$.
Using \cite[Theorem 3(iii) and (vi)]{detdiv}, it follows that
$\text{Div}(\lambda) = - \sum d_i \cdot t_i$ (minus sign because
the complex $\mc{H}^{\cdot}$ of locus citatus is our complex
$\zeta_*(\oplus \mc{E}|_{C_i}(-t_i))$ shifted by $1$).
\end{proof}

\noindent
We end this section with a simple semi-continuity lemma.

\begin{lem}
\label{lem-semicty}
\marpar{lem-semicty}
Let $S$ be an affine scheme.
Let $\pi : C \to S$ be a flat family of proper curves over $S$.
Let $\mc{F}$ be a quasi-coherent sheaf on $C$
which is locally finitely presented and flat over $S$.
For every integer $i\geq 0$ there exists an open subscheme $U_i$ of $S$
such that for every $S$-scheme $T$, the derived pushforward to $T$ of
the pullback of $\mc{F}$ to $T\times_S C$ is concentrated in degrees
$\leq i$ if and only if $T$ factors through $U_i$.  Moreover, after
basechange from $S$ to $U_i$, the formation of $R^i\pi_*(\mc{F})$ is
compatible with arbitrary basechange.

\mni
This also holds when $S$ is an arbitrary scheme, algebraic space,
etc., instead of an affine scheme.
\end{lem}

\begin{proof}
This is probably true in some vast generality for morphisms of algebraic
stacks, but we do not know a reference. In our case we may deduce
it from the schemes case as follows.
By Proposition ~\ref{prop-algebraic}, there exists a
faithfully flat morphism of affine schemes
$S'\rightarrow S$ such that $S'\times_S C$ is projective over $S'$.
By \cite[Proposition 13.1.9]{LM-B}, the statement for the original
family over $S$ follows from the statement over $S'$.  Also by limit
arguments, it suffices to consider the case when $S'$ is of finite
type.  Now the result follows from \cite[Section 7]{EGA3} or
\cite[Section 5]{AVar}. The general case follows from the affine
case by \cite[Proposition 13.1.9]{LM-B}.
\end{proof}

%%%%%%%%%%%%%%%%%%%%%%%%%%%%%%%%%%%%%%%%%%%%%%%%%%%%%%%%%%%%%%%
%%
%% Free sections
%% 
%%%%%%%%%%%%%%%%%%%%%%%%%%%%%%%%%%%%%%%%%%%%%%%%%%%%%%%%%%%%%%%

\section{Free sections}
\label{sec-newfree}
\marpar{sec-newfree}

\noindent
This section is a brief summary of recent work on Hom
stacks together with results about free curves in this context.  

\mni
Let $S$ be an algebraic space. Let $\mc{C}\rightarrow S$ be a
proper, flat, finitely presented algebraic stack over $S$. Let
$f:\mc{X}\rightarrow \mc{C}$ be a $1$-morphism of algebraic
$S$-stacks.
%
% Jason: The way I define the category of stacks, there is
% defined composition of 1-morphisms, and hence you do not
% need the pullback funtors below.
%
%We may and we do assume each of our stacks comes together with a
%\emph{pullback functor}, i.e., a \emph{clivage normalis\'ee}
%as in \cite[D\'efinition VI.7.1]{SGA1}.  

\begin{defn}
\label{defn-preFGA}
\marpar{defn-preFGA}
With notation as above.
Let $T$ be an $S$-scheme, or an algebraic space over $S$.

\mni
\textbf{1.}
A \emph{family of sections of $f$ over $T$}
is a pair $(\tau,\theta)$ consisting of a $1$-morphism of $S$-stacks
$\tau:T \times_S \mc{C} \rightarrow \mc{X}$
together with a $2$-morphism
$\theta : f\circ \tau \Rightarrow \text{pr}_{\mc{C}}$, giving a
$2$-commutative diagram
\xyoption{2cell}
\UseTwocells
$$
\xymatrixcolsep{5pc}
\xymatrix{
&
\mc{X} \ar[d]^f \\
T \times_S \mc{C}
\rtwocell^{f \circ \tau}_{\text{pr}_{\mc{C}}}{\theta}
\ar[ur]^\tau \ar[d]
&
\mc{C}\ar[d] \\
T \ar[r]
&
S 
}
$$

\mni
\textbf{2.}
For two families of sections of $f$, say
$(T', \tau', \theta')$ and $(T, \tau, \theta)$,
a \emph{morphism} $(u, \eta)$ from the first family
to the second family is given by
a morphism $u : T' \to T$ and a $2$-morphism
$\eta: \tau \circ (u \times \text{id}_{\mc{C}}) \Rightarrow \tau'$
such that
$$
\xymatrix{
f \circ \tau' \ar[r]^-{\eta} \ar[d]^{\theta'}
&
f \circ \tau \circ (u \times \text{id}_{\mc{C}}) \ar[d]^{\theta} \\
\text{pr}^{T'\times_S\mc{C}}_{\mc{C}} \ar[r]^-{=}
&
\text{pr}^{T\times_S\mc{C}}_{\mc{C}} \circ (u \times \text{id}_{\mc{C}})
}
$$
commutes.
\end{defn}

\begin{defn}
\label{defn-sec}
\marpar{defn-sec}
The \emph{stack of sections of $f$} is the fibred category
$$
\Sec(\mc{X}/\mc{C}/S) \longrightarrow \textit{Schemes}/S
$$
whose objects are family of sections of $f$ and whose
morphisms are morphisms of families of sections of $F$.
\end{defn}

\noindent
We recall some results from the literature.

\begin{thm}
\label{thm-algebraic}
\marpar{thm-algebraic}
Let $\mc{C} \to S$, $f : \mc{X} \to \mc{C}$ be as above.

\mni
\textbf{1.}
See \cite[Part IV.4.c, p. 221-19]{FGA}. Assume that $C=\mc{C}$ is a
projective scheme over $S$ and that $X=\mc{X}$ is a
quasi-projective scheme over $C$ either globally over $S$,
resp. fppf locally over $S$.  Then $\Sec(X/C/S)$ is an algebraic space
which is locally finitely presented and separated over $S$.  Moreover,
globally over $S$, resp. fppf locally over $S$, the connected
components of $\Sec(X/C/S)$ are quasi-projective over $S$.  

\mni
\textbf{2.} See \cite[Theorem 1.1]{OHom}. Assume that $\mc{C}$ and $\mc{X}$
are separated over $S$ and have finite diagonal. Also assume that fppf
locally over $S$ there is a finite, finitely presented cover of $\mc{C}$ by
an algebraic space.  Then $\Sec(\mc{X}/\mc{C}/S)$ is a limit preserving
algebraic stack over $S$ with separated, quasi-compact diagonal.  

\mni
\textbf{3.} See \cite[Proposition 2.11]{Lieblich}. Assume $S$ excellent
and $\mc{C} = C$ is a proper flat family of curves.
Assume that, fppf locally over $S$, we can write $\mc{X} = [Z/G]$
for some algebraic space $Z$ separated and of finite type over $S$
and some linear group scheme $G$ flat over $S$. Then
$\Sec(\mc{X}/C/S)$ is a limit preserving algebraic stack over $S$.
\end{thm}

\begin{proof}
Besides the references, here are some comments. For the first case
we may realize the space of sections as an open subscheme of Hilbert
scheme of $C\times_SX$ over $S$. In the second and third case we
think of $\Sec(\mc{X}/\mc{C}/S)$ as a substack of the 
Hom-stack $\text{Hom}_S(\mc{C}, \mc{X})$. The references
guarantee the existence of the Hom-stack as an algebraic stack.
Next we consider the $2$-cartesian diagram
$$
\xymatrix{
\Sec(\mc{X}/\mc{C}/S) \ar[r] \ar[d]
&
\text{Hom}_S(\mc{C}, \mc{X}) \ar[d]
\\
S \ar[r]
&
\text{Hom}_S(\mc{C}, \mc{C})
}
$$
In case 2 resp.\ 3 the lower horizontal arrow is representable
by the result from \cite{OHom} resp.\ because of our earlier
results on flat families of proper curves.
\end{proof}

\noindent
Of course we will most often use the space of sections for
a family of varieties over a curve. We formulate a set of hypothesis
that is convenient for the developments later on.

\begin{hyp}
\label{hyp-curves}
\marpar{hyp-curves}
Let $\kappa$ be a field.
Let $C$ be a projective, smooth, geometrically connected curve over $\kappa$.
Let $f:X\rightarrow C$ be a quasi-projective morphism.
Let $\mc{L}$ be an $f$-ample invertible sheaf on $X$.
Denote the smooth locus of $X$ by $X_{\text{smooth}}$.
Denote the open subset of $X_{\text{smooth}}$ on which $f$
is smooth by $X_{f,\text{smooth}}$.
\end{hyp}

\noindent
According to the theorem above the space of sections
$\Sec(X/C/\kappa)$ is a union of quasi-projective
schemes over $\kappa$ in this case. The locus
$\Sec(X_{f,\text{smooth}}/C/\kappa)$ is an open
subscheme of $\Sec(X/C/\kappa)$.

\begin{defn}
\label{defn-FGA}
\marpar{defn-FGA}
In the situation \ref{hyp-curves}.
For every integer $e$, the
\emph{scheme of degree $e$ sections of $f$} is the
open and closed subscheme $\Sec^e(X/C/\kappa)$ of
$\Sec(X/C/\kappa)$ parametrizing sections $\tau$
of $f$ which pullback $\mc{L}$ to a degree $e$
invertible sheaf on $C$. The
\emph{universal degree $e$ section of $f$} is denoted
$$
\sigma :
\Sec^e(X/C/\kappa) \times_\kappa C
\longrightarrow
X.
$$
The \emph{Abel map} is the $\kappa$-morphism
$$
\alpha : 
\Sec^e(X/C/\kappa)
\longrightarrow
\underline{\text{Pic}}^e_{C/\kappa}
$$
associated to the invertible sheaf $\sigma^*\mc{L}$ by the universal
property of the Picard scheme $\underline{\text{Pic}}^e_{C/\kappa}$.  
\end{defn}

\noindent
There are several variations of the notion of ``free curve'' as
described in \cite[Definition II.3.1]{K}.
Of the following two notions, the weaker
notion arises more often geometrically
and implies the stronger notion in a precise sense.  

\begin{defn}
\label{defn-newfree}
\marpar{defn-newfree}
In Situation \ref{hyp-curves}.
Let $T$ be a $\kappa$-scheme. Let
$\tau:T\times_\kappa C \rightarrow X_{f,\text{smooth}}$
be a family of degree $e$ sections of $f$ with
image in $X_{f,\text{smooth}}$. 

\mni
\textbf{1.}
Let $D$ be an effective Cartier divisor in $C$.
The family $\tau$ is \emph{weakly $D$-free} if
the $\OO_T$-module homomorphism
$$
(\text{pr}_T)_*(\tau^* T_f)
\longrightarrow
(\text{pr}_T)_*(\tau^* T_f \otimes_{\OO_{T\times_\kappa C}} \text{pr}_C^*\OO_D)
$$
is surjective.  Here $T_f = \textit{Hom}(\Omega^1_{X/C}, \OO_X)$.

\mni
\textbf{2.}
Let $D$ be an effective Cartier divisor in $C$.
The family $\tau$ is \emph{$D$-free} if the 
sheaf of $\OO_T$-modules
$$
R^1(\text{pr}_T)_*(
\tau^* T_f \otimes_{\OO_{T\times_\kappa C}}\text{pr}_C^* \OO_C(-D)
)
$$
is zero.

\mni
\textbf{3.}
For an integer $m\geq 0$, the family is
\emph{weakly $m$-free} if the
basechange to $\text{Spec}(\overline{\kappa})$
is weakly $D$-free for every
effective, degree $m$, Cartier divisor $D$ in
$C\times_\kappa \text{Spec}(\overline{\kappa})$.  

\mni
\textbf{4.}
For an integer $m\geq 0$, the family is
\emph{$m$-free} if the
basechange to $\text{Spec}(\overline{\kappa})$ 
is $D$-free for every effective, 
degree $m$, Cartier divisor $D$ in
$C\times_\kappa \text{Spec}(\overline{\kappa})$.
\end{defn}

\noindent
Whenever we talk about $D$-free, weakly $D$-free,
weakly $m$-free or $m$-free sections we tacitly assume
that the image of the section is contained in $X_{f,\text{smooth}}$.

\begin{defn}
\label{defn-classicalfree}
\marpar{defn-classicalfree}
In Situation \ref{hyp-curves} a section $s : C \to X$
of $f$ is called \emph{unobstructed} if it is $0$-free.
It is called \emph{free} if it is $1$-free.
\end{defn}

\noindent
It it true and easy to prove that these correspond to
the usual notions. Namely, if a section is unobstructed
then its formal deformation space is the spectrum of
a power series ring. Also, a section $s$ is free if and only
if it is unobstructed and $s^*T_{X/C}$ is globally generated,
which means that its deformations sweep out an open in $X$.

\begin{lem}
\label{lem-newfreeopen}
\marpar{lem-newfreeopen}
In Situation \ref{hyp-curves}.

\mni
\textbf{1.}
Let $D$ be an effective Cartier divisor on $C$.
There is an open subscheme $U$ of $\Sec^e(X/C/\kappa)$
parametrizing exactly those families whose image
lies in $X_{f,\text{smooth}}$ and which are $D$-free.

\mni
\textbf{2.}
Let $m \geq 0$ be an integer.
There is an open subscheme $V$ of $\Sec^e(X/C/\kappa)$
parametrizing exactly those families whose image
lies in $X_{f,\text{smooth}}$ and which are $m$-free.

\mni
\textbf{2.}
The set of $\overline{\kappa}$-points of
$\Sec^e(X/C/\kappa) \times \text{Sym}^m(C)$ 
corresponding to pairs $(\tau, D)$ such that
$\tau$ is $D$-free is open.
\end{lem}

\begin{proof}
The first part follows from Lemma~\ref{lem-semicty}.
To see the second part consider a
family $\tau : T\times_\kappa C \to X_{f,\text{smooth}}$.
Let $T' = T \times_\kappa \text{Sym}^m(C)$.
On $T' \times _\kappa C$ we have a ``universal'' relative
Cartier divisor $\mc{D}$ (of degree $m$ over $T'$).
Hence we may consider the open $U \subset T' = T \times_\kappa \text{Sym}^m(C)$
where the relevant $R^1$ sheaf vanishes. Let $V \subset T$
be the largest open subset such that $V \times_\kappa \text{Sym}^m(C)$
is contained in $U$. We leave
it to the reader to verify that a base change by $T'' \to T$
of $\tau$ is $m$-free if and only if $T'' \to T$ factors
through $V$.
The third part is similar to the others.
\end{proof}

\begin{lem}
\label{lem-newfree2}
\marpar{lem-newfree2}
In situation \ref{hyp-curves}. Let
$\tau : T\times_\kappa C \to X_{f,\text{smooth}}$
be a family of sections, let $u : T' \to T$ be a morphism
and let $\tau' = \tau \circ (u, \text{id}_C)$.

\mni
\textbf{1.}
Let $D$ be an effective Cartier divisor in $C$.
If $\tau$ is weakly $D$-free, then so is $\tau'$.

\mni
\textbf{2.}
Let $m \geq 0$ be an integer.
If $\tau$ is weakly $m$-free then so is $\tau'$.

\mni
The converse holds if $u$ is faithfully flat.
\end{lem}

\begin{proof}
We first prove 1. The sheaf of $\OO_T$-modules 
$\mc{G} = (\text{pr}_T)_*(\tau^* T_f \otimes_{\OO_{T\times C}}
\text{pr}_C^*\OO_D)$ from Definition \ref{defn-newfree}
is finite locally free (of rank the degree
of $D$ times the dimension of $X$ over $C$)
and its formation commutes with base change.
The sheaf $\mc{F} = (\text{pr}_T)_*(\tau^* T_f)$ is quasi-coherent
but its formation only commutes with flat base change in general.
Denote $\mc{F}'$ the corresponding sheaf for $\tau'$.
Assume the map $\mc{F} \to \mc{G}$ is surjective.
Then $\mc{F}' \to \mc{G}' = u^*\mc{G}$ is surjective because
$u^*\mathcal{F} \to u^*\mc{G}$ factors through $\mathcal{F}'$.
The statement about the faithfully flat case follows from this
discussion as well.

\mni
To prove 2 we use the arguments of
1 for the universal situation over
$T \times_\kappa \text{Sym}^m(C)$
as in the proof of Lemma \ref{lem-newfreeopen}
above (left to the reader).
\end{proof}

\begin{lem}
\label{lem-newfree0}
\marpar{lem-newfree0}
In situation \ref{hyp-curves}. Let
$\tau : T\times_\kappa C \to X_{f,\text{smooth}}$
be a family of sections.

\mni
\textbf{1.}
Let $D'\leq D$ be effective Cartier divisors on $C$.
If $\tau$ is $D$-free, resp.\ weakly $D$-free, then it
is also $D'$-free, resp.\ weakly $D'$-free.

\mni
\textbf{2.}
For integers $m\geq m' \geq 0$, if $\tau$ is
$m$-free, resp.\ weakly $m$-free, then it is also $m'$-free,
resp.\ weakly $m'$-free.
\end{lem}

\begin{proof}
Consider the finite locally free sheaves
of $\OO_T$-modules defined by $\mc{G} =
(\text{pr}_T)_*(\tau^* T_f \otimes_{\OO_{T\times C}}
\text{pr}_C^*\OO_D)$
and
$ \mc{G}' = (\text{pr}_T)_*(\tau^* T_f \otimes_{\OO_{T\times C}}
\text{pr}_C^*\OO_{D'})$
that occur in Definition \ref{defn-newfree}. It is clear that
there is a surjection $\mc{G} \to \mc{G}'$. On the other
hand consider the sheaves $ \mc{H} = R^1(\text{pr}_T)_*(
\tau^* T_f \otimes_{\OO_{T\times C}}\text{pr}_C^* \OO_C(-D)
)$ and $ \mc{H}' = R^1(\text{pr}_T)_*(
\tau^* T_f \otimes_{\OO_{T\times C}}\text{pr}_C^* \OO_C(-D')
)$. There is a surjection $\mc{H} \to \mc{H}'$
induced by the map $\OO_C(-D) \to \OO_C(-D')$ with finite cokernel
(tensored with $\tau^* T_f$ this has a cokernel finite over $T$).
In this way we deduce 1.

\mni
In order to prove 2 we work on
$T' = T \times_\kappa \text{Sym}^m(C) \times_\kappa \text{Sym}^{m'-m}(C)$.
There are ``universal'' relative effective Cartier divisors
$\mc{D} \subset T' \times_\kappa C$ (of degree $m$) and 
$\mc{D}' \subset T' \times_\kappa C$ (of degree $m'$) with
$\mc{D} \subset \mc{D}'$. There are sheaves of $\OO_{T'}$-modules
$\mc{G}$, $\mc{G}'$, $\mc{H}$, and $\mc{H}'$ which are
variants of the sheaves above defined using the relative
Cartier divisors $\mc{D}$ and $\mc{D}'$. If $\tau$ is
$m$-free, resp.\  weakly $m$-free, then $\mc{H} = 0$,
resp.\ $(\text{pr}_{T'})_*((\tau')^* T_f)$ surjects onto
$\mc{G}$. The arguments just given to prove
part 1 show that there are canonical surjections
$\mc{H} \to \mc{H}'$ and $\mc{G} \to \mc{G}'$. Hence we
deduce that $\mc{H}' = 0$ resp.\ $(\text{pr}_{T'})_*((\tau')^* T_f)$
surjects onto $\mc{G}'$. Then this in turn implies that
$\tau$ is $m$-free, resp.\ weakly $m$-free.
\end{proof}

\begin{lem}
\label{lem-newfree1}
\marpar{lem-newfree1}
In situation \ref{hyp-curves}. Let
$\tau : T\times_\kappa C \to X_{f,\text{smooth}}$
be a family of sections.

\mni
\textbf{1.}
Let $D$ be an effective Cartier divisor on $C$.
Then $\tau$ is $D$-free if and only if it is both
$0$-free and weakly $D$-free.

\mni
\textbf{2.}
Let $m\geq 0$ be an integer. Then $\tau$ is
$m$-free if and only if it is both $0$-free
and weakly $m$-free.
\end{lem}

\begin{proof}
We first prove 1.
Saying a family is $0$-free is equivalent to saying it is $D'$-free
where $D'$ is the empty Cartier divisor. Thus, if a family is
$D$-free then it is $0$-free by Lemma~\ref{lem-newfree0}.
The long exact sequence of higher direct images associated
to the short exact sequence
$$
0
\rightarrow
\tau^*T_f \otimes_{\OO_{T\times_\kappa C}} \text{pr}_C^*\OO_C(-D) 
\rightarrow
\tau^*T_f
\rightarrow
\tau^*T_f \otimes_{\OO_{S\times_\kappa C}} \text{pr}_C^* \OO_D
\rightarrow 
0
$$
shows that 
$$
(\text{pr}_T)_*(\tau^*T_f)
\rightarrow
(\text{pr}_T)_*(\tau^*T_f \otimes_{\OO_{T\times_\kappa C}} \text{pr}_C^*
\OO_D)
$$
is surjective so long as $R^1(\text{pr}_T)_*(\tau^*T_f
\otimes_{\OO_{T\times_\kappa C}} \text{pr}_C^*\OO_C(-D))$ is the zero
sheaf.  Moreover, the converse holds so long as
$R^1(\text{pr}_S)_*(\tau^*T_f)$ is the zero sheaf.
From this it follows that if a family of sections is $D$-free,
then it is weakly $D$-free and the converse holds so long as the
family is $0$-free.

\mni
The proof of 2 is left to the reader. (It is similar to the
proofs above.)
\end{proof}

\begin{lem}
\label{lem-newfree3}
\marpar{lem-newfree3}
In situation \ref{hyp-curves}. Let
$\tau : T\times_\kappa C \to X_{f,\text{smooth}}$
be a family of sections.

\mni
\textbf{1.}
Let $D$ be a nontrivial effective Cartier divisor of $C$. If
$\deg(D) \geq 2g(C)$ and if $\tau$ is weakly $D$-free,
then $\tau$ is $(\deg(D) - 2g(C))$-free
(and thus $D$-free by the previous lemmas).

\mni
\textbf{2.}
Let $m \geq 2g$ be an integer, and $m \geq 1$ if $g(C) = 0$.
Every weakly $m$-free family is $m$-free.  
\end{lem}

\begin{proof}
We prove part 1.
By Lemma ~\ref{lem-newfreeopen}, there is a maximal open subscheme $V$
of $T$ over which $\tau$ is $(\deg(D)-2g(C))$-free.
The goal is to prove that $V$
is all of $T$.  For this, it suffices to prove that every geometric
point of $S$ factors through $V$. By Lemma~\ref{lem-newfree2}, the
basechange of the family to any geometric point is weakly $D$-free.
Thus, it suffices to prove the lemma when $\kappa$ is algebraically
closed and $T = \text{Spec}(\kappa)$, i.e., $\tau$ is a section
$\tau: C \rightarrow X_{f,\text{smooth}}$.

\mni
Let $m$ denote $\text{deg}(D)$.
Let us write $\mc{T} = \tau^* T_f$, and let $r$ be the
rank of $\mc{T}$. The assumption is that the restriction map
$$
H^0(C, \mc{T}) \rightarrow H^0(C,\mc{T}\otimes_{\OO_C} \OO_D)
$$
is surjective. Let $p \in D$ be a point. The restriction map
$$
H^0(C,\mc{T}\otimes_{\OO_C}\OO_C(-D+p)) \rightarrow
H^0(C,\mc{T}\otimes_{\OO_C} \kappa(p))
$$
is also surjective.  In other words, there exist $r$ global sections of
$\mc{T}(-D+p)$ generating the fibre at $p$.  These global sections
give an injective $\OO_C$-module homomorphism
$\OO_C^{\oplus r} \rightarrow \mc{T}(-D+p)$
which is surjective at $p$. Thus the cokernel is a torsion sheaf.
Let $E \subset C$ be any effective divisor of degree $m - 2g(C)$.
Twisting by $D - E - p$, this gives an injective $\OO_C$-module
homomorphism
$$
\OO_C(D - E - p)^{\oplus r} \rightarrow \mc{T}(-E)
$$
whose cokernel is torsion.  From the long exact sequence of
cohomology, there is a surjection
$$
H^1(C,\OO_C(D - E - p))^{\oplus r} \rightarrow H^1(C,\mc{T}(-E)).
$$
Since $\text{deg}(D - E - p) > 2g(C)-2$ we see that
$h^1(C,\OO_C(D-E-p))$ equals $0$.  Therefore $h^1(C,\mc{T}(-E))$ also
equals $0$.

\medskip\noindent
Part 2 is proved in a similar fashion.
\end{proof}

\begin{prop}
\label{prop-newfree4}
\marpar{prop-newfree4}
Situation as in \ref{hyp-curves}.
Let $D \subset C$ be an effective divisor. 

\mni
\textbf{1.}
Let $\tau : T\times_\kappa C \to X_{f,\text{smooth}}$
be a family of sections. If the induced morphism
$
(\tau|_{T\times_\kappa D}, \text{Id}_D):
T \times_\kappa D
\to
X_{f,\text{smooth}}\times_C D
$
is smooth, then the family is weakly $D$-free.

\mni
\textbf{2.}
Let $U$ be the open subscheme of
$\Sec^e(X_{f,\text{smooth}}/C/\kappa)$ over
which the universal section is $D$-free (see Lemma \ref{lem-newfreeopen}).
The morphism
$$
(\sigma,\text{Id}_D):U \times_\kappa D \rightarrow
X_{f,\text{smooth}}\times_C D
$$
is smooth.

\mni
\textbf{3.}
The scheme $\Sec(X/C/\kappa)$ is smooth at every point
which corresponds to a section that is $D$-free.
\end{prop}

\begin{proof}
When $D$ is the empty divisor, both 1 and 2 are vacuous, and
3 follows from the vanishing of the obstruction space for
the deformation theory. When $D$ is not empty, one argues as
in \cite[Proposition II.3.5]{K}.
\end{proof}

\noindent
In order to state the following two propositions
we need some notation. Consider the Hilbert scheme $\text{Hilb}^m_{X/\kappa}$
of finite length $m$ closed subschemes 
of $X$ over $\kappa$. There is an open subscheme
$H^m_{X,f} \subset \text{Hilb}^m_{X/\kappa}$ parametrizing
those $Z \subset X$ such that (a) $Z \subset X_{f,\text{smooth}}$ and
(b) $f|_Z : Z \to C$ is a closed immersion. By construction
there is a natural quasi-projective morphism
$$
H^m_{X,f} \longrightarrow \text{Sym}^m(C).
$$
A point of $H^m_{X,f}$ also corresponds to a pair $(Z', \sigma)$,
where $Z' \subset C$ is a closed subscheme of length $m$, and
$\sigma : Z' \to X_{f, \text{smooth}}$ which is a section of $f$.
And there are no obstructions to deforming a
map from a zero dimensional scheme to a smooth scheme. Therefore
the morphism $H^m_{X,f} \to \text{Sym}^m(C)$ is smooth.

\medskip\noindent
The fibre of this morphism
over $[D]$ (where $D$ is a degree $m$ effective Cartier
divisor on $C$) is the space of sections of
$D\times_C X_{f,\text{smooth}} \to D$. Let us denote this fibre
by $H^m_{X,f,D}$. Note that $H^m_{X,f,D}$ equals the Weil restriction
of scalars $\text{Res}_{D/\kappa}(D\times_C X_{f,\text{smooth}})$. In fact
$H^m_{X,f}$ is a Weil restriction of scalars as well, which leads
to another construction. Namely
$H^m_{X,f} = 
\text{Res}_{\mc{D}/\text{Sym}^m(C)}(\mc{D}\times_C X_{f,\text{smooth}})$,
where $\mc{D} \subset \text{Sym}^m(C)\times_\kappa C$ is
the universal degree $m$ effective Cartier divisor on $C$.
We denote the evaluation map
$$
ev^m :
H^m_{X,f} \times_{\text{Sym}^m(C)} \mc{D}
\longrightarrow
X_{f,\text{smooth}}
$$

\medskip\noindent
Note that given a section $\tau$ of $f : X \to C$ which
ends up in the smooth locus of $f$, we may restrict $\tau$
to any divisor $D$ and obtain a point of $H^m_{X,f,D}$.
In other words there is a canonical morphism
$$
res : 
\Sec(X_{f,\text{smooth}}/C/\kappa) \times_\kappa \text{Sym}^m(C)
\longrightarrow
H^m_{X,f}
$$
of schemes over $\text{Sym}^m(C)$.

\begin{lem}
\label{lem-resfreesmooth}
The morphism $res$ is smooth at any point
$(\tau, D)$ such that $\tau$ is $D$-free.
\end{lem}

\begin{proof}
This is akin to Proposition \ref{prop-newfree4}
and proved in the same way. It also follows from part 2 of that
Proposition by the fibrewise criterion of smoothness since
both $\Sec(X_{f,smooth}/C/\kappa) \times_\kappa \text{Sym}^m(C)$ and
$H^m_{X,f}$ are smooth over $\text{Sym}^m(C)$ at the corresponding
points.
\end{proof}

\noindent
Here are two technical results
that will be used later.

\begin{prop}
\label{prop-KMM}
\marpar{prop-KMM}
See \cite[1.1]{KMM92c}.
In situation \ref{hyp-curves}, with $\text{char}(\kappa) = 0$.
Let $e$ and $m \geq 0$ be integers. There exists a constructible subset
$W^m_e$ of $H^m_{X,f}$ whose intersection $W^m_e \cap H^m_{X,f,D}$
with the fibre over every point $[D]$ of $\text{Sym}^m(C)$ contains
a dense open subset of $H^m_{X,f,D}$ with the following property:
For every $\kappa$-scheme $T$ and for every family of degree $e$ sections
$$
\tau :
T \times_\kappa C
\longrightarrow
X_{f,\text{smooth}},
$$
if the restriction of $\tau$ to $T \times_\kappa D$ is in $W_e\cap
H^m_{X,f,D}$, then $\tau$ is weakly $D$-free. 
\end{prop}

\begin{proof}
It suffices to construct $W^m_e$ over $\overline{\kappa}$,
and hence we may and do assume $\kappa$ algebraically closed.
Stratify $\Sec^e(X_{f,\text{smooth}}/C/\kappa)$ by a
finite set of irreducible locally closed strata $T_i$,
each \emph{smooth} over $\kappa$. For each $i$ we have the restriction
morphism $res_i : T_i \times \text{Sym}^m(C) \to H^m_{X,f}$.
This is a morphism of smooth schemes of finite type over $\text{Sym}^m(C)$.
Let $S_i \subset T_i \times \text{Sym}^d(C)$ be closed set of
points where the derivative of $res_i$
$$
\text{d}res_i : T_{T_i \times \text{Sym}^m(C)} 
\longrightarrow
res_i^*(T_{H^m_{X,f}})
$$
is not surjective. By Sard's theorem, or generic smoothness the image
$E_i = res_i(S_i) \subset H^m_{X,f}$ is not dense in
any fibre $H^m_{X,f,D}$. It is constructible
as well by Chevalley's theorem. Take
$W^m_e$ to be the complement of the finite union $\bigcup E_i$.
The reason this works is the following: Suppose that
$\tau : C \to X_{f,\text{smooth}}$ is a section such that
$\tau|_D$ is in $W^m_e$. Let $i$ be such that
$\tau$ corresponds to a point $[\tau] \in T_i$. By construction
of $W^m_e$ the morphism $T_i \to H^m_{X,f,D}$ is smooth with surjective
derivative at $\tau$. Note that $T_{\tau}(T_i) \subset H^0(C, \tau^*T_f)$,
and that $T_{res(\tau \times [D])}(H^m_{X,f,D}) = H^0(C, \tau^*T_f|_D)$.
Thus the surjectivity of the derivative
in terms of Zariski tangent spaces implies
that $\tau$ is weakly $D$-free.
\end{proof}

\begin{lem}
\label{lem-avoid}
\marpar{lem-avoid}
Let $k$ be an uncountable algebraically closed field.
Suppose we are given the commutative ``input''
diagram of varieties over $k$:
$$
\text{input: }
\xymatrix{
V \ar[d]^\Phi & & \\
U \ar[d] \ar[r] & Y \ar[ld] & C\times T \ar[l]_\tau \ar[lld] \\
C & & 
}\ \ \ \text{output: }
\xymatrix{
V \ar[d] \ar[r] & \overline{V} \ar[d] & C\times T' \ar[d] \ar[l] \\
U \ar[d] \ar[r] & Y \ar[ld] & C\times T \ar[l] \ar[lld] \\
C & & 
}
$$
We assume that $C$ is a nonsingular projective curve,
$Y \to C$ is proper, $U \subset Y$ is open dense,
$U \to C$ is smooth with geometrically irreducible fibres,
$V \to U$ is smooth surjective with birationally rationally
connected geometric fibres. We assume that
the image of $C \times T \to Y$ meets $U$.
Let $m > 0$. We assume the associated rational map (see proof)
from $T \times \text{Sym}^m(C)$ to $H^m_U$ is dominant.
Then there exists a variety $T'$, a dominant morphism $T' \to T$,
a compactification $V \subset \overline{V}$ and
a morphism $C \times T' \to \overline{V}$ fitting into the
commutative ``output'' diagram such that the image
of $C \times T' \to \overline{V}$ meets $V$ and the associated 
rational map from $T' \times \text{Sym}^m(C)$ to
$H^m_V$ is dominant.
\end{lem}

\begin{rmk}
This technical lemma just says that a general $m$-free section
of $Y \to C$ lifts to an $m$-free section of $V$.
\end{rmk}

\begin{proof}
For $t\in T(k)$ we denote $\tau_t : C \to Y$ the restriction of
$\tau$ to $C \times \{t\}$. In other words, we think of $\tau$
as a family of maps from $C$ to $Y$.
Let $W \subset \text{Sym}^m(C) \times T$ denote
the set of pairs $(D, t)$ such that $\tau_t(D) \subset U$.
This is nonempty by assumption.
The ``associated rational map'' is the map
$$
\text{Sym}^m(C) \times T \supset W \longrightarrow H^m_U,\ \ 
(D, t) \longmapsto (\tau_t|_D :D \to U).
$$
Note that $V$ is a nonsingular variety. 
We choose a smooth projective
compactification $V \subset \overline{V}$ such that
$\Phi$ extends to a morphism $\overline{\Phi} : \overline{V} \to Y$,
see \cite{Hir}. By Sard's Theorem we may shrink $U$ and assume that
the fibres of $\overline{\Phi}$ are projective smooth varieties
over all points of $U$. In this case these fibres are rationally
connected varieties by our assumption on $\Phi$. We
may shrink $T$ and assume that
$\tau_t(C)$ meets $U$ for all $t \in T(k)$.

\medskip\noindent
The assumption on the fibres of $\Phi$ implies that
for every $t\in T(k)$ we may find a section
$\sigma : C \to \overline{V}$ such that $\overline{\Phi}
\circ \sigma = \tau_t$. See \cite{GHS}.
It was already known before the results of \cite{GHS},
see for example \cite{KMM}, that one may, given any $m$ sufficiently
general points $c_1,\ldots,c_m$ of the curve $C$, assume
the $m$-tuple
$(\sigma(c_i) \in \overline{V}_{\tau_t(c_i)})_{i=1\ldots m}$
is a general point of
$V_{\tau_t(c_1)} \times \ldots \times V_{\tau_t(c_m)}$.
This is done by the ``smoothing combs'' technique of \cite{KMM}.
Namely, after attaching sufficiently many free teeth
(in fibres of $\overline{V} \to Y$) to $\sigma$,
the comb can be deformed to a section $\sigma$ which passes
through $m$ very general points of $m$ very general fibres.
(Another approach would be to use the main result of the
paper \cite{HT06}. Note that since the fibres of $\overline{V} \to U$
are smooth and rationally connected we are allowed to use that result.
Of course this paper proves something much stronger!)
Combined with the assumption that
$\text{Sym}^m(C) \times T \supset W \longrightarrow H^m_U$
is dominant this implies that in fact the $m$-points
$\sigma(c_i)$ hit a general point of 
$V_{c_1} \times \ldots \times V_{c_m}$ which is the fibre of
$H^m_V \to \text{Sym}^m(C)$ over the point $\sum c_i$.

\medskip\noindent
The paragraph above solves the problem ``pointwise'', but since
the field $k$ is uncountable this implies the result of the lemma
by a standard technique. Namely, the above produces lots (uncountably
many) of points on the scheme
$\Sec(\overline{V}\times_Y(C\times T)/(C\times T)/T)$
which implies that it must have a suitable irreducible component
$T'$ dominating $T$ (details left to the reader).
\end{proof}

%%%%%%%%%%%%%%%%%%%%%%%%%%%%%%%%%%%%%%%%%%%%%%%%%%%%%%%%%%%%%%%
%%
%% Kontsevich stable maps
%% 
%%%%%%%%%%%%%%%%%%%%%%%%%%%%%%%%%%%%%%%%%%%%%%%%%%%%%%%%%%%%%%%

\section{Kontsevich stable maps}
\label{sec-Kontsevich}
\marpar{sec-Kontsevich}

\noindent
Let $k$ be an algebraically closed field and let $Y$ be a
quasi-projective $k$-scheme.

\begin{defn}
\label{defn-stablemap}
\marpar{defn-stablemap}
A \emph{stable map} to $Y$ over $k$ is given by
\begin{enumerate}
\item a proper, connected, at-worst-nodal $k$-curve $C$,
\item a finite collection $(p_i)_{i\in I}$ of distinct, smooth, closed
points of $C$, and
\item a $k$-morphism $h:C\rightarrow Y$.
\end{enumerate}
These data have to satisfy the \emph{stability} condition that the
log dualizing sheaf $\omega_{C/k}(\sum_{i\in I} p_i)$ is $h$-ample.
\end{defn}

\noindent
The stability of the triple $(C, (p_i), h)$ can also be expressed
in terms of the numbers of special points on $h$-contracted components
of $C$ of arithmetic genus $0$ and $1$, or by saying that the
automorphism group scheme of the triple is finite. See \cite{Kv}.

\mni
There are many invariants of a stable map of a numerical or
topological nature.  One invariant is the number $n = \#I$ of marked
points. Occasionally it is convenient to mark the points by some
unspecified finite set $I$, but usually $I$ is simply
$\{1,2,\dots,n\}$.  Another invariant is the arithmetic genus $g$ of
$C$.  Yet another invariant is the curve class of $h_*([C])$.  Often we do
not need the full curve class.  The following definition makes precise
the notion of a ``partial curve class''. We will use $\text{NS}(Y)$
to denote the N\'eron-Severi group of $Y$, i.e., the group of
Cartier divisors up to numerical equivalence on $Y$.

\begin{defn}
\label{defn-partial}
\marpar{defn-partial}
Let $k$ be an algebraically closed field and let $Y$ be a
quasi-projective $k$-scheme.  A \emph{partial curve class} on $Y$ is a
triple $(A, i, \beta)$ consisting of 
an Abelian group $A$ and homomorphisms of Abelian groups
$$
i:A \rightarrow \text{NS}(Y)
\text{ and }
\beta:A \rightarrow \ZZ.
$$
The triple will usually be denoted by $\beta$, with $A$ being denoted
$A_\beta$ and $i$ being denoted $i_\beta$.    
We say a stable map $(C,(p_i)_{i\in I},h)$
\emph{belongs to the partial curve class $\beta$} or
\emph{has partial curve class $\beta$} if for
every element $a$ in $A$ and for every Cartier divisor
$D$ on $Y$ with class $i(a)$, the degree of $h^*D$ on
$C$ equals $\beta(a)$.  
\end{defn}

\begin{rmk}
\label{rmk-partial}
\marpar{rmk-partial}
If there exists a stable map belonging to $\beta$ then
$\text{Ker}(\beta)$ contains 
$\text{Ker}(i)$, i.e., $\beta$ factors through $i(A)\subset
\text{NS}(Y)$.  From this point of view it is reasonable to restrict
to cases where $A$ is a subgroup of $\text{NS}(Y)$.  But it is
sometimes useful to allow $i$ to be noninjective.
\end{rmk}

\noindent
We will also use a relative version of stable maps
and of partial curve classes.

\begin{defn}
\label{defn-stablemaprel}
\marpar{defn-stablemaprel}
Let $Y \to S$ be a quasi-projective morphism.

\mni
\textbf{1.}
A \emph{family of stable maps to $Y/S$} is
given by
\begin{enumerate}
\item an $S$-scheme $T$,
\item a proper, flat family of curves $\pi: C \rightarrow T$
(see Definition \ref{defn-curves}),
\item a finite collection of $T$-morphisms $(\sigma_i : T \to C)_{i\in I}$,
and 
\item a $S$-morphism $h: C \to Y$.
\end{enumerate}
These data have to satisfy the \emph{stability} condition
that for every algebraically closed field $k$ and morphism
$t: \text{Spec}(k) \to T$, the basechange
$(C_t,(\sigma_{i}(t))_{i\in I},h_t)$ is a stable map to
$Y_t = \text{Spec}(k) \times_{t,S} Y$ as in
Definition \ref{defn-stablemap}.

\mni
\textbf{2.}
Given two families of stable maps
$(T' \to S, \pi' : C'\to T',(\sigma'_i)_{i\in I}, h')$ and
$(T \to S, \pi : C\to T,(\sigma_i)_{i\in I}, h)$
a \emph{morphism} from the first family to the second family is
given by an $S$-morphism $u : T' \to T$, and a morphism
$\phi : C' \to C$ such that (a) $(\phi, u)$ forms a morphism
of flat families of curves (see \ref{defn-curves}), (b) $\sigma_i \circ u =
\phi \circ \sigma'_i$, and (c) $h' = h \circ \phi$.
\end{defn}

\noindent
Let $Y \to S$ be a quasi-projective morphism. Denote $\text{Pic}_{Y/S}$
the relative Picard functor, see \cite[Section 8.1]{Ner}. It is an
fppf sheaf on the category of schemes over $S$. We define the
relative N\'eron-Severi functor $\text{NS}_{Y/S}$ of $Y/S$
to be the quotient of $\text{Pic}_{Y/S}$ by the subsheaf of
sections numerically equivalent to zero on all geometric fibres.
Finally, we define the relative N\'eron-Severi group
$\text{NS}(Y/S) = \text{NS}_{Y/S}(S)$ to be the sections
of this sheaf over $S$.

\begin{defn}
\label{defn-partialrel}
\marpar{defn-partialrel}
A \emph{partial curve class} on $Y/S$ is a triple $(A,i,\beta)$
consisting of an Abelian group $A$, and homomorphisms of Abelian groups
$$
i:A\rightarrow \text{NS}(Y/S)
\text{ and }
\beta:A \rightarrow \ZZ.
$$
For every morphism of schemes $T \to S$
there is a \emph{pullback partial curve class} on $Y_T/T$
given by $(A,i_T,\beta)$ where $i_T$ is the composition
$$
A \xrightarrow{i} \text{NS}(Y/S) \rightarrow \text{NS}(Y_T/T).
$$
A family of stable maps
$(T \to S, \pi : C\to T,(\sigma_i)_{i\in I}, h)$ to $Y/S$
\emph{belongs to the partial curve class $\beta$} or
\emph{has partial curve class $\beta$} if for every
algebraically closed field $k$ and every morphism $t:\text{Spec}(k) \to T$,
the pullback $(C_t,(\sigma_{i}(t))_{i\in I},h_t)$
belongs to the pullback partial curve class $(A, i_t, \beta)$.  
\end{defn}

\noindent
Suppose that $(T \to S, \pi : C\to T,(\sigma_i)_{i\in I}, h)$
is a family of stable maps to $Y/S$ as in Definition \ref{defn-stablemaprel}.
For every $g \geq 0$ there is an open and closed subscheme of $T$
where the geometric fibres of $\pi$ have genus $g$. This is true because
$R^1\pi_* \OO_C$ is a finite locally free sheaf of $\OO_T$-modules,
and the genus is its rank. In addition, given a relative partial
curve class $\beta$, there is an open and closed subscheme of $T$
where the geometric fibres $h_t : C_t \to Y_t$ belong to the
partial curve class $\beta_t$. This is because the degree of
an invertible sheaf on a proper flat family of curves is locally
constant on the base. Thus the following definition makes sense.

\begin{defn}
\label{defn-Kontsevich}
\marpar{defn-Kontsevich}
Let $Y \to S$ be quasi-projective. For every finite set $I$ the
\emph{Kontsevich stack of stable maps to $Y/S$}
is the stack
$$
\Kgnb{*,I}(Y/S)
\longrightarrow
\textit{Schemes}/S
$$
whose objects and morphisms are as in Definition \ref{defn-stablemaprel}
above. When $I$ is just $\{1,\dots,n\}$, this stack is denoted
$\Kgnb{*,n}(Y/B)$. For every integer $g\geq 0$, there is an open
and closed substack $\Kgnb{g,I}(Y/B)$ parametrizing stable maps
whose geometric fibres all have arithmetic genus $g$.  And for
every partial curve class $\beta$ on $Y/S$, there is an open and
closed substack $\Kgnb{g,I}(Y/B,\beta)$ parametrizing stable maps
belonging to the partial curve class $\beta$.  
\end{defn}

\noindent
The following theorem is well known to the experts.
We state it here for convenience, and because we do
not know an exact reference.

\begin{thm}
\label{thm-Kontsevich}
\marpar{thm-Kontsevich}
Compare \cite{Kv}. Let $Y \to S$ be quasi-projective, and
$S$ excellent.

\mni
\textbf{1.}
The stack $\Kgnb{*,I}(Y/S)$ is a locally finitely presented, algebraic
stack over $S$ with finite diagonal, satisfying the valuative criterion
of properness over $S$.

\mni
\textbf{2.}
If $S$ lives in characteristic $0$ then $\Kgnb{*,I}(Y/S)$ is a  
Deligne-Mumford stack over $S$.

\mni
\textbf{3.}
If $Y$ is proper over $S$ and if $(A,i,\beta)$
is a partial curve class such that
$i(A)$ contains the class of a $S$-relatively
ample invertible sheaf on $Y$,
then $\Kgnb{g,I}(Y/S,\beta)$ is proper over
$S$ with projective coarse moduli space.
\end{thm}

\begin{proof} (Sketch.)
Forgetting the sections gives a $1$-morphism
from $\Kgnb{*,I}(Y/S)$ to $\textit{CurveMaps}(Y/S)$.
It is easy to see that this morphism is representable and smooth.
Hence, by Proposition \ref{prop-Lieblich} the stack $\Kgnb{*,I}(Y/S)$
is algebraic of locally of finite presentation over $S$,
with finitely presented diagonal.
For the valuative criterion of properness, see \cite[bottom page 338]{Kv}.
Finiteness of the diagonal is, in view of quasi-compactness (see above),
a consequence of the valuative criterion of properness
for $\text{Isom}$ between stable maps (basically trivial),
combined with the fact that by definition stable maps have finite
automorphism groups (gives quasi-finiteness).
This proves 1. In characteristic zero a stack with finite diagonal
is Deligne-Mumford, because automorphism group schemes are reduced.
Finally, to see 3 note that $\Kgnb{g,I}(Y/S,\beta)$ is quasi-compact
in this case by the argument in \cite[page 338]{Kv}. Thus it
is proper over $S$, see \cite{OProper} and \cite[Appendix]{Faltings}
to see why. By \cite{K-M} the coarse moduli space exists, and
by the above it is proper over $S$. The projectivity of this coarse
moduli space will not be used in the sequel; a reference is
\cite[Section 8.3]{AV}.
\end{proof}

\begin{notat}
\label{notat-e}
\marpar{notat-e}
A simple case of a partial curve class is as follows.
Let $Y \to S$ be quasi-projective and let
$\mc{L}$ be an $S$-ample invertible sheaf on $Y$.
For every integer $e$ there is a partial curve class 
$\beta = (\ZZ, 1 \mapsto [\mc{L}], 1 \mapsto e)$.
When $\mc{L}$ is understood, the corresponding
Kontsevich moduli space is denoted
$\Kgnb{g,n}(Y/S,e)$.
\end{notat}

%%%%%%%%%%%%%%%%%%%%%%%%%%%%%%%%%%%%%%%%%%%%%%%%%%%%%%%%%%%%%%%
%%
%% Stable sections and Abel maps
%% 
%%%%%%%%%%%%%%%%%%%%%%%%%%%%%%%%%%%%%%%%%%%%%%%%%%%%%%%%%%%%%%%

\section{Stable sections and Abel maps}
\label{sec-ssections}
\marpar{sec-ssections}

\noindent
Let $C\to S$ be a proper, flat family of curves (see \ref{defn-curves})
whose geometric fibres are connected smooth projective curves of
some fixed genus $g(C/S)$. The group $\text{NS}(C/S)$ (see discussion
preceding \ref{defn-partialrel}) is a free Abelian group of rank $1$,
simply because any two invertible sheaves on the family are numerically
equivalent on geometric fibres if and only if they have the same
degree on the fibres. Denote $[\text{point}]$ the element of
$\text{NS}(Y/S)$ that fppf locally on $S$ corresponds to an invertible
sheaf on $C$ having degree $1$ on the fibres.

\mni
Consider the relative partial curve class
$\beta_{C/S} = (\ZZ, 1 \mapsto [\text{point}], 1 \mapsto 1)$
on $C/S$ (compare \ref{notat-e}). For all
nonnegative integers $g$ and $n$, we denote
$$
\Kgnb{g,n}(C/S,1) := \Kgnb{g,n}(C/S, \beta_{C/S}).
$$
Loosely speaking the stable maps $h$ arising from points of
this algebraic stack have one component that maps birationally
to a fibre of $C \to S$, and all other components are contracted.
If $g = g(C/S)$, then the contracted components all have genus
zero.

\mni
Next, let $f : X \to C$ be a proper morphism.
Let $\mc{L}$ be an $f$-ample invertible sheaf on $X$.
Denote by $[\text{fibre}]$ the pullback $f^*[\text{point}]$
in $\text{NS}(X/S)$.  For every integer $e$, denote by $\beta_e$
the relative partial curve class $(\ZZ \oplus \ZZ, i, \beta_e)$
with $i(1,0) = [\text{fibre}]$, $i(0,1) = \mc{L}$ and
$\beta_e(1,0) = 1$, $\beta_e(0,1) = e$. Using this notation,
for all nonnegative integers $g$ and $n$, we may consider
the algebraic stacks
$$
\Kgnb{g,n}(X/S,\beta_e)
$$
introduced in the previous section. Loosely speaking a stable map
$h$ arising from points of this algebraic stack give rise to
generalized sections of $X_s/C_s$ for some geometric point $s$
of $S$ having total degree $e$ with respect to $\mc{L}$.

\mni
There is an obvious way in which $\beta_e$ and $\beta_{C/S}$ are related.
Thus by \cite[Theorem 3.6]{BM}, there are $1$-morphisms of
algebraic $S$-stacks
$$
\Kgnb{g,n}(f) :
\Kgnb{g,n}(X/S,\beta_e)
\longrightarrow
\Kgnb{g,n}(C/S,1).
$$

\begin{defn}
\label{defn-relstablemap}
\marpar{defn-relstablemap} 
Given $C \to S$, $f : X \to C$ and $\mc{L}$ as above.
For every nonnegative integer $n$ and every integer $e$,
\emph{the space of $n$-pointed, degree $e$, stable sections of $f$}
is the algebraic stack
$$
\Sigma^e_n(X/C/S) := 
\Kgnb{g(C/S),n}(X/S,\beta_e).
$$
We will also use the variant $\Sigma^e_I(X/C/S)$
where the sections are labeled by a given finite set $I$.
\end{defn}

\begin{defn}
\label{defn-sections}
\marpar{defn-sections}
Compare Definitions \ref{defn-sec} and \ref{defn-FGA}.
With notation as above.
A \emph{family of $n$-pointed, degree $e$ sections} of $X/C/S$
over $T \to S$ is given by a family $(\tau,\theta)$ of sections
of $f$ (see Definition \ref{defn-preFGA}) together with pairwise
disjoint sections $\sigma_i : T \to T\times_S C$, $i=1,\ldots,n$,
such that $\tau^*\mc{L}$ has degree $e$ on the fibres of
$T\times_SC\to T$. We leave it to the reader to define \emph{morphisms}
of families of $n$-pointed, degree $e$ sections of $X/C/S$.
\end{defn}

\noindent
We will denote such a family as $(T\to S, \sigma_i : T \to T\times_S C ,
h : T\times_S C \to X)$; in other words $\tau$ is replaced by $h$ and
we drop the notation $\theta$ since in this case it just signifies
that $f \circ h$ equal the projection map $T\times_S C \to S$.

\mni
Families of $n$-pointed, degree $e$ sections of $X/C/S$ naturally
form an algebraic space $\Sec^e_n(X/C/S)$ locally of finite presentation
over $S$. This is because the natural $1$-morphism
$$
\Sec^e_n(X/C/S) \longrightarrow \Sec(X/C/S)
$$
is representable and smooth, combined with Theorem \ref{thm-algebraic}.
There is another obvious $1$-morphism, namely
$$
\Sec^e_n(X/C/S) \longrightarrow \Sigma^e_n(X/C/S).
$$
This $1$-morphism is representable by open immersions of schemes. In this
precise sense, the proper, algebraic $S$-stack $\Sigma^e_n(X/C/S)$ is a
``compactification'' of $\Sec^e_n(X/C/S)$.

\mni
We will also use the variant $\Sec^e_I(X/C/S)$ where the sections
are labeled by a given finite set $I$.

\begin{notat}
\label{notat-index}
\marpar{notat-index}
Let $n$ be a nonnegative integer.
Let $I$ be a set of cardinality $n$.
Let $e$ be an integer.
Let $\delta$ be a nonnegative integer.
Let $J$ be a set of cardinality of $\delta$.
Let 
$
\underline{e} = (e_0, (e_j)_{j\in J})
$
be a collection of 
integers such that 
$e = e_0 + \sum_{j\in J} e_j$.
Let
$
\ul{I} = (I_0,(I_j)_{j\in J})
$
be a collection of subsets of $I$ such that
$
I = I_0 \sqcup \sqcup_{j\in J} I_j 
$ 
is a partition of $I$.
The triple $(J,\ul{e},\ul{I})$ is called an \emph{indexing triple} for
$(e,I)$.   
In the special case that $I=\emptyset$, so that every subset $I_0$ and
$I_j$ is also $\emptyset$, this is called an \emph{indexing pair} for $e$
denoted $(J,\ul{e})$.
\end{notat}

\begin{defn}
\label{defn-ndcomb}
\marpar{defn-ndcomb}
A \emph{family of $\ul{I}$-pointed, degree $\ul{e}$ combs} in
$(X/C/S, \mc{L})$ as above is given by the following data
\begin{enumerate}
\item an $S$-scheme $T$
\item an object 
$
(\pi_0:\mc{C}_0 \rightarrow T,
(p_i:T\rightarrow \mc{C}_0)_{i\in I_0} \cup
(q_j:T\rightarrow \mc{C}_0)_{j\in J},
h_0:\mc{C}_0 \rightarrow X)
$
of $\Sigma^{e_0}_{I_0\sqcup J}(X/C/S)$ over $T$, and
\item for $j\in J$, an object
$
(\pi_j:\mc{C}_j \rightarrow T,
(p_i:T\rightarrow \mc{C}_j)_{i\in I_j} \cup
(r_j:T\rightarrow \mc{C}_j),
h_j:\mc{C}_j \rightarrow X)
$
of $\Kgnb{0,I_j\sqcup \{j\}}(X/C,e_j)$ as in Notation \ref{notat-e}
over $T$.
\end{enumerate}
These data should satisfy the requirement that
for every $j\in J$, $h_0\circ q_j$ equals $h_j\circ r_j$ as morphisms
$T\rightarrow X$.
\end{defn}

\noindent
\emph{Morphisms} of $\ul{I}$-pointed, degree $\ul{e}$ combs in $X/C/S$ are
defined in the obvious way.  
The category of families of $\ul{I}$-pointed, degree $\ul{e}$ combs in
$X/C/S$ with pullback diagrams as morphisms is an $S$-stack denoted
$\Co^{\ul{e}}_{\ul{I}}(X/C/S)$.  

\mni
Given a family
$(\zeta_0,(\zeta_j)_{j\in J})$ of combs, the family $\zeta_0$ is
the \emph{handle} and the families $(\zeta_j)_{j\in J}$ are the
\emph{teeth}.  There is a ``forgetful'' $1$-morphism
$$
\Phi_\text{handle} :
\Co^{\ul{e}}_{\ul{I}}(X/C/\kappa)
\longrightarrow
\Sigma^{e_0}_{I_0\sqcup J}(X/C/\kappa)
$$
called the \emph{handle $1$-morphism}.  Similarly, for every $j\in
J$ 
there is a 
$1$-morphism
$$
\Phi_{\text{tooth},j}:\Co^{\ul{e}}_{\ul{n}}(X/C/\kappa)
\longrightarrow
\Kgnb{0,I_j\sqcup\{j\}}(X/C,e_j)
$$
called the \emph{$j^\text{th}$ tooth $1$-morphism}.  Moreover as
constructed in 
\cite[Theorem 3.6]{BM}, specifically Case II on p.\ 18, there is also a
\emph{total curve $1$-morphism}
$$
\Phi_{\text{total}} :
\Co^{\ul{e}}_{\ul{I}}(X/C/S)
\longrightarrow
\Sigma^e_I(X/C/S)
$$
obtained by attaching $r_j(T)$ in $\mc{C}_j$ to $q_j(T)$ in $\mc{C}_0$
for every $j\in J$.

\mni
By definition, $\Co^{\ul{e}}_{\ul{I}}(X/C/S)$ together with the handle
and tooth $1$-morphisms is the equalizer of the collection of marked
point $1$-morphisms associated to $(q_j,r_j)_{j\in J}$.  Since each of
$\Sigma^{e_0}_{I_0\sqcup J}(X/C/S)$ and
$\Kgnb{0,I_j\sqcup \{j\}}(X/C,e_j)$ is a proper, algebraic $S$-stack with
finite diagonal, the same is true for the equalizer
$\Co^{\ul{e}}_{\ul{I}}(X/C/S)$.  

\mni
The \emph{canonical open substack},
$\Co^{\ul{e}}_{\ul{I},\text{open}}(X/C/S)$,
is defined to be the inverse image under $\Phi_{\text{handle}}$ of the
open substack $\Sec^{e_0}_{I_0\sqcup J}(X/C/S)$ of
$\Sigma^{e_0}_{I_0\sqcup J}(X/C/S)$.    

\begin{prop}
\label{prop-easy}
\marpar{prop-easy}
Given indexing triples $(J',\ul{e}',\ul{I}')$ and
$(J'',\ul{e}'',\ul{I}'')$, the images of
$$
\Phi_{\text{total}}:\Co^{\ul{e}'}_{\ul{I}',\text{open}}(X/C/S)
\rightarrow \Sigma^e_I(X/C/S)
$$
and
$$
\Phi_{\text{total}}:\Co^{\ul{e}''}_{\ul{I}'',\text{open}}(X/C/S)
\rightarrow \Sigma^e_I(X/C/S)
$$
intersect only if there exists a bijection $\phi:J'\rightarrow
J''$ such that $e''_{\phi(j)} = e'_j$ and $I''_{\phi(j)} = I'_j$ for
every $j\in J'$.
If such a bijection exists, the images are
equal.  Finally, for every algebraically closed field $k$ over $S$, every
object of $\Sigma^e_I(X/C/S)$ over $\SP k$ is in the image of 
$$
\Phi_{\text{total}}:\Co^{\ul{e}}_{\ul{I},\text{open}}(X/C/S)
\rightarrow \Sigma^e_I(X/C/S)
$$
for some collection $(J,\ul{e},\ul{I})$.  
\end{prop}

\begin{proof}
This is left to the reader.
\end{proof}

\begin{lem}
\label{lem-Abel}
\marpar{lem-Abel}
For every integer $e$ there exists a $1$-morphism over $S$
$$
\alpha_{\mc{L}} :
\Sigma^e(X/C/S)
\longrightarrow
\textit{Pic}^e_{C/S}
$$
with the following property: For every indexing pair $(J,\ul{e})$ for $e$
and every family $(T \to \ldots)$ in $\Co^{\ul{e}}_{\text{open}}(X/C/S)$
over $T$ as in Definition \ref{defn-ndcomb} the composition
$$
T
\to
\Co^{\ul{e}}_{\text{open}}(X/C/S)
\to
\Sigma^e(X/C/S)
\to
\textit{Pic}^e_{C/S}
$$
corresponds to the class of the invertible sheaf
$$
h_0^*(\mc{L})(\sum_{j\in J} e_j \cdot \ul{q_j(T)})
$$
on $T\times_S C$.  
\end{lem}

\begin{proof}
We use the material preceding Definition \ref{defn-detzeta}.
Any flat proper family of nodal curves is LCI in the sense of
Definition \ref{defn-LCI}. Thus there is a morphism of stacks
$$
\Sigma^e(X/C/S)
\longrightarrow
\textit{CurvesMaps}_{\text{LCI}}(C \times B\mathbf{G}_m)
$$
which associates to $(T \to S, \pi : C' \to T, h : C' \to X)$
the object $(T \to S, \pi : C' \to T, f \circ h : C' \to C,
h^*\mc{L})$.  Hence we may use the construction
of Definition \ref{defn-detzeta} which gives a $1$-morphism
$$
\textit{CurvesMaps}_{\text{LCI}}(C \times B\mathbf{G}_m)
\longrightarrow
\textit{Pic}_{C/S}.
$$
To prove this map satisfies the property expressed in 
the lemma it suffices to apply Lemma \ref{lem-newAbel}.
\end{proof}

%%%%%%%%%%%%%%%%%%%%%%%%%%%%%%%%%%%%%%%%%%%%%%%%%%%%%%%%%%%%%%%
%%
%% Peaceful chains
%% 
%%%%%%%%%%%%%%%%%%%%%%%%%%%%%%%%%%%%%%%%%%%%%%%%%%%%%%%%%%%%%%%

\section{\Peaceful\ chains}
\label{sec-peace}
\marpar{sec-peace}

\noindent
Let $S$ be an excellent scheme.
Let $f : X \to S$ be a proper, flat morphism with geometrically
irreducible fibres. Let $\mc{L}$ be an $f$-ample invertible sheaf on $X$.
Recall (\ref{notat-e}) that $\Kgnb{0,n}(X/S,1)$ parametrizes stable
$n$-pointed genus $0$ degree $1$ maps from connected nodal curves to
fibres of $X \to S$. Over a geometric point $s : \text{Spec}(k) \to S$
such a stable map $(C/k, p_1,\ldots,p_n \in C(k), h : C \to X_s)$
is nonconstant on exactly one irreducible component $L$ of $C$ because
$\mc{L}$ is ample. Then, $L \cong \PP^1_k$, the map $L \to X_s$ is birational
onto its image and corresponds to a point of $\Kgnb{0,0}(X/S,1)$.
If $n=0$ or $n=1$ then $L = C$. If $n=2$ it may happen that $C = L \cup C'$,
with $p_1, p_2 \in C'(k)$ and where $C'$ is contracted. However, one gets
the same algebraic stack (in the case $n=2$) by considering moduli of
triples $(L, p_1,p_2\in L(k), h : L \to X_s)$ where $p_1$ and $p_2$ are
allowed to be the same point of $L$. We will use this below without 
further mention. In any case it follows that $\Kgnb{0,n}(X/S,1)$ is
actually an algebraic space. We will call $L \to X_s$ a \emph{line}, and
$(C/k, p_1,\ldots,p_n \in C(k), h : C \to X_s)$ an \emph{$n$-pointed line}.

\begin{defn}
\label{defn-freeline}
\marpar{defn-freeline}
Notation as above.
An $n$-pointed line is \emph{free} if the morphism $L \to X_s \to X$
factors through the smooth locus of $X/S$ and the pull back of
$T_{X_s}$ via the morphism $L \to X_s$ is globally generated.
\end{defn}

\noindent
We leave it to the reader to show that the locus of free lines
defines an open of $\Kgnb{0,n}(X/S,1)$. For every integer $n$
there is an algebraic space $\FCh_2(X/S,n)$ over $S$ parametrizing
$2$-pointed chains of $n$ free lines in fibres of $f$.
To be precise, $\FCh_2(X/S,n)$ is an open subset of the
$n$-fold fibre product\footnote{This fibre product
will be denoted $\Ch_2(X/C, n)$ later on.}
$$
\Kgnb{0,2}(X/S,1)
\times_{ev_2,X,ev_1}
\Kgnb{0,2}(X/S,1)
\times_{ev_2,X,ev_1}
\dots
\times_{ev_2,X,ev_1}
\Kgnb{0,2}(X/S,1).
$$
The points of $\FCh_2(X/S)$ over a geometric point
$s : \text{Spec}(k) \to S$
correspond to $n$-tuples of $2$-pointed lines
$(C_i/k, p_{i,1}, p_{i,2} \in C_i, h_i : C_i \to X_s)$
such that $h_i(p_{i,2}) = h_{i+1}(p_{i+1,1})$ for
$i=1,\dots,n-1$ and such that each of the $2$-pointed lines
is free. We will sometime denote this by
$(C/k, p,q\in C(k), h: C\to X_s)$. This means that
$C$ is the union of the curves $C_i$ with $p_{i,2}$ identified
with $p_{i+1,1}$, with $h$ equal to $h_i$ on $C_i$ and 
$p = p_{1,1}$ and $q = p_{n+1,2}$. Note that there are $n+1$
natural evaluation morphisms
$$
\text{ev}_i :
\FCh_2(X/S)
\longrightarrow
X,\ \ i=1,\ldots,n+1
$$
over $S$, namely $\text{ev}_i(
[(C_i/k, p_{i,1}, p_{i,2} \in C_i, h_i : C_i \to X_s)]_{i=1,\ldots,n}
) = h_i(p_{i,1})$ for $i=1,\ldots, n$ and $\text{ev}_{n+1}$
on the same object evaluates to give $h_n(p_{n,2})$. In the variant
notation described above we have
$\text{ev}_1((C/k, p,q \in C(k), h : C \to X_s)) = h(p)$ 
and $\text{ev}_{n+1}((C/k, p,q \in C(k), h : C \to X_s)) = h(q)$.
Note that these morphisms map into the \emph{smooth}
locus of the morphism $X \to S$.

\begin{lem}
\label{lem-FChevsmooth}
\marpar{lem-FChevsmooth}
With notation as above.
For every positive integer $n$ and for every integer
$i=1,\dots,n+1$, the evaluation morphism
$
\text{ev}_i : 
\FCh_2(X/S,n) \longrightarrow X
$
is a smooth morphism. In particular $\FCh_2(X/S, n)$ is smooth
over $S$.
\end{lem}

\begin{proof}
We claim the morphism $ev : \Kgnb{0,1}(X/S,1) \to X$
is smooth on the locus of free lines. This is true because, given
a free $1$-pointed line $h : L \to X_s$, $p\in L(k)$ the
corresponding deformation functor has obstruction space
$H^1(L, h^*T_{X_s}(-p))$. This is zero as $h^*T_{X_s}$ is globally
generated by assumption. A reference for the case where $S$ is a point
is \cite[Proposition II.3.5]{K}. Combined with a straightforward
induction argument, this proves that the morphisms $\text{ev}_i$ are smooth.
\end{proof}

\begin{lem}
\label{lem-addline}
\marpar{lem-addline}
Notation as above.
Suppose that $s : \text{Spec}(k) \to S$ is a geometric
point. Suppose that $(C/k, p,q\in C(k), h : C \to X_s)$ is a
free $2$-pointed chain of lines of length $n$.
If $\text{ev}_{1,n+1} : \FCh_2(X/S,n) \to X\times_S X$ is smooth at
the corresponding point then:
\begin{enumerate}
\item $H^1(C, h^*T_{X_s}(-p-q)) = 0$, and
\item for any free line $h_L : L \to X_s$ passing through
$h(q)$ and any point $r \in L$ the morphism
$\text{ev}_{1,n+2} :
\FCh_2(X/S, n+1) \to X\times_S X$ is smooth at the point
corresponding to the chain of free lines of length $n+1$
given by $(C \cup L, p, r \in (C\cup L)(k), h \cup h_L)$.
\end{enumerate}
\end{lem}

\begin{proof}
If the morphism $\text{ev}_{1,n+1}$ is smooth then in particular
the map $h^*T_{X_s} \to h^*T_{X_s}|_p \oplus h^*T_{X_s}|_q$
is surjective on global sections. This easily implies the first
statement. Let $C' = C\cup L$, and $h' = h \cup h_L$. The first
part and the freeness of $L$ imply that $H^1(C', (h')^*T_{X_s}(-p-r))$
is zero. The second part of the lemma follows by deformation
theory from this vanishing.
\end{proof}

\begin{defn}
\label{defn-newpeace}
\marpar{defn-newpeace}
In the situation above. A geometric point $x\in X_s(k)$ lying
over the geometric point $s : \text{Spec}(k) \to S$ is
\emph{$\peaceful$} if both
\begin{enumerate}
\item the evaluation morphism $\text{ev}:\Kgnb{0,1}(X/S,1) \to X$
is smooth at every point in $\text{ev}^{-1}(x)$, and
\item for all $n\gg 0$, the evaluation
morphism $\text{ev}_{1, n+1}:\FCh_2(X/S,n) \rightarrow X\times_S X$
is smooth at some point $(C/k, p,q\in C(k), h : C \to X_s)$
with $h(p) = x$.
\end{enumerate}
\end{defn}

\noindent
The second condition implies that there is at least one free
line through every \peaceful\ point, and hence that a \peaceful\ point
lies in the smooth locus of $X\to S$. The first implies, upon considering
the derivative of $\text{ev}$ that all lines through a \peaceful\ point
are free. Note that there may be situations where there
are no \peaceful\ points whatsoever.

\begin{lem}
\label{lem-openpeace}
\marpar{lem-openpeace}
In the situation above. The set $X_{f,\pax}$ of $\peaceful$
points of $X$ is an open subset of $X$. Furthermore, there
exists a single $n_0$ such that the second condition
of Definition \ref{defn-newpeace} holds for all
peaceful points and all $n \geq n_0$.
\end{lem}

\begin{proof}
Denote by $Z$ the singular set of the morphism
$\text{ev}:\Kgnb{0,1}(X/S,1) \rightarrow X$.
Since $\text{ev}$ is proper, $\text{ev}(Z)$ is a
closed subset of $X$. Thus the complement 
$U' = X \setminus \text{ev}(Z)$ is the maximal open
subset of $X$ such that $\text{ev}: \text{ev}^{-1}(U') \rightarrow U'$
is smooth.
For every $n$, denote by $V_n$ the open subset of
$\FCh_2(X/S,n)$ on which
$$
\text{ev}_{1,n+1}:\FCh_2(X/S,n) \rightarrow X\times_S X
$$
is smooth. The image $V_n' := \text{ev}_1(V_n))$
is open by Lemma \ref{lem-FChevsmooth}. By Lemma \ref{lem-addline}
we have $V_1 \subset V_2 \subset \ldots$. Since $S$ and hence $X$
is Noetherian we find $n_0$ such that $V_{n_0} = V_{n_0 +1} = \ldots$.
It is clear that the set $V' := U' \cap \cup_{n\geq 1} V'_n$ is open
and equals $X_{f,\pax}$.  
\end{proof}

\begin{rmk}
\label{remark-twopoints}
\marpar{remark-twopoints}
In fact, consider the open subset $W_n \subset X\times_SX$ above which
the morphism $\text{ev}_{1,n+1} : \FCh_2(X/S, n) \to X\times_S X$
has some smooth point. This is an increasing sequence by Lemma
\ref{lem-addline} as well and hence equal to $W_{n_1}$ for some
$n_1$.
\end{rmk}

\begin{defn}
\label{defn-peacechain}
\marpar{defn-peacechain}
In the situation above.
Let $U$ be an open subset of $X_{f,\pax}$. Let $n > 0$.
A \emph{$U$-adapted free chain of $n$ lines} is a
point $[h] \in \FCh_s(X/S, n)$ such that $\text{ev}_i([h])
\in U$ for all $i=1,\ldots,n+1$. When $U$ is all of $X_{f,\pax}$,
this is sometimes called a \emph{peaceful chain}.
The subset of $\FCh_2(X/S,n)$ parametrizing $U$-adapted
free chains is denoted $\FCh^U_2(X/S,n)$. When $U$ is all
of $X_{f,\pax}$ we sometimes use the notation
$\PCh_2(X/S,n)$.
\end{defn}

\noindent
Without more hypotheses on our morphism $f : X \to S$ and
relatively ample invertible sheaf $\mc{L}$ we cannot say
much more (since after all the set of lines in fibres might
be empty).

\begin{hyp}
\label{hyp-peace}
\marpar{hyp-peace}
Let $S$ be an excellent scheme of characteristic $0$.
Let $f : X \to S$ be a proper flat morphism with geometrically
irreducible fibres. Let $\mc{L}$ be $f$-ample.
In addition we assume the following.
\begin{enumerate}
\item There exists an open subset $U$ of $X$ surjecting to $S$
and such that the geometric fibres of the 
evaluation morphism
$$
\text{ev}:\Kgnb{0,1}(X/S,1) \rightarrow X
$$
over $U$ are nonempty, irreducible and rationally connected.
\item There exists a positive integer $m_0$ and an open subset $V$ of
$X\times_S X$ surjecting to $S$ and such that the geometric fibres of
the evaluation morphism
$$
\text{ev}_{1, n+1} :
\FCh_2(X/S,m_0)
\longrightarrow
X\times_S X
$$
over $V$ are nonempty, irreducible and birationally rationally connected.
\end{enumerate}
\end{hyp}

\noindent
Note that the second condition in particular implies that the smooth
locus of $X \to S$ intersects every fibre of $X\to S$. Namely, the open
$V$ is contained in $\text{Smooth}(X/S) \times_S \text{Smooth}(X/S)$ by
the definition of free lines. In addition, the second condition also
shows that the smooth locus of each geometric fibre of $X\to S$ is
rationally chain connected by free rational curves, and hence
rationally connected.

\begin{lem}
\label{lem-peacefuldense}
\marpar{lem-peacefuldense}
In Situation \ref{hyp-peace} the \peaceful\ points are
dense in every smooth fibre of $X \to S$.
\end{lem}

\begin{proof}
It is well known that lines passing through a general point of
a smooth projective variety are free. Hence the open $U'$
of the proof of Lemma \ref{lem-openpeace} has a nonempty
intersection with every smooth fibre of $f$. On the other
hand, condition 2 of \ref{hyp-peace} implies that for every
point $s$ of $S$ the morphism
$\FCh_2(X/S, m_0)_s \to (X\times_S X)_s$ is a dominant
morphism of smooth varieties. Hence it is smooth at some
point $\xi \in \FCh_2(X/S, m_0)_s$ by generic smoothness.
The ``crit\`ere de platitude par fibre'' \cite[Theorem 11.3.10]{EGA4}
implies that $\text{ev}_{1,n+1}$ is flat at $\xi$ and hence smooth
as desired. Thus the open $V'_{m_0}$ of the proof
of lemma \ref{lem-openpeace} has nonempty fibre $V'_{m_0,s}$.
We win because $X_{f,\pax} \supset U' \cap V'_{m_0}$.
\end{proof}

\begin{lem}
\label{lem-peacechain}
\marpar{lem-peacechain}
In Situation \ref{hyp-peace}.
If $U$ is a nonempty open subset of $X_{f,\pax}$, then
$\FCh^U_2(X/S,n)$ is an open dense subset of $\FCh_2(X/S,n)$.
In fact, it is dense in the fibre of $\FCh_2(X/S,n) \to S$
over any point $s$ for which $U_s \not = \emptyset$.
For every integer $i=1,\dots,n+1$, each morphism
$$
\text{ev}_{i} :
\FCh^U_2(X/S,n)
\longrightarrow
U,
$$
is a smooth morphism and every geometric fibre is nonempty, irreducible and
birationally rationally connected.
\end{lem}

\begin{proof}
By Lemma~\ref{lem-FChevsmooth}, the morphisms $\text{ev}_i$
are smooth.  Thus, the preimages $\text{ev}_{i}^{-1}(U)$
and are dense open subsets. Thus the finite intersection 
$ \FCh^U_2(X/S,n) = \cap_{i=1}^{n+1} ( \text{ev}_{i}^{-1}(U) )$
is also a dense open subset. Same argument for fibres.

\mni
Next let $x$ be a geometric point of $U$. The fibre of
$$
\text{ev}_{1} :
\FCh^U_2(X/S,n)
\longrightarrow
U
$$
over $x$ is the total space of a tower of birationally rationally
connected fibrations. Indeed, the variety parametrizing choices
for the first line $L_1$ of the chain containing $x = p_{1,1}$ is the fibre
of
$$
\text{ev}:\Kgnb{0,1}(X/S,1) \rightarrow X
$$
over $x$. Since $\text{ev}$ is a proper morphism and since
the geometric generic fibre of $\text{ev}$ is rationally connected,
every fibre in the smooth locus of $\text{ev}$ is rationally connected
by \cite[2.4]{KMM} and \cite[Theorem IV.3.11]{K}. In particular, the
fibre over every $\peaceful$ point $x$ is in the smooth locus and thus
is rationally connected. Next, given $L_1$, the variety parametrizing
choices for the attachment point $p_{1,2}$ is the intersection of
$L_1$ with $U$, which is birationally rationally connected since it
is open in $L_1 \cong \PP^1$. Next, the variety parametrizing
lines $L_2$ containing $p_{1,2} = p_{2,1}$ is rationally connected
for the same reason that the variety parametrizing lines $L_1$ is rationally
connected, etc. By \cite{GHS}, the total space of a tower of
birationally rationally connected fibrations is itself birationally
rationally connected. Thus the geometric fibres of $\text{ev}_1$ are
birationally rationally connected.  

\mni
The proof for the other morphisms $\text{ev}_i$ is similar.  
\end{proof}

\begin{prop}
\label{prop-peace}
\marpar{prop-peace}
In Situation \ref{hyp-peace}.
There exists a positive integer $m_1$ such that for every integer
$n\geq m_1$ and for every nonempty open subset $U$ of $X_{f,\pax}$ there
exists an open subset $U_n$ of $\FCh^U_2(X/S,n)$ whose intersection with
every geometric fibre of
$$
\text{ev}_{1,n+1} :
\FCh^U_2(X/S,n)
\longrightarrow
U\times_S U
$$
is nonempty, smooth, irreducible and birationally rationally connected.
\end{prop}

\begin{proof}
Let $n_0$ be the integer of Lemma \ref{lem-openpeace}.
Let $m_0$ be the integer of Hypothesis \ref{hyp-peace}.
We claim that $m_1 = 2n_0 + m_0$ works.

\mni
Namely, suppose that $n \geq m_1$. Write $n = a + m_0 + b$,
with $a \geq n_0$ and $b \geq n_0$. Consider the opens
$$
A = \text{ev}_{a+1}(\FCh_2^U(X/S, a)) \subset X,
\text{ and }
B = \text{ev}_{1}(\FCh_2^U(X/S, b)) \subset X.
$$
These have nonempty fibre over every point of $S$ over which the
fibre of $U \to S$ is nonempty by Lemma \ref{lem-peacechain}.
For $B$ you technically speaking also have to use that there
is an automorphism on the space of chains which switches the
start with the end of a chain. As a first approximation we let
$$
U_n
=
\text{ev}_{a+1}^{-1}(A) \cap \text{ev}_{a+1+m_0}^{-1}(B)
\cap \text{ev}_{a+1, a+1+m_0}^{-1}(V)
\subset
\FCh_2^U(X/S, n).
$$
Here $V \subset X\times_SX$ is the open of Hypothesis \ref{hyp-peace}.
Let us describe the geometric points of $U_n$.
Let $s : \text{Spec}(k) \to S$ be a geometric point of $S$ such
that $U_s \not = \emptyset$. A point $\xi \in U_n(k)$ is given by
the following data:
\begin{enumerate}
\item a point $x \in A_s(k)$, a point $y \in B_s(k)$, such that
$(x,y) \in V_s(k)$,
\item a $U$-adapted free chain of $a$ lines $\xi_a$ with end point at $x$,
\item a $U$-adapted free chain of $b$ lines $\xi_b$ with start point at $y$, and
\item a $U$-adapted $2$-pointed free chain $\xi_m$ of $m_0$ lines connecting
$x$ to $y$.
\end{enumerate}
The morphism $U_n \to U\times_S U$ maps $\xi$ as above to the pair
$(x', x'') \in U_s(k)^2$ consisting of the start point $x'$ of $\xi_a$
and the end point $x''$ of $\xi_b$. According to Lemma \ref{lem-peacechain}
the set of choices of $\xi_a$ and $\xi_b$, given $x'$ and $x''$, is a
nonempty smooth
irreducible, birationally rationally connected variety. For every pair
$(\xi_a, \xi_b)$ we get a pair of points $(x,y)$ as above and since
$V_s \not=\emptyset$ it is a nonempty open condition to have $(x,y)
\in V_s(k)$. By the
Hypothesis \ref{hyp-peace} the set of choices of $\xi_m$ connecting
$x$ with $y$ forms an irreducible, birationally rationally connected
variety. Using \cite{GHS}, it follows that the fibre of $U_n$
over $(x',x'')$ is smooth and birationally rationally connected.

\mni
We are not yet assured that $\text{ev}_{1,n+1} : U_n \to U\times_S U$
is smooth. Let $U_n' \subset U_n$ be the smooth locus of this morphism.
We claim that $U_n' \subset U_n$ is dense in every geometric fibre
$\text{ev}_{1,n+1}^{-1}((x', x''))$ we studied above, which will finish
the proof of the proposition. By irreducibility
of this fibre it suffices to show $U_n'$ intersects the fibre. Since
$a \geq n_0$ we know by Lemma \ref{lem-openpeace} that there is a
chain $\xi_a$ as above which represents a point where $\text{ev}_{1, a+1}$
is smooth. Similarly for $\xi_b$. By openness of smoothness, we may
assume the associated pair $(x,y)$ is an element of $V_s(k)$.
Pick any free chain $\xi_m$ connecting $x$ and $y$ as above.
We claim that the resulting chain of length $a+m_0+b$ represents
a smooth point for the morphism $\text{ev}_{1,n+1}$. The last part
of the proof is similar to the proof of Lemma \ref{lem-addline}
and uses the results of that lemma. Namely,
$\xi_m$ corresponds to a chain of lines $L_1,\ldots, L_n$ and
points $p_{i,j}$, $i\in \{1,\ldots,n\}$, $j\in \{1,2\}$, such
that (where we simply write $T_{X_s}$ to denote the pullback)
$H^1(L_1\cup\ldots\cup L_a, T_{X_s}(- p_{1,1} - p_{a,2})) = 0$,
$H^1(L_{a+m_0+1}\cup\ldots\cup L_n, T_{X_s}(- p_{a+m_0+1,1} - p_{n,2})) = 0$,
and the lines $L_{a+1}, \ldots, L_{a+m_0}$ are free. It is easy to
see that this implies that $H^1(L_1\cup\ldots\cup L_n, T_{X_s}(-p_{1,1}
-p_{n,2}))=0$. This finishes the proof as in Lemma \ref{lem-addline}.
\end{proof}

%%%%%%%%%%%%%%%%%%%%%%%%%%%%%%%%%%%%%%%%%%%%%%%%%%%%%%%%%%%%%%%
%%
%% Porcupines
%% 
%%%%%%%%%%%%%%%%%%%%%%%%%%%%%%%%%%%%%%%%%%%%%%%%%%%%%%%%%%%%%%%

\section{Porcupines}
\label{sec-porcupines}
\marpar{sec-porcupines}

\begin{hyp}
\label{hyp-defn-porc}
\marpar{hyp-defn-porc}
Here $k$ is an uncountable
algebraically closed field of characteristic $0$.
Let $C$ be a smooth, irreducible, proper curve over $k$.
Let $X$ be smooth and proper over $k$. 
Let $f:X\rightarrow C$ be proper and flat with geometrically
irreducible fibres.
The invertible $\OO_X$-module $\mc{L}$ is $f$-ample.
\end{hyp}

\noindent
We are going to apply the material of the previous section to
this situation. In particular, let $X_{f,\pax}$ be the set
of peaceful points relative to $f : X \to C$, see
Lemma \ref{lem-openpeace}. Here is one more type of stable map
that arises often in what follows.

\begin{defn}
\label{defn-ppine}
\marpar{defn-ppine}
Notation as above.
Let $e$ be an integer and let $n$ be a nonnegative integer.
A degree $e$ \emph{porcupine with $n$ quills} over $k$ is
given by the following data
\begin{enumerate}
\item a section $s : C \to X$ of degree $e$ with respect to $\mc{L}$,
\item $n$ distinct points $q_1,\ldots,q_n \in C(k)$, and
\item $n$ $1$-pointed lines $(r_i \in L_i(k), L_i \to X_{q_i})$.
\end{enumerate}
These data have to satisfy the requirements that
(a) the set $s^{-1}(X_{f, \pax}) \not= \emptyset$ and
each $q_i$ is in this open,
(b) the points $r_i$ and $q_i$ get mapped to the same
$k$ point of $X$, and (c) the section $s$ is free
(see Definition \ref{defn-classicalfree}).
\end{defn}

\noindent
It may be useful to discuss this a bit more.
The first remark is that by our definitions the image
of $s$ lies in the smooth locus of $X / C$ and
all the lines $L_i$ are free and lie in the smooth
locus as well. In other words, the union $C \cup (\bigcup L_i) \to X$
is going to be a comb (see Definition \ref{defn-ndcomb})
in the smooth locus. We will call the handle $s : C \to X$
the \emph{body} of the porcupine, and we will call the
teeth $L_i \to X$ the \emph{quills}.

\mni
\emph{Families} of degree $e$ porcupines with $n$ quills are defined
in the obvious manner; for example we may observe that the space of
porcupines of degree $e$ and $n$ quills is an open substack of the 
moduli stack $\Co^{\ul{e}}_{\ul{I}}(X/C/k)$ with $\ul{I} =
(\emptyset, \emptyset, \ldots, \emptyset)$, and $\ul{e} = 
(e,1,\ldots,1)$ of combs studied in Section \ref{sec-ssections}.
The parameter space for porcupines is denoted $\Pine{e,n}(X/C/k)$.

\mni
Denote by
$$
\begin{matrix}
\Phi_{\text{total}} :
&
\Pine{e,n}(X/C/\kappa)
&
\longrightarrow
&
\Sigma^{e+n}(X/C/\kappa),
\\
\Phi_{\text{body}} :
&
\Pine{e,n}(X/C/\kappa)
&
\longrightarrow
&
\text{Sections}^e_n(X/C/k),
\\
\Phi_{\text{quill},i} :
&
\Pine{e,n}(X/C/\kappa)
&
\longrightarrow
&
\Kgnb{0,1}(X/C,1),
\\
\Phi_{\text{forget}} :
&
\Pine{e,n+m}(X/C/\kappa)
&
\longrightarrow
&
\Pine{e,n}(X/C/\kappa)
\end{matrix}
$$
the obvious forgetful morphisms. Note that
$\Phi_{\text{forget}}$ is smooth with irreducible fibres.

\begin{lem}
\label{lem-existDfree}
\marpar{lem-existDfree}
In Situation \ref{hyp-defn-porc}. Let $n \geq 1$.
If the geometric generic fibre of $f$ is rationally connected
then there exists a $n$-free section of $f$. If the set of 
\peaceful\ points is nonempty then we may assume the
section meets this locus.
\end{lem}

\begin{proof}
First of all, by \cite{GHS} there exists a section $\tau_0$ of $f$.
Because $X$ is smooth, the image of the section is necessarily contained in
the smooth locus of the morphism. The section $\tau_0$
is not necessarily $n$-free.
By Lemma \ref{lem-newfree3} it suffices to construct a section
that is $D$-free for some divisor $D$ of degree $2g(C)+n$.
Set $m = 2g(C) + n$. Consider the subsets $W^m_e$ of
$H^m_{X,f}$ introduced in Proposition \ref{prop-KMM}.
We know there exist nonempty opens $U_e^m \subset H^m_{X,f}$,
such that $U_e^m \subset W^m_e$.
Also by the proposition, a section $\tau : C \to X_{f,smooth}$
will be $D$-free for some effective divisor $D \subset C$
of degree $m$ provided that the restriction of
$\tau$ to $D = \sum t_i$ is in $\bigcap_{e} U_e^m$.
By Lemma \ref{lem-avoid}
(applied with $T = \text{Spec}(k)$ and $U=Y=C$ and $V = X_{f,smooth}$)
we can find a section $\tau$ with the desired properties.
\end{proof}

\noindent
Before we state the next lemma, we wish to remind the reader
that an irreducible scheme is necessarily nonempty. Also, see
Remark \ref{rmk-irredAbel} below for a more ``lightweight''
variant of some of the following lemmas.

\begin{lem}
\label{lem-Zd}
\marpar{lem-Zd}
In Situation \ref{hyp-defn-porc}, assume in addition
Hypothesis \ref{hyp-peace} for the restriction of $f$
to some nonempty open $S \subset C$. Let $e_0$ be an
integer and let $Z$ be an irreducible component
of $\Sigma^{e_0}(X/C/k)$ whose general point
parametrizes a free section of $f$.  
For every integer $e\geq e_0$ there exists a unique irreducible
component $Z_e$ of $\Sigma^e(X/C/k)$ such that every porcupine
with body in $Z$ and with $e-e_0$ quills is parametrized by $Z_e$.  
\end{lem}

\begin{proof}
Take any point $[s]$ of $Z$ corresponding to a section $s$ which
is free.
By Proposition ~\ref{prop-newfree4}, the space $\Sigma^{e_0}(X/C/k)$
is smooth at the point $[s]$. Thus
this point is contained in a unique irreducible component of
$\Sigma^{e_0}(X/C/k)$, namely $Z$. By Lemma \ref{prop-newfree4}
free sections deform to contain a general point of $X$.
The set of \peaceful\ points is nonempty by Lemma \ref{lem-peacefuldense}.
Hence there are also points $[s]$ of $Z$ which correspond to
free sections $s$ which intersect the \peaceful\ locus.  
For every such section $s$, the inverse image
$s^{-1}(X_{f,\pax})$ is a dense open subset of $C$.  Thus
the variety parametrizing $(e-e_0)$-tuples of closed points
$(q_1,\dots,q_{e-e_0})$ in
$s^{-1}(X_{f,\pax})$ is smooth and irreducible.  By
definition of \peaceful\ the fibre of
$$
\text{ev}:\Kgnb{0,1}(X/C,1) \rightarrow X
$$
over $s(q_i)$ is smooth and by Lemma \ref{lem-peacechain} it is
irreducible. Putting the pieces together,
the variety $W_{Z,e}$ parametrizing degree $e$ porcupines with body
in $Z$ is smooth and irreducible.  Moreover, because the body is unobstructed
and because the teeth are all free, each such porcupine is unobstructed.
Thus $\Sigma^e(X/C/k)$ is smooth at every point of $W_{Z,e}$. Thus $W_{Z,e}$
is contained in a unique irreducible component $Z_e$ of $\Sigma^e(X/C/k)$.
\end{proof}

\begin{lem}
\label{lem-Zdgenpt}
\marpar{lem-Zdgenpt}
Notation and assumptions as in Lemma \ref{lem-Zd}.
Let $e \geq e_0$. A general point of $Z_e$ corresponds
to a section of $X \to C$ which is free and meets $X_{f,\pax}$.
\end{lem}

\begin{proof}
Another way to state this lemma is that (1) we can smooth
porcupines, and (2) a smoothing of a porcupine is
a porcupine. Consider a porcupine $(s, r_i \in L_i(k), L_i \to X)$.
Let $h : C \cup \bigcup L_i \to X$ be the corresponding stable map.
We noted in the proof of Lemma \ref{lem-Zd} above that the map
$h$ is unobstructed. Hence any smoothing of 
the comb $C \cup \bigcup L_i \to C$ can be followed by
a deformation of $h$. We can realize $C \cup \bigcup L_i$
as the special fibre of a family of curves over $C$ whose
general fibre is $C$, simply by blowing up $C \times \PP^1$
in suitable points of the fibre $C \times \{0\}$.
Thus we can smooth the stable map $h$.

\medskip\noindent
Consider a general smoothing $s' : C \to X$ of $h$.
The statement on the intersection with $X_{f,\pax}$ is
okay since this is clearly an open condition on all of $Z_e$.
To see that the smoothing $s'$ of $(s, r_i \in L_i(k), L_i \to X)$
is free we will use Lemma \ref{lem-semicty}.
It suffices to show for any effective Cartier divisor
$D \subset C$ of degree $1$ that $H^1(C \bigcup L_i, h^*T_f(-D)) = 0$.
And this is immediate from the assumption that
$H^1(C, s^*T_f(-D)) = 0$ and $T_f|_{L_i}$ is globally generated.
\end{proof}

\begin{lem}
\label{lem-Zdup}
\marpar{lem-Zdup}
Notation and assumptions as in Lemma \ref{lem-Zd}.
For any $e' \geq e \geq e_0$, any porcupine of
total degree $e'$ with body in $Z_e$ and $e' - e$ quills
corresponds to a point of $Z_{e'}$.
\end{lem}

\begin{proof}
Let $[s] \in Z_e$ be the moduli point corresponding to
a section of $X \to C$ which meets $X_{f,\pax}$ and
is free.
Because $Z_e$ is irreducible, and because
of the construction of $Z_e$ we can find an irreducible
curve $T$, a morphism $h : T \to Z_e$, and points $0,1 \in T(k)$
such that $h(0) = [s]$ and $h(1)$ corresponds to a
point of $W_{Z,e}$ (see proof of Lemma \ref{lem-Zd}).
After deleting finitely many points $\not = 0, 1$
from $T$ we may assume every point $t \in T(k)$
$t \not= 1$ corresponds to a section $s_t$ which is
free (see Lemma \ref{lem-newfreeopen}) and meets $X_{f,\pax}$.
Thus every point $h(t) \in Z_e$ correponds to a porcupine.
Consider the scheme $T' \to T$ parametrizing
choices of $e' - e$ quills attached to the body of
the porcupine $h(t)$, $t\in T(k)$ in pairwise distinct points
of $X_{f,\pax}$.
As in the proof of Lemma \ref{lem-Zd} there is an irreducible
parameter space of choices, in other words $T'$ is irreducible.
Also $T' \subset \Sigma^{e'}(X/C/k)$ lies in the smooth locus,
because each point of $T'$ corresponds to a porcupine.
The result follows because $T'_1 \subset Z_{e'}$ by definition
of $Z_{e'}$ and $T'_0$ parametrizes porcupines with body $s$.
\end{proof}

\begin{lem}
\label{lem-Zdchains}
\marpar{lem-Zdchains}
Notation and assumptions as in Lemma \ref{lem-Zd}.
Let $e \geq e_0$. Let $[s] \in Z_e(k)$ be a point
corresponding to a $0$-free section $s : C \to X$. Let 
$(h_i : C_i \to X, p_i \in C_i(k))$ be a finite number of
$1$-pointed chains of free lines in fibres of $f$.
Assume $h_i(p_i) \in s(C)$ and
assume $h_i(p_i)$ pairwise distinct. Then
$s(C) \cup \bigcup C_i$ defines an unobstructed
nonstacky point of $\Sigma^{e'}(X/C/k)$ which lies in
$Z_{e'}$ with $e' = e + \sum \deg(C_i)$.
\end{lem}

\begin{proof}
There are no automorphisms of the stable curve
$s(C) \cup \bigcup C_i$, see our discussion of
lines in Section \ref{sec-peace}.
Since $s$ is $0$-free and each line in each chain
$C_i$ is free, there are no obstructions to deforming
$s(C) \cup \bigcup C_i$, hence $\Sigma^{e'}(X/C/k)$
is smooth at the corresponding point.
Because the evaluation morphisms $\FCh_1(X/C, n) \to X$
are smooth (Lemma \ref{lem-FChevsmooth}),
we may replace $s$ by any deformation of $s$.
By Lemma \ref{lem-Zdgenpt} we may assume
$s^{-1}(X_{f,\pax})$ is nonempty and $s$ is
free. By the smoothness of $\FCh_1(X/C, n) \to X$
again we may deform the $1$-pointed chains $(C_i \to X, p_i)$
and assume $p_i \in s(C) \cap X_{f,\pax}$.

\medskip\noindent
The lemma follows from Lemma \ref{lem-Zdup} by induction on
the length of the chains of free lines attached to the
body. Namely, let $L_i \subset C_i$ be the first line
of the chain, i.e., the line that contains $p_i$.
Let $C_i'$ be the rest of the chain, and let
$p_i' \in C_i'$ be the attachment point (where
it is attached to $L_i$).
Any deformation of $s(C) \cup \bigcup_{i=1}^m L_i$ to a section
$s' : C \to X$ lies in $Z_{e_0 + m}$ by construction.
By the smoothness of $\FCh_1(X/C, n) \to X$
again we may deform the $1$-pointed chains $(C_i' \to X, p_i')$
along with the given deformation to obtain $p_i' \in s'(C)$.
Repeat.
\end{proof}

\begin{lem}
\label{lem-Zdbetter}
\marpar{lem-Zdbetter}
Notation and assumptions as in Lemma \ref{lem-Zd}.
For any $m \geq 0$ there exist $e \gg 0$ such that
$Z_e$ contains a point corresponding to a section
which is $m$-free and meets $X_{f,\pax}$.
\end{lem}

\begin{proof}
Let $[s] \in Z$ be the moduli point corresponding
to a section $s : C \to X$ which is free
and meets $X_{f,\pax}$. Let $N = m + 2g(C) + 1$.
Consider pairwise distinct points $q_1,\ldots, q_N \in C(k)$
such that $s(q_i) \in X_{f,\pax}$.
Choose 2-pointed free chains of lines
$(h_i : C_i \to X, p_i, q_i \in C_i(k))$ of some common length $n$
such that $s(q_i) = h(p_i)$ and $z_i = h(q_i)$ are
general points $z_i \in X_{q_i}(k)$.
We may choose these by the smoothness assumption in
Definition \ref{defn-newpeace}.
By Lemma \ref{lem-peacechain} we may in fact assume
all attachment points in the chains of lines
are \peaceful. We may also assume all the
attachment points are distinct, since that condition defines
an open dense subset of each fibre
of $\text{ev}_1 : \FCh_2(X/S,n) \to X$.
By Lemma \ref{lem-Zdchains} any
smoothing of the stable map $s(C) \cup \bigcup C_i \to X$
defines a point of $Z_e$ (for a suitable $e$).
The smoothing will still pass through $r$ general
points of the fibres $X_{q_i}$. Whence applying Proposition
\ref{prop-KMM} we get that the resulting section is $D$-free
for some divisor of degree $N$. By Lemma \ref{lem-newfree3}
we win.
\end{proof}

\begin{lem}
\label{lem-Abel-irred}
\marpar{lem-Abel-irred}
Notation an assumptions as in Lemma \ref{lem-Zd}.
For every $e\geq e_0 + g(C)$, the Abel map
$$
\alpha_{\mc{L}}|_{Z_e} :
Z_e
\longrightarrow
\underline{\text{Pic}}^e_{C/k}
$$
is dominant with irreducible geometric generic fibre.
\end{lem}

\begin{proof}
Let $n$ be an integer with $n\geq g(C)$ and let $e=e_0+n$.
Let $W_{Z,e}$ be as in the proof of Lemma~\ref{lem-Zd}.
As in the proof of Lemma \ref{lem-Zd}, every point of $W_{Z,e}$
is a smooth point of $Z_e$. Therefore to prove that
$$
\alpha_{\mc{L}}|_{Z_e} :
Z_e
\longrightarrow
\underline{\text{Pic}}^e_{C/k}
$$
is dominant with irreducible geometric generic fibre, it suffices to
prove that the restriction
$$
\alpha_{\mc{L}}|_{W_{Z,e}} :
W_{Z,e}
\longrightarrow
\underline{\text{Pic}}^e_{C/k}
$$
is dominant with irreducible geometric generic fibre.  And by
Lemma \ref{lem-Abel}, the morphism $\alpha_{\mc{L}}|_{W_{Z,e}}$
factors as follows
$$
W_{Z,e}
\to
Z \times C^{e-e_0}
\to
Z\times \text{Pic}^{e-e_0}_{C/k}
\to
\text{Pic}^{e_0}_{C/k} \times \text{Pic}^{e-e_0}_{C/k}
\to
\text{Pic}^e_{C/k}
$$
where each of the component maps is the obvious morphism.
The first map has smooth, geometrically irreducible fibres as
was explained in the proof of the Lemma \ref{lem-Zd} above.
And since $e-e_0$ is $\geq g(C)$, the second is dominant with
irreducible geometric generic fibre.

\mni
The trick is to factor the composition $\beta$ of the last two maps.
Namely, it is the composition of the map
$$
(\text{Id}_Z,\beta):Z \times_k \text{Pic}^{e-e_0}_{C/k} \rightarrow
Z \times_k \text{Pic}^e_{C/k}, \ \ (s,[D]) \mapsto
(s,\alpha_{\mc{L},Z}(s)+[D])
$$
and the projection
$$
\text{pr}_2: Z \times_k \text{Pic}^e_{C/k} \rightarrow
\text{Pic}^e_{C/k}.
$$
Since $(s,[D]) \mapsto (s,-\alpha_{\mc{L},Z}(s)+[D])$ is an inverse of
$(\text{Id}_Z, \beta)$, we see $(\text{Id}_Z, \beta)$ is an isomorphism
of schemes.  Moreover,
considering the source as a scheme over $\text{Pic}^e_{C/k}$ via
$\beta$ and considering the target as a scheme over
$\text{Pic}^e_{C/k}$ via $\text{pr}_2$, it is an isomorphism of
schemes over $\text{Pic}^e_{C/k}$.  Therefore every geometric fibre of
$\beta$ is isomorphic to $Z$, which is irreducible by hypothesis.  So
$\beta$ is dominant with irreducible geometric generic fibre.  Since a
composition of dominant morphisms with irreducible geometric generic
fibres is of the same sort, $\alpha_{\mc{L}}|_{W_{Z,e}}$, and thus
$\alpha_{\mc{L}}|_{Z_e}$, are dominant with irreducible geometric
generic fibres.
\end{proof}

\begin{rmk}
\label{rmk-irredAbel}
\marpar{rmk-irredAbel}
The proof of Lemmas \ref{lem-Zd} and \ref{lem-Abel-irred}
works in a more general setting. Namely, suppose that
$k$, $X \to C$ and $\mc{L}$ are as in \ref{hyp-defn-porc}.
Redefine a porcupine temporarily by replacing condition (a)
with the condition: all lines on $X_{q_i}$ through $s(q_i)$
are free and they form an irreducible variety. Assume that
for a general $c\in C(k)$ and a general $x\in X_c(k)$ the
space of lines through $x$ is nonempty and irreducible.
Then the conclusions of Lemmas \ref{lem-Zd} and \ref{lem-Abel-irred}
hold.
\end{rmk}

%%%%%%%%%%%%%%%%%%%%%%%%%%%%%%%%%%%%%%%%%%%%%%%%%%%%%%%%%%%%%%%
%%
%% Pencils of Porcupines
%% 
%%%%%%%%%%%%%%%%%%%%%%%%%%%%%%%%%%%%%%%%%%%%%%%%%%%%%%%%%%%%%%%

\section{Pencils of porcupines}
\label{sec-pencils}
\marpar{sec-pencils}

\noindent
Let $k$, $f : X \to C$ and $\mc{L}$ be as in Hypothesis
\ref{hyp-defn-porc}. A \emph{ruled surface} or a \emph{scroll}
in $X$ will be a morphism $R \to X$ such that $R \to C$ is proper smooth
and all fibres are lines in $X$. The following will be used
over and over again.

\begin{lem}
\label{lem-ruledsmooth}
\marpar{lem-ruledsmooth}
In Situation \ref{hyp-defn-porc} any ruled surface
$R \to X$ lies in the smooth locus of $X \to C$.
\end{lem}

\begin{proof}
Since by assumption the total space of $X$ is smooth
any section of $X \to C$ lies in the smooth locus
of $X \to C$. On the other hand, weak approximation
(see for example \cite{HT06})
holds for the ruled surface $R \to C$, a fortiori
for any point $r \in R(k)$ we can find a section
of $R \to C$ passing through $r$. The lemma follows.
\end{proof}

\noindent
Let $h : C' \to X$ be a porcupine, by which we mean
$C' = C \cup \bigcup L_i \to X$ as in Definition
\ref{defn-ppine}.

\begin{defn}
\label{defn-pen}
\marpar{defn-pen}
Notations as above. A \emph{pen} for $C'$ is a ruled surface $R$
such that $C' \to X$ factors through $R \to X$.
\end{defn}

\noindent
We will always assume a pen for $C'$ comes with a factorization.
If lines in fibres of $X \to C$ are automatically smooth
(e.g.\ if $\mc{L}$ is very ample on the fibres of $X \to C$)
then $R$ and $C'$ are
closed subschemes of $X$ the definition just means $C' \subset R$.

\begin{lem}
\label{lem-Abelpen}
\marpar{lem-Abelpen}
Let $C'_0$ and $C'_\infty$ be porcupines whose
bodies $s_0(C)$ and $s_\infty(C)$ are penned in a common
ruled surface $R$. Let $(e_0, n_0)$, resp.\ $(e_\infty, n_\infty)$
be the numerical invariants of $C'_0$, resp.\ $C'_{\infty}$
and assume $e_0 + n_0 = e_\infty + n_\infty$.
Denote the attachment points of $C'_0$,
resp.\ of $C'_\infty$, by $q_{0,i}$, resp.\  $q_{\infty, j}$.
The Abel images $\alpha_{\mc{L}}([C'_0])$ and
$\alpha_{\mc{L}}([C'_\infty])$ are equal if and only if
$D_0 \sim D_\infty$ in the divisor class group of $R$
where 
$$
D_0 := s_0(C) + \sum\nolimits_{i=1}^{n_0} R_{q_{0,i}}
\text{ and }
D_\infty := s_\infty(C) + \sum\nolimits_{j=1}^{n_\infty} R_{q_{\infty,j}}
$$
\end{lem}

\begin{proof}
Note that $\text{Pic}(R)$ is isomorphic to $\ZZ \oplus \text{Pic}(C)$.
As generator for the summand $\ZZ$ we take the class of $\mc{L}$. The
divisors $R_q$ on $R$ correspond to the elements $\OO_C(q)$ in
$\text{Pic}(C)$. Let $K_{rel}$ denote the relative canonical
divisor of $R \to C$. The divisor $s_0(C)$ corresponds to the
element $-K_R + (-1, s_0^*\mc{L})$, by a fun calculation left to
the reader. Hence $D_0 \sim -K_r + (-1, s_0^*\mc{L}(\sum q_{0,i}))$
Similar for $s_\infty$. The result now follows from Lemma ~\ref{lem-Abel}.
\end{proof}

\noindent
In the above lemma, if $D_0$ and $D_\infty$ are linearly equivalent,
denote by $(D_\lambda)_{\lambda\in \Pi}$ the pencil of effective
Cartier divisors on $R$ spanned by $D_0$ and $D_\infty$.
At this point we start constructing families of porcupines over
rational curves. In order to do this we state and prove a few
lemmas.

\begin{lem}
\label{lem-pivot}
\marpar{lem-pivot}
In situation \ref{hyp-defn-porc} assume the restriction
of $f$ to some nonempty open $S \subset C$ satisfies
Hypothesis \ref{hyp-peace}. Suppose that $C'$ and
$C''$ are porcupines whose bodies and attachment points
agree, but which may have different quills. Then there
exists a rational curve in the smooth nonstacky locus of
$\Sigma^{e+n}(X/C/k)$ connecting the corresponding points.
\end{lem}

\begin{proof}
A porcupine always represents a smooth nonstacky point of
$\Sigma^{e+n}(X/C/k)$. The space parametrizing choice of
quills given the body $s$ and the attachment points
$q_1,\ldots, q_n$ is the product of $n$ fibres of the evaluation map
$\Kgnb{0,1}(X/C,1) \to X$ at peaceful points. But these fibres
are smooth projective rationally connected varieties by
definition of \peaceful\ and Lemma \ref{lem-peacechain}.
\end{proof}

\noindent
In the following lemma and below we will say that the
porcupine $C''$ is an \emph{extension} of $C'$ if you get $C'$
from $C''$ by deleting some of its quills. In other words,
this means that $\Phi_{\text{forget}}(C'') = C'$, where
$\Phi_{\text{forget}}$ is the forgetful morphism
defined in Section \ref{sec-porcupines}.

\begin{lem}
\label{lem-pentogether}
\marpar{lem-pentogether}
Same assumptions as in Lemma \ref{lem-pivot}.
Let $s_0, s_\infty : C \to X$ be sections.
Assume $s_0$ and $s_\infty$ are free,
and that $V = s_0^{-1}(X_{f,\pax}) \cap s_\infty^{-1}(X_{f,\pax})$
is not empty. Assume $s_0$ and $s_\infty$ are penned in a common ruled
surface $R$. Then there exists a nonnegative
integer $E$ with the following property:
for all $e$ with
$e \geq \min\{E, \deg(s_0), \deg(s_\infty)\}$,
for any pairwise distinct points
$q_{0,1},\dots,q_{0, e - \deg(s_0)} \in C(k)$,
for any pairwise distinct points
$q_{\infty, 0}, \ldots, q_{\infty, e - \deg(s_\infty)} \in C(k)$
and for every extension $C'_0$, resp.\ $C'_\infty$, of
$s_0$, resp.\ $s_\infty$, to a porcupine with quills
over the points $q_{0,1},\dots, q_{0, e - \deg(e_0)}$,
resp.\ $q_{\infty, 0},\ldots,q_{\infty, e - \deg(e_\infty)}$
if $s_0^*(\mc{L})(\sum q_{0,i})
\cong s_\infty^*(\mc{L})(\sum q_{\infty,i})$ as invertible
sheaves on $C$, then the points $[C'_0]$ and
$[C'_\infty]$ are connected in the stack $\Sigma^e(X/C/k)$
by a chain of rational curves whose nodes
parametrize unobstructed, non-stacky stable sections.
Moreover, the integer $E$ is bounded above by
$\varphi(\deg(R), \deg(s_0))$ where
$\varphi : \mathbb{Z}^2 \to \mathbb{Z}_{\geq 0}$
is a function
such that $B \geq B' \Rightarrow \varphi(B,e) \geq \varphi(B', e)$.
\end{lem}

\begin{proof}
We first describe our choice of $E$.
For any effective Cartier divisor $D \subset C$
denote $R_D \subset R$ the corresponding divisor in $R$.
Note that there exists an integer $E_1$ such that
the linear system $|s_0(C) + R_D|$ is base point free and
very ample for all $D$ with $\deg_C(D) \geq E_1$.
This is true because $s_0(C)$
is relatively very ample for $R \to C$. We may also choose
$E_2$ such that codimension $1$ points of
$|s_0(C) + R_D|$ parametrize nodal curves
for all $D$ with $\deg_C(D) \geq E_2$.
This follows from a simple Bertini type
argument once we have chosen $E_2$ big enough:
namely, so big that
$$
H^0(R, \OO_R(s_0(C) + R_D))
\longrightarrow
H^0(R_Z, \OO_{R_Z}(s_0(C)))
$$
is surjective for $Z \subset C$ any length $2$
closed subscheme. Such an $E_2$ exists by the relatively
very ample comment once more. We will take $E = E_2$.
The easiest way to see that $E \leq \varphi(\deg(R), \deg(s_0))$
with $\varphi$ as in the lemma
is to remark that the family of pairs $s_0 : C \to X, R \to X$
with bounded degrees is bounded.

\medskip\noindent
Choose $e \geq E$.
By Lemma \ref{lem-pivot} it suffices to prove the lemma
assuming all the quills are in the surface $R$.
Let $q_{i,j}$ be as in the statement of the lemma. 
Note that $s_0(C) + \sum R_{q_{0,j}} \sim 
s_\infty(C) + \sum R_{q_{\infty,j}}$ as divisors
on $R$ by Lemma \ref{lem-Abelpen}.

\mni
Let us connect our points
$s_0(C) + \sum R_{q_{0,j}} \in |s_0(C) + \sum R_{q_{0,j}}|$
and $s_\infty(C) + \sum R_{q_{\infty, j}}
\in |s_0(C) + \sum R_{q_{0,j}}|$ by a pair of
general lines $\Pi_1, \Pi_2$. In other words, $\Pi_1$ and $\Pi_2$
intersect, $\Pi_1$ contains the first point and $\Pi_2$ contains the second,
but otherwise the lines are general. We claim that this gives two
rational curves in $\Sigma^e(X/C/k)$ that do the job.
Since the curves $s_0(C) + \sum R_{q_{0,j}}$ and
$s_\infty(C) + \sum R_{q_{\infty,j}}$
are nodal and by our choice of $e$ we see that all points
of $\Pi_1$ and $\Pi_2$ correspond to nodal curves, hence automatically
$\Pi_1$ and $\Pi_2$ correspond to families of stable sections. Finally,
since the point $s_0(C) + \sum R_{q_{0,j}} \in \Pi_1$ corresponds to a
smooth point of $\Sigma^e(X/C/k)$ we see that a general
point of $\Pi_1$ corresponds to a smooth point as well. And since
$\Pi_1 \cap \Pi_2$ is just a general point on $\Pi_1$ this node
corresponds to a smooth point of $\Sigma^e(X/C/k)$
as desired.
\end{proof}

\noindent
Here is one of the key technical lemmas of this paper.

\begin{lem}
\label{lem-findgoodchains}
\marpar{lem-findgoodchains}
In Situation \ref{hyp-defn-porc} assume the restriction
of $f$ to some nonempty open $S \subset C$ satisfies
Hypothesis \ref{hyp-peace}. Let $D$ be an effective
Cartier divisor of degree $2g(C) + 1$ on $C$.
Suppose we have sections
$s_0, s_\infty$ of $X\to C$ which are $D$-free\footnote{It is not
necessary to have the same divisor for the two sections.}.
Let $T_0$, resp.\ $T_\infty$ be the unique irreducible
component of $\Sec(X/C/k)$ of which $s_0$, resp.\ $s_\infty$
is a smooth point (see Proposition \ref{prop-newfree4}).
There exists a dense open $W \subset T_0 \times T_\infty$,
an integer $n \geq 1$, integers $e_1,\ldots, e_{n+1}$, $B$
such that for every point $(\tau_0 , \tau_\infty) \in W$
we have the following:
\begin{enumerate}
\item there exist sections
$\tau_0 = \tau_1, \tau_2, \ldots, \tau_{n+1} =\tau_\infty$,
\item the degree of $\tau_i$ is $e_i$ for $i=1,\ldots,n+1$,
\item $\tau_i^{-1}(X_{f,\pax}) \not = \emptyset$ for
$i = 1, \ldots, n + 1$,
\item $\tau_i$ is free for $i = 1, \ldots, n+1$,
\item there exist ruled surfaces $R_i \to X$, $i=1,\ldots,n$,
\item the degree of $R_i$ is at most $B$, and
\item $\tau_i, \tau_{i+1}$ factor through $R_i$ for $i=1,\ldots,n$.
\end{enumerate}
\end{lem}

\begin{proof}
We will write $m = 2g(C) + 1$ in order to ease notation.
Let $W_0 \subset T_0 \times \text{Sym}^m(C)$,
be the open subset corresponding to pairs $(\tau_0, D_0)$,
such that $\tau_0$ is $D_0$-free.
Similarly we have
$W_\infty \subset T_\infty \times\text{Sym}^m(C)$.
Our assumption is that $W_0$ and $W_\infty$ are not empty.
Recall the spaces $H^m_{X,f}$ defined above
Proposition \ref{prop-KMM}.
By Lemma \ref{lem-resfreesmooth} the maps
$$
\begin{matrix}
res_0 : W_0 \to H^m_{X,f}, \ \ 
(\tau_0, D_0) \mapsto (\tau_0|_{D_0} : D_0 \hookrightarrow X),
\\
res_\infty : W_\infty \to H^m_{X,f},\ \ 
(\tau_\infty, D_\infty)
\mapsto (\tau_\infty|_{D_\infty} : D_\infty \hookrightarrow X)
\end{matrix}
$$
are smooth. Let $e_0 = \deg(s_0)$ and $e_\infty = \deg(s_\infty)$.

\medskip\noindent
Let $S \subset C$ be the nonempty open over which
Hypothesis \ref{hyp-peace} holds. Denote $U = X_{f,\pax}$.
Consider the morphism
$$
\text{ev}_{1, n+1} :
\FCh_2^U(X/C, n)
\longrightarrow
U\times_S U \subset X\times_C X
$$
we studied in Section \ref{sec-peace}.
By Proposition \ref{prop-peace}
there exists an $n$, and an open $V \subset \FCh_2^U(X/C, n)$
such that $\text{ev}_{1,n+1}|_V : V \to U\times_S U$
is surjective, smooth, with irreducible and birationally
rationally connected fibres.
Also, the morphisms $\text{ev}_i|_{V} : V \to U$
are smooth, see Lemma \ref{lem-peacechain}.

\medskip\noindent
Consider the following commutative diagram
$$
\xymatrix{
V \ar[d]^{\text{ev}_{1,n+1}} & & \\
U\times_S U \ar[d] \ar[r] & X\times_C X \ar[ld] &&
C\times T_0 \times T_\infty \ar[ll]_{\tau_0 \times \tau_\infty} \ar[llld] \\
C & & 
}
$$
By the remarks above
this diagram satisfies all the assumptions of
Lemma \ref{lem-avoid}. Let us just point out that
the associated rational map from
$\text{Sym}^m(C) \times T_0 \times T_\infty$
to $H^m_{U\times_SU} = H^m_U \times_{\text{Sym}^m(C)} H^m_U$
is the product map
$$
\text{Sym}^m(C) \times T_0 \times T_\infty
\supset W_0 \times_{\text{Sym}^m(C)} W_\infty
\longrightarrow
H^m_U \times_{\text{Sym}^m(C)} H^m_U
$$
and by the remarks above is even smooth.
We conclude that there exists a variety $T'$,
a dominant morphism $T' \to T_0 \times T_\infty$,
a compactification $V \subset \overline{V}$
over $X\times_C X$ and a morphism
$
\tau : C \times T' \to \overline{V}
$
such that $C \times T' \to \overline{V} \to X\times_C X$
equals $C \times T' \to C\times T_0 \times T_\infty
\to X\times_C X$, and such that the induced rational
map from $\text{Sym}^m(C) \times T'$ to $H^m_V$ is dominant.
See Lemma \ref{lem-avoid} and its proof for explanations.
We may assume the compactification $\overline{V}$ dominates
the compactification
$$
V
\subset
\FCh_2^U(X/C, n)
\subset 
\Ch_2(X/C, n)
$$
by the space $\Ch_2(X/C, n)$ of $2$-pointed chains
of (not necessarily free) lines
in fibres of $X/C$. For example because in the proof of Lemma \ref{lem-avoid}
the compactification $\overline{V}$ could have been chosen
to dominate this compactification at the outset.

\medskip\noindent
At this point, every $t \in T'(k)$ gives rise to a morphism
$\tau_t : C \to \Ch_2(X/C, n)$, which in turn gives rise to 
$n$ ruled surfaces $R_1, \ldots, R_n$ in $X$ over $C$,
and $n+1$ sections $\tau_1,\ldots,\tau_{n+1}$ of $X/C$
such that
\begin{enumerate}
\item $\tau_1$ corresponds to a point of $T_0$,
\item $\tau_{n+1}$ corresponds to a point of $T_\infty$,
\item $\tau_i, \tau_{i+1}$ factor through $R_i$ for $i=1,\ldots,n$. 
\end{enumerate}
By shrinking $T'$ we may assume that both $\tau_1$
and $\tau_{n+1}$ meet $U$ and are $D$-free for some divisor $D$
of degree $m$.

\medskip\noindent
Let $e_i = \deg(\tau_i)$ for $i=2,\ldots,n$.
Recall the constructible sets $W^m_{e_i} \subset H^m_{X,f}$
constructed in Proposition \ref{prop-KMM}, in particular
recall that each $W^m_{e_i}$ contains a dense open
of $H^m_{X,f}$, and hence of $H^m_U$.
Because the associated rational map from 
$\text{Sym}^m(C) \times T'$ to $H^m_V$ is dominant,
and because each $\text{ev}_i : V \to U$ is smooth
we see that after replacing $T'$ by a nonempty open subvariety,
for each $t\in T'(k)$ the associated maps $\tau_i$
map a general divisor $D \subset C$ of degree $m$
to a point of $W^m_{e_i}$. In particular, $\tau_i$
meets $U = X_{f, \pax}$ and there exists a divisor
$D$ of degree $m$ such that $\tau_i$ is $D$-free.
By Lemma \ref{lem-newfree3} all the sections
$\tau_i$ are $1$-free, i.e., free. Since $T' \to T_0 \times T_\infty$ is 
dominant its image contains an open subset and the lemma
is proved.
\end{proof}

\noindent
Let $P \subset \Pine{e,n}(X/C/k)$ be an irreducible component.
Since the space $\Pine{e,n}(X/C/k)$ is smooth this is also a connected
component. And since the forgetful morphisms
$\Phi_{\text{forget}} : \Pine{e,n+m}(X/C/k) \to \Pine{e,n}(X/C/k)$
are smooth with geometrically irreducible fibres, we see that
$P$ determines a sequence $(P_m)_{m \geq 1}$ of irreducible
components, $P_m \subset \Pine{e, n + m}(X/C/k)$.
Note that points of $P_{d - e - n}$ correspond to porcupines
with total degree $d$.

\begin{prop}
\label{prop-pentogether}
\marpar{prop-pentogether}
In Situation \ref{hyp-defn-porc} assume the restriction
of $f$ to some nonempty open $S \subset C$ satisfies
Hypothesis \ref{hyp-peace}. Let $D \subset C$
be an effective Cartier divisor of degree $2g(C) + 1$.
Let $P \subset
\Pine{e_0,n_0}(X/C/k)$ and $P' \subset \Pine{e_\infty,n_\infty}(X/C/k)$
be irreducible components. Assume that a general
point of $P$ corresponds to a porcupine whose body
is $D$-free, and similarly for $P'$. 
There exists an integer $E$ and for every
$e \geq E$ there exists a dense open subscheme
$$
U \subset P_{e - e_0 - n_0}
\times_{\alpha_{\mc{L}}, \underline{\text{Pic}}^e_{C/k}, \alpha_{\mc{L}}}
P'_{e - e_\infty - n_\infty}
$$
such that for any pair of points $(p, q) \in U$ the points $p$ and $q$ are
connected in $\Sigma^e(X/C/k)$ by a chain of rational curves whose
marked points and nodes parametrize unobstructed, non-stacky stable sections.  
\end{prop}

\begin{proof}
Let $T_0$, resp.\ $T_\infty$ denote the irreducible components
of $\Sec(X/C/k)$ to which the bodies of porcupines in $P$, resp.\ 
$P'$ belong. Let $n, e_1,\ldots, e_{n+1}, B$ and
$W \subset T_0\times T_\infty$ be the data
found in Lemma \ref{lem-findgoodchains}.
Note that $e_0 = e_1$ and $e_\infty = e_{n+1}$.
Set $E = \max\{e_i\} + \max\{0,\varphi(B, e_i)\} + 2g(C) + 999$
where $\varphi(-,-)$ is the function mentioned in Lemma
\ref{lem-pentogether}. Pick $e \geq E$.
The open subset $U$ will correspond to the pairs $(p,q)$
in the fibre product with
$p = (s_0, q_{0,i}, r_{0,i} \in L_{0,i}(k), L_{0,i} \to X)$,
$q = (s_\infty, q_{\infty,i},
r_{\infty,i} \in L_{\infty,i}(k), L_{\infty,i} \to X)$
such that the pair $(s_0, s_\infty) \in W$. 

\medskip\noindent
Namely, let $R_1,\ldots, R_n$, $\tau_1,\ldots,\tau_\infty$
be as in Lemma \ref{lem-findgoodchains} adapted to the
pair $(s_0, s_\infty)$. Denote
$\xi = \alpha_{\mc{L}}(p) = \alpha_{\mc{L}}(q)
\in \underline{\text{Pic}}^e_{C/k}(k)$ the 
corresponding degree $e$ divisor class on $C$.
For each $i = 2, \ldots, n$ it is possible to
find a reduced effective divisor $Q_i \subset C$,
$Q_i \subset \tau_i^{-1}(X_{f,\pax})$
of degree $e - e_i \geq 2g(C) + 999$ such that 
$\tau_i^*(\mc{L})(Q_i)$ is $\xi$.
Now we will repeatedly use Lemmas \ref{lem-pivot} and
\ref{lem-pentogether} to find our chain of rational curves.
Let $p_1$ be the point of $\Pine{e_0, e - e_0}(X/C/k)$
gotten from $p$ by moving all $e - e_0 = e - e_1$ quills
attached to $s_0 = \tau_1$ into the ruled surface $R_1$.
Then $p$ is connected to $p_1$ by a suitable chain
of rational curves by Lemma \ref{lem-pivot}.
Let $p_2 \in \Pine{e_2, e - e_2}(X/C/k)$ 
be the porcupine with body $\tau_2$ and quills attached to
the points of the divisor $Q_2$ lying in the ruled surface
$R_1$. By Lemma \ref{lem-pentogether}
we see that $p_1$ is connected to $p_2$
by a suitable chain of rational curves.
Let $p_3 \in \Pine{e_2, e - e_2}(X/C/k)$ 
be the porcupine with body $\tau_2$ and quills attached to
the points of the divisor $Q_2$ lying in the ruled
surface $R_2$. By Lemma \ref{lem-pivot} the points
$p_2$ and $p_3$ are connected by a suitable chain of
rational curves. And so on and so forth.
\end{proof}

\begin{cor}
\label{cor-Zdsame}
\marpar{cor-Zdsame}
In Situation \ref{hyp-defn-porc}, assume in addition
Hypothesis \ref{hyp-peace} for the restriction of $f$
to some nonempty open $S \subset C$. Let $e_0, e_\infty$
be integers and let $Z, Z'$ be irreducible components
of $\Sigma^{e_0}(X/C/k)$, $\Sigma^{e_\infty}(X/C/k)$ whose general point
parametrizes a free section of $f$.
Consider the famlies of components $Z_e$, $Z'_e$ constructed
in Lemma \ref{lem-Zd}. Then for $e \gg 0$ we have
$Z_e = Z'_e$.
\end{cor}

\begin{proof}
By Lemmas \ref{lem-Zdbetter} and \ref{lem-Zdup} we may assume that
$Z$ and $Z'$ each contain a section which meets $X_{f,smooth}$ and
is $2g(C)+1$-free. By Proposition \ref{prop-pentogether}
there exists an $E$ and for all $e \geq E$
there are points $p \in Z_e$ and $q \in Z'_e$ corresponding
to porcupines with bodies in $Z$ and $Z'$ such that
$p$ can be connected to $q$ by a chain of rational curves
in $\Sigma^e(X/C/k)$ whose marked points correspond to
unobstructed, non-stacky stable sections. In particular such a
chain lies in a unique irreducible component!
\end{proof}

%%%%%%%%%%%%%%%%%%%%%%%%%%%%%%%%%%%%%%%%%%%%%%%%%%%%%%%%%%%%%%%
%%
%% Varieties and peaceful chains
%% 
%%%%%%%%%%%%%%%%%%%%%%%%%%%%%%%%%%%%%%%%%%%%%%%%%%%%%%%%%%%%%%%

\section{Varieties and peaceful chains}
\label{sec-varlines} 
\marpar{sec-varlines}

\noindent
In this section we spell out what consequences we can
draw from the work done in the previous sections for
the moduli spaces of rational curves on a smooth
projective variety whose moduli spaces of lines
satisfy the conditions imposed in Section \ref{sec-peace}.

\begin{hyp}
\label{hyp-peaceY}
\marpar{hyp-peaceY}
Let $k$ be an uncountable algebraically closed field of characteristic $0$.
Let $Y$ be a smooth projective variety over $k$. Let $\mc{L}$ be an
ample invertible sheaf on $Y$. We assume that $Y \to \text{Spec}(k)$
satisfies Hypothesis \ref{hyp-peace}.
\end{hyp}

\noindent
Assume that Hypothesis \ref{hyp-peaceY} holds.
Let $X = \PP^1 \times Y$ and 
$$
f = \text{pr}_1 :
X = \PP^1 \times Y
\longrightarrow
\PP^1.
$$
The first thing to remark is that a morphism
$g : \PP^1 \to Y$ gives rise to a section $s : \PP^1 \to X$,
and conversely. Also, the section $s$ is $0$-free, i.e., unobstructed
(see Definition \ref{defn-classicalfree}),
if and only if the rational curve $g : \PP^1 \to Y$ is unobstructed
(at worst summands $\OO(-1)$ in the direct sum decomposition
of $g^*T_Y$). Similarly, $s$ is $1$-free, i.e., free (see Definition
\ref{defn-classicalfree}), if and only
if the rational curve $g : \PP^1 \to Y$ is free.
Arguing in this way we see that we can interpret many of the
previous results for the family $f : X \to \PP^1$ in terms of
the moduli spaces of lines on $Y$.

\medskip\noindent
Instead of reproving everything from scratch in this setting
we make a list of statements and we point out the corresponding
lemmas, and propositions in the more general treatement above.
We denote $\Kgnb{0,0}(Y, e)$ the Kontsevich moduli space
of degree $e$ rational curves in $Y$. Also, let $Y_{\pax}$ denote
the set of \peaceful\ points (w.r.t.\ $Y\to \text{Spec}(k)$), see
Definition \ref{defn-newpeace}.
\begin{enumerate}
\item There is a sequence of irreducible components
$Z_e \subset \Kgnb{0,0}(Y,e)$, $e \geq 1$ uniquely
characterized by the property that every comb in $Y$
whose handle is contracted to a point of $Y_{\pax}$ and
which has $e$ teeth mapping to lines in $Y$ is in $Z_e$.
This is the exact analogue of Lemma \ref{lem-Zd}, starting with
the fact that a constant map $\PP^1 \to Y$ is free!
\item A general point of $Z_e$ corresponds to a map
$\PP^1 \to Y$ which is free and meets $Y_{\pax}$.
This is the exact analogue of Lemma \ref{lem-Zdgenpt}.
\item Consider a comb in $Y$ whose handle has degree $e$
with $e' -e$ teeth which are lines.
Such a comb is in $Z_{e'}$ if the handle is free,
is in $Z_e$, and the attachment points map to $Y_{\pax}$.
This is the exact analogue of Lemma \ref{lem-Zdup}.
\item Any stable map $f : C \to Y$ in $\Kgnb{0,0}(Y, e')$
which has exactly one component a free curve corresponding
to  a point in $Z_{e}$ and all other components free lines
corresponds to a smooth nonstacky point of $Z_{e'}$.
This is the exact analogue of Lemma \ref{lem-Zdchains}.
\item
\label{better}
\marpar{better}
Given any integer $m \geq 0$ there exists $e \gg 0$
such that the general point of $Z_e$ corresponds to
a map $f : \PP^1 \to Y$ with $H^1(\PP^1, f^*T_Y(-m)) = 0$.
This is the exact analogue of Lemma \ref{lem-Zdbetter}.
\item
\label{samecomp}
\marpar{samecomp}
Let $Z' \subset \Kgnb{0,0}(Y, e_0)$ be an irreducible
component whose general point corresponds to a free rational
curve. There exists an integer $E \gg 0$ such that
for all $e \geq E$ we have the following
property: Consider a comb in $Y$ whose handle has degree $e_0$
with $e - e_0$ teeth which are lines.
Such a comb is in $Z_{e}$ if the handle is free,
is in $Z'$, and the attachment points map to $Y_{\pax}$.
(And in particular any smoothing
of this comb defines a point of $Z_e$.)
This is the exact analogue of Corollary \ref{cor-Zdsame},
once you observe that such a comb defines an unobstructed
nonstacky point of $\Kgnb{0,0}(Y, e)$.
\end{enumerate}
It is perhaps not necessary, but we point out that
instead of defining $Z_e$ as in (1) above, we could just
define $Z_e$ as the unique irreducible component of
$\kgnb{0,0}(Y,e)$ which contains all trees of free lines.
This follows from (4) (with $e' = 0$).

\begin{lem}
\label{lem-fibreZdirred}
\marpar{lem-fibreZdirred}
In Situation \ref{hyp-peaceY} above.
Let $Z_{e,1} \subset \Kgnb{0,1}(Y,e)$ be the
unique irreducible component dominating $Z_e \subset \Kgnb{0,0}(Y,e)$.
The general fibre of $\text{ev} : Z_{e,1} \to Y$
is irreducible.
\end{lem}

\begin{proof}
Let us temporarily denote by $Z^\circ_{e, 1} \subset Z_{e,1}$
the dense open locus of stable $1$-pointed maps $\bigcup C_i \to Y$
all of whose irreducible components $C_i \to Y$ are free and
whose nodes map to points of $Y_{\pax}$.
The morphism $\text{ev} : Z^\circ_{e, 1} \to Y$ is smooth.
Consider the locus
$W_{e,1} \subset Z^\circ_{e, 1}$ whose points
correspond to those combs whose handles are contracted
to a point of $Y_{\pax}$ and whose teeth are (automatically) free
lines and whose marked point is on the handle.
By definition of the irreducible components $Z_e$ the 
space $W_{e,1}$ is nonempty. It is smooth and
$W_{e,1} \to Y$ is still smooth onto $Y_{\pax}$.
Clearly, Hypothesis \ref{hyp-peace} for $Y \to \text{Spec}(k)$
implies the general fibre of $W_{e, 1} \to Y$ is
irreducible. From this, and the irreducibility
of $Z_{e, 1}$ it follows that the general
fibre of $Z_{e, 1} \to Y$ is irreducible, for
example by \cite[Proposition 4.5.13]{EGA4}.
\end{proof}

\begin{lem}
\label{lem-Zdabs}
\marpar{lem-Zdabs}
In Situation \ref{hyp-peaceY} above. Let
$Z' \subset \Kgnb{0,0}(Y, e_0)$ and $E$ be as in (\ref{samecomp}) above.
Suppose that $g : C \to Y$ is a genus $0$ Kontsevich stable map
such that $C = \bigcup_{i=1}^A C_i \cup \bigcup_{j=1}^B L_j$ is
a decomposition into irreducible components with the
following properties:
(1) each $C_i \to Y$ is free and corresponds to a point
of $Z'$, (2) each $L_i \to Y$ is a free line and a leaf of
the tree, and (3) $B \geq E - e_0$. Then $g$ defines an
unobstructed point of $Z_{Ae_0 + B}$.
\end{lem}

\begin{proof}
We will show there is a connected chain of curves
contained in the unobstructed locus of $\Kgnb{0,0}(Y, 1)$
which connects $g$ to a point of $Z_{Ae_0+B}$.
We may first move the map a little bit such that
all the nodes of $C$ are mapped into $Y_{\pax}$.
Next, we may $1$ by $1$ slide all the lines $L_i$
along the curves $C_i$ and onto one of
the curves $C_{i_0}$ which is a leaf of the tree $\bigcup C_i$.
Say this curve is $C_1$, and is attached to
$C'' = \bigcup_{i\not=1} C_i$ at the point $p \in C_1$
and $q \in C''$. Set $C' = C_1 \cup \bigcup L_i$, so
$C = C' \cup_{p \sim q} C''$. By (\ref{samecomp}) above
we see that $C' \to Y$
defines an unobstructed point of $Z_{e_0 + B}$.

\medskip\noindent
We may assume (after possibly moving the map a little bit)
that the attachment point $p$ is mapped to a point
$y \in Y$ such that the fibre of $Z_{e_0+B, 1} \to Y$
is irreducible over this point, see Lemma \ref{lem-fibreZdirred}.
Hence we may connect
the $1$-pointed stable map $(C', p) \to Y$ 
inside the fibre of $Z_{e_0 + B,1} \to Y$ over $y$ to a curve
which is made out of a comb $C'''$ whose handle is contracted and
whose teeth are free lines. Since the marked point
is throughout mapped to the same $y \in Y(k)$ this connects
the original stable map $C' \cup C'' \to Y$ to a stable map
$C''' \cup C'' \to Y$ where the number of lines in $C'''$ is
now $e_0 + B$. We may again slide these lines over to some
irreducible component $C_i$ of $C''$ and continue until all
the $C_i$ are gone, and so obtain a point of $Z_{Ae_0 + B}$.
This proves the lemma.
\end{proof}

\noindent
The notation in the section conflicts with the notation introduced
in Lemmma \ref{lem-Zd} in case $Y$ is a fibre of a family as
in Situation \ref{hyp-defn-porc} such that Hypothesis \ref{hyp-peace}
holds over a nonempty open of $C$. In that situation we will use
the notation $Z_e(X_t)$ to denote the irreducible component 
defined in this section for the irreducible fibre $X_t$.

\begin{lem}
\label{lem-Zdgluemore}
\marpar{lem-Zdgluemore}
In Situation \ref{hyp-defn-porc}, assume in addition
Hypothesis \ref{hyp-peace} for the restriction of $f$
to some nonempty open $S \subset C$. Suppose that
$s : C \to X$ is a free section, and let $Z \subset \Sigma(X/C/k)$
be the unique irreducible component containing the moduli point
$[s]$. Let $Z_e$ as in Lemma \ref{lem-Zd}. For any
$t_1, \ldots, t_\delta \in S(k)$, for any free rational maps
$s_i : \PP^1 \to X_{t_i}$ such that $s_i(0) = s(t_i)$,
if $s_i$ corresponds to $Z_{e_i}(X_{t_i})$,
then the comb $C \cup \bigcup \PP^1 \to X/C$ defines an unobstructed
point of $Z_{\deg(s) + \sum e_i}$.
\end{lem}

\begin{proof}
This is very similar to the proof of Lemma \ref{lem-Zdabs} above.
Namely, we first move the comb (as a comb) such that the points
$s(t_i)$ are in the locus where the fibres of $Z_{e_i, 1} \to X_{t_i}$
are irreducible. then we connect the $s_i$ in these fibres to
chains of lines. After this we can apply Lemma \ref{lem-Zdchains}
for example.
\end{proof}

%%%%%%%%%%%%%%%%%%%%%%%%%%%%%%%%%%%%%%%%%%%%%%%%%%%%%%%%%%%%%%%
%%
%% Families of varieties and lines in fibres
%% 
%%%%%%%%%%%%%%%%%%%%%%%%%%%%%%%%%%%%%%%%%%%%%%%%%%%%%%%%%%%%%%%

\section{Families of varieties and lines in fibres}
\label{sec-famvarlines} 
\marpar{sec-famvarlines}

\noindent
This section contains three lemmas that did not seem to
fit well in other sections.

\begin{lem}
\label{lem-goodfibres}
\marpar{lem-goodfibres}
In Situation \ref{hyp-defn-porc}, assume in addition
Hypothesis \ref{hyp-peace} for the restriction of $f$
to some nonempty open $S \subset C$. Every geometric
fibre $X_t$ of $f$ is integral and every pair of points
of the smooth locus $X_t^{\circ}$ of the fibre is contained in
(1) a chain of lines contained in $X_t^{\circ}$,
and (2) a very free rational curve in the smooth locus.
\end{lem}

\begin{proof}
Let $t \in C(k)$.
Situation \ref{hyp-defn-porc} requires
$X_t$ to be irreducible.
Since $X$ is smooth, by \cite{GHS},
the fibre $X_t$ intersects the
smooth locus of $f$. Hence 
$X_t$ is generically reduced and
Cohen-Macaulay, and hence is reduced.

\medskip\noindent
Let $x, y \in X_t(k)$
be smooth points of $X_t$.
We will construct a very free
rational curve in $X_t$ passing through $x$
and $y$. There exists a finite cover
$C' \to C$ etale over $t$ and over the
discriminant of $f : X \to C$,
a point $t' \in C'(k)$
and sections $a, b : C' \to X_{c'}$ such
that $a(t') = x$, $b(t') = y$.
To construct $C'$ choose two complete
intersection curves $C_x, C_y \subset X$
general apart from the condition $x \in C_x$ and
$y \in C_y$ and let $C'$ be an irreducible
component of the normalization of $C_x \times_C C_y$
(details left to the reader).

\medskip\noindent
Set $X' = X_{C'}$ and $f' : X' \to C'$ the
base change of $f$. By construction, all conditions
of Situation \ref{hyp-defn-porc}, and 
Hypothesis \ref{hyp-peace} remain satisfied for $X' \to C'$.
Hence, upon replacing $X\to C$ by $X' \to C'$, and
$t$ by $t'$ we may assume
there are sections $s_0, s_\infty$ such that
$s_0(t) = x$ and $s_\infty(t) = y$.
Attaching very free rational curves in fibres of $f$
and deforming we may assume that $s_0$ and $s_\infty$
are sufficiently general such that
$(s_0, s_\infty) : C \to X\times_C X$ meets the
open set $V \subset X_S\times_S X_S \subset X\times_C X$
from Hypothesis \ref{hyp-peace}.

\medskip\noindent
By \cite{GHS} we can find a morphism $\phi : C \to \Ch_2(X/C, n)$
which meets $\FCh_2(X_S/S, n)$ and such that
$\text{ev}_{1, n + 1} \circ \phi = (s_0, s_\infty)$.
As in the proof of Lemma \ref{lem-findgoodchains}
this translates into a collection of
$n$ ruled surfaces $R_1,\ldots, R_n$ in $X/C$,
$n+1$ sections $\tau_1,\ldots,\tau_{n+1}$
of $X/C$ such that $\tau_1 = s_0$, $\tau_{n+1} = s_\infty$
and such that $\tau_i, \tau_{i+1}$ factor through $R_i$.

\medskip\noindent
Note that the fibres $R_{1,t}, R_{2,t},\ldots,R_{n,t}$
form a chain of lines connecting $x$ to $y$ in the fibre
$X_t$. By Lemma \ref{lem-ruledsmooth} this chain
lies in the smooth locus $X_t^{\circ}$ of $X\to C$.
Thus any pair of points of $X_t^{\circ}$
may be connected by a chain of lines
in $X_t^{\circ}$. The result follows upon applying
\cite[IV Theorem 3.9.4]{K} for example.
\end{proof}

\begin{rmk}
\label{rmk-improve}
\marpar{rmk-improve}
The Lemmas \ref{lem-ruledsmooth} and \ref{lem-goodfibres} could
be stated with slightly weaker hypotheses. Namely, we
could remove the assumption (in \ref{hyp-defn-porc}) that all the
fibres of $X \to C$ are geometrically irreducible.
This does not pose a problem for Lemma \ref{lem-ruledsmooth}. For 
\ref{lem-goodfibres} it means that one has to show there are
no irreducible components of $X_t$ which have multiplicity $> 1$.
We may try to prove this by specializing the ``chain of ruled
surfaces'' $R_1,\ldots, R_n, \tau_1, \ldots, \tau_{n+1}$
as you move one of the two points $x$, $y$ into a presumed
higher multiplicity component. What might happen is that the
ruled surfaces may break, and we do not see how to conclude
the proof. However, if $\mc{L}$ is very ample on all the fibres
(an assumption that always holds in practice), then this argument
works: in this case any line that meets a higher multiplicity
component must be contained in it and it is easy to conclude from this.
\end{rmk}

\begin{lem}
\label{lem-nonfreesweep}
\marpar{lem-nonfreesweep}
Let $X$ be a nonsingular projective variety over
an algebraically closed field $k$ of characteristic $0$.
Let $\mc{L}$ be an ample invertible sheaf on $X$.
\begin{enumerate}
\item The locus $T \subset X$ swept out by non free
lines is a proper closed subset. Let $D'_1,\ldots, D'_r$
be desingularizations (see \cite{Hir}) of the irreducible components
$D_1,\ldots, D_r$ of $T$ which have codimension $1$ in $X$.
\item There exists a closed codimension $2$ subset $T' \subset X$
containing the singular locus of $T$ and all codimension $\geq 2$
components such that any non free line not contained in $T'$ is
the image of a free rational curve on some $D'_i$.
\end{enumerate}
\end{lem}

\begin{proof}
The first assertion is basic, see \cite[II Theorem 3.11]{K}.
For the second, first let $T' \subset T$ be the closed
subset over which $\sqcup D'_i \to T$ is not an isomorphism.
Any line in $T$ not contained in $T'$ is the image of a
unique rational curve on some $D'_i$. Finally we apply the
general fact that since these curves cover $D'_i$ the
general one is free.
\end{proof}

\begin{lem}
\label{lem-limitfree}
\marpar{lem-limitfree}
Notations and assumptions as in Lemma \ref{lem-nonfreesweep}.
Let $C$ be a nonsingular curve, and let
$f : \mathbf{P}^1\times C \to X$ be a family of lines in $X$.
Assume that $\gamma = f|_{\{0\}\times C} : C \to X$ does not meet $T'$,
and meets $T$ transversally at all of its points of intersection.
Then for each $c \in C(k)$ the line $L_c \to X$ is free.
\end{lem}

\begin{proof}
For a general point $t \in C(k)$ the line
$f_t : \mathbf{P}^1 \to X$ is free since $f_t(0) \not\in T$.
Thus let $t \in C$ be a point such that $f_t(0) \in T$.
If $f_t(\mathbf{P}^1) \not\subset T$ we win also.
Hence by Lemma \ref{lem-nonfreesweep} the line
$f_t : \mathbf{P}^1 \to X$ is the image of a free
rational curve $g : \mathbf{P}^1 \to D'$ on the
resolution $D'=D'_i \to D = D_i$ for some $i$.
Consider the maps
$$
g^*T_D \to f_t^*T_X \to f_t^* N_{D}X.
$$
Note that the composition is zero.
Since $g^*T_D$ is globally generated (because $g$ is free)
we see the only way $f_t^*T_X$ could have a negative
summand is if $f_t^*T_X \to f_t^*N_DX$ factors
through a negative invertible sheaf on $\mathbf{P}^1$.
However, the given deformation $f : \mathbf{P}^1 \times C
\to X$ gives rise to a vector $\theta \in H^0(\mathbf{P}^1,
f_t^*T_X)$ which at the point $0$ correspond to
the tangent vector $d\gamma \in T_{\gamma(0)}X$ which
by assumption points out of $D$. Thus $f_t^*T_X \to f_t^*N_DX$
is nonzero on $H^0$ and we win.
\end{proof}

%%%%%%%%%%%%%%%%%%%%%%%%%%%%%%%%%%%%%%%%%%%%%%%%%%%%%%%%%%%%%%%
%%
%% Perfect pens
%% 
%%%%%%%%%%%%%%%%%%%%%%%%%%%%%%%%%%%%%%%%%%%%%%%%%%%%%%%%%%%%%%%

\section{Perfect pens}
\label{sec-perfectpens} 
\marpar{sec-perfectpens}

\noindent
In this section we introduce the notion of twisting ruled surfaces.
Recall that ruled surfaces in $X \to C$ w.r.t.\ $\mc{L}$ were defined
at the start of Section \ref{sec-pencils}. The morphism
$R \to X$ maps into the smooth locus of $X \to C$, see
Lemma \ref{lem-ruledsmooth}. Our first task is to show
that (many) free sections lie on scrolls of free lines.

\begin{lem}
\label{lem-freeruled}
\marpar{lem-freeruled}
In Situation \ref{hyp-defn-porc} assume Hypothesis \ref{hyp-peace}
holds for the restriction of $f$ to some nonempty open
$S \subset C$. Let $e_0$ be an integer and let $Z$ be an
irreducible component of $\Sigma^{e_0}(X/C/k)$ whose
general point parametrizes a free section. Then there exists
a nonempty open $U \subset Z$ such that every $u\in U(k)$
corresponds to a section $s$ such that
(a) it is penned by a ruled surface
$R \to X$, and (b) any ruled surface $R$
penning $s$ has the property that all of its
fibres are free lines in $X$.
\end{lem}

\begin{proof}
We are going to use Lemma \ref{lem-limitfree}.
Note that $\mc{L}$ may not be ample on $X$. However, to prove
the lemma we may replace $\mc{L}$ by $\mc{L} \otimes f^*\mc{N}$
for some suitable very ample sheaf $\mc{N}$ on $C$. This
will allow us to assume that $\mc{L}$ is ample on the total space
$X$, and it allows us to assume that any rational curve in $X$
of degree $1$ is in a fibre of $f : X \to C$.

\medskip\noindent
Let $T' \subset T \subset X$, $D_i$ be as in Lemma \ref{lem-nonfreesweep}.
Because $T' \subset X$ has codimension $\geq 2$ there is a nonempty
open $U' \subset Z$ such that every $u\in U'(k)$
corresponds to a section $s$ which is disjoint from $T'$.
Furthermore, we then pick a nonempty open $U'' \subset U'$ such
every $u\in U''(k)$ corresponds to a section $s$ that meets
each irreducible component $D_i$ transversally in smooth points,
see \cite[II Proposition 3.7]{K}. Finally, we may, by 
Hypothesis \ref{hyp-peace} part (1), find a further
nonempty open $U \subset U''$ such that each section $s: C \to X$
corresponding to a point of $U$ meets the locus over which the
evaluation morphism $\text{ev} : \Kgnb{0,1}(X/C,1) \to X$
has irreducible rationally connected fibres.

\medskip\noindent
Let $s : C \to X$ correspond to a $k$-point of $U$.
By \cite{GHS} we can find a morphism $g : C \to \Kgnb{0,1}(X/C,1)$
such that $\text{ev} \circ g = s$. This corresponds to
a ruled surface $R \to X$ which pens $s$. This proves (a).
Next, let $R$ be any ruled surface penning $s$.
By Lemma \ref{lem-ruledsmooth} it lies in the smooth locus of
$X \to C$. By Lemma \ref{lem-limitfree} and our choice of $s$
all fibres of $R \to C$ are free curves in $X$. This proves (b).
\end{proof}

\noindent
To define the notion of a twisting ruled surface, we introduce some
notation. Let $f : X \to C$, $\mc{L}$ be as in Situation \ref{hyp-defn-porc}.
Consider a ruled surface $h : R \to X$. The following commutative
diagram of coherent sheaves on $R$ with exact rows
$$
\xymatrix{
0 \ar[r] &
T_{R/C} \ar[r] \ar[d] &
h^*T_{X/C} \ar[r] \ar[d] &
N_{R/X}\ar@{=}[d] \ar[r] &
0
\\
0 \ar[r] &
T_{R} \ar[r] &
h^*T_{X} \ar[r] &
N_{R/X} \ar[r] &
0
}
$$
defines the coherent sheaf $N_{R/X}$. If the sheaf $\mc{L}$
is relatively very ample (which is always satisfied in practice),
the map $h : R \to X$ will be a closed immersion and $N_{R/X}$
will be a locally free sheaf.
We will call the sheaf $N_{R/X}$ the \emph{normal bundle}
regardless of whether $R \to X$ is an embedding or not.

\begin{rmk}
\label{rmk-technical}
\marpar{rmk-technical}
Here are some technical remarks for those readers who
enjoy thinking about the case where $\mc{L}$ is only
assumed ample and not very ample on the fibres of $X \to C$.
The first is that $N_{R/X}$ is flat over $C$.
This follows as the maps $T_{R/C}|_{R_t} \to h^*T_{X/C}|_{R_t}$
are injective, and \cite[Section (20.E)]{Matsumura}.
In particular it has depth $\geq 1$, its torsion 
is supported in codimension $\geq 1$, and any fibre
meets the torsion locus in at most finitely many points.
This also shows that $N_{R/X}|_{R_t} = N_{R_t/X_t}$
the normal bundle of the line $R_t \to X_t$ (defined
similarly). The second is that the deformation theory of
the morphism $R \to X$ (not the $\text{Hom}$-space)
is given by $H^0(R, N_{R/X})$ (infinitesimal deformations)
and $H^1(R, N_{R/X})$ (obstruction space). This is because
the cotangent complex of $h : R \to X$ is given by
$h^*\Omega^1_{X} \to \Omega^1_R$ which is quasi-isomorphic
to $h^*\Omega^1_{X/C} \to \Omega^1_{R/C}$. It follows that
$\mathbb{E}\text{xt}^i(h^*\Omega^1_{X/C} \to \Omega^1_{R/C},\OO_R)
= H^{i-1}(R, N_{R/X})$ as usual.
\end{rmk}

\begin{defn}
\label{defn-twisting}
\marpar{defn-twisting}
In Situation \ref{hyp-defn-porc}. Let $R \to X$ be 
a ruled surface in $X$ and let $D$ be a Cartier
divisor on $R$. For every nonnegative integer $m$, 
we say $(R,D)$ is \emph{$m$-twisting} if
\begin{enumerate}
\item the complete linear system $|D|$ is basepoint free,
\item the cohomology group $H^1(R,\OO_R(D))$ is $0$,
\item $D$ has relative degree $1$ over $C$,
\item the normal bundle $N_{R/X}$ is globally generated,
\item $H^1(R, N_{R/X})$ equals $0$, and 
\item we have $H^1(R, N_{R/X}(- D - A)) = 0$
for every divisor $A$ which is the pull back
of any divisor on $C$ of degree $\leq m$.
\end{enumerate}
\end{defn}

\noindent
This is a ``relative'' definition -- it is defined with respect to the
morphism $f$. An important special case is an ``absolute'' variant,
see Definition \ref{defn-twistabs} below.

\medskip\noindent
Suppose we are in Situation \ref{hyp-defn-porc} and suppose
that $(R,D)$ is an $m$-twisting ruled surface in $X/C$.
By assumption $|D|$ is nonempty, base point free, and
of relative degree $1$. Hence a general element is smooth
and defines a section $\sigma : C \to R$. We will often
say ``let $(R, \sigma)$ be an $m$-twisting surface'' to denote
this situation. Having chosen $\sigma$ we can think of $R \to X$
as a family of stable $1$-pointed lines. Let
$g = g_{(R,\sigma)} : C \to \Kgnb{0,1}(X/C, 1)$
denote the associated morphism.

\begin{lem}
\label{lem-twistimplies}
\marpar{lem-twistimplies}
In the situation above:
\begin{enumerate}
\item The image of $g = g_{(R,\sigma)}$ 
lies in the unobstructed (and hence smooth) locus of
the morphism $\text{ev} : \Kgnb{0,1}(X/C, 1) \to X$.
\item We have $H^1(C, g^*T_{\text{ev}}(-A)) = 0$ for
every divisor $A$ of degree $\leq m$ on $C$.
\item Let $\Phi : \Kgnb{0,1}(X/C, 1) \to \Kgnb{0,0}(X/C, 1)$
be the forgetful morphism (which is smooth).
Then $g^*T_{\Phi}$ is globally generated, and $H^1(C, g^*T_{\Phi}) = 0$.
\item The section $h \circ \sigma : C \to X$ is free, see
Definition \ref{defn-classicalfree}.
\end{enumerate}
\end{lem}

\begin{proof}
The fact that $N_{R/X}$ is globally generated implies that
for each $t\in C(k)$ the fibre $R_t \to X$ is free. This follows
upon considering the exact sequences
$0 \to T_R|_{R_t} \to h^*T_X|_{R_t} \to N_{R/X}|_{R_t} \to 0$
(exact by flatness of $N_{R/X}$ over $C$), and the fact that
$T_R|_{R_t}$ is globally generated. This implies the image
of $g$ is in the unobstructed locus for $\text{ev}$.

\medskip\noindent
Let $\pi : R \to C$ be the structural morphism.
The pull back of the relative tangent bundle $T_{\text{ev}}$
by $g$ is canonically identified with $\pi_*N_{R/X}(-\sigma)$.
This is true because the normal bundle of a fibre $R_t \to X$
is an extension of the restriction $N_{R/X}|_{R_t}$ by a 
rank $1$ trivial sheaf on $R_t \cong \PP^1$. Details left
to the reader. The assumptions of Definition \ref{defn-twisting}
imply that $R^1\pi_* N_{R/X}(-\sigma) = 0$. Hence
$H^1(C, g^*T_{\text{ev}}(-A)) = H^1(R, N_{R/X}(-\sigma - \pi^*A))$.
Thus we get the desired vanishing from the definition of
twisting surfaces.

\medskip\noindent
The pull back $g^*T_{\Phi}$ is canonically identified
with $\sigma^* \OO_R(\sigma)$.
The global generation of the sheaf $g^*T_{\Phi}$ is
therefore a consequence of the base point freeness
of $\OO_R(\sigma)$ of Definition \ref{defn-twisting}.
The vanishing of $H^1(C, g^*T_{\Phi})$ follows on considering
the long exact cohomology sequence associated
to $0 \to \OO_R \to \OO_R(\sigma) \to \sigma_*\sigma^* \OO_R(\sigma) \to 0$
and the vanishing of $H^1(R, \OO_R(\sigma))$ in Definition
\ref{defn-twisting}.

\medskip\noindent
Consider the exact sequence $\sigma^*T_{R/C} \to \sigma^*h^*T_{X/C}
\to \sigma^*N_{R/X} \to 0$.
Note that $\sigma^*T_{R/C} = \sigma^*
\OO_R(\sigma)$. By Definition \ref{defn-twisting} both
$\sigma^*\OO_R(\sigma)$ and $\sigma^*N_{R/X}$ are globally
generated. Thus by the exact sequence 
$\sigma^*h^*T_{X/C}$ is globally generated.
In the previous paragraph we showed
that $H^1(C, \sigma^*\OO_R(\sigma))=0$.
There is an exact sequence
$N_{R/X}(-\sigma) \to N_{R/X} \to \sigma_* \sigma^*N_{R/X} \to 0$.
By Definition \ref{defn-twisting}, $H^1(R, N_{R/X}) = 0$. The group
$H^2(R, N_{R/X}(-\sigma))$ vanishes as its Serre dual
$\text{Hom}_R(N_{R/X}(-\sigma), \omega_R)$ is zero
(hint: consider restriction to fibres).
Together these imply that $H^1(C, \sigma^*N_{R/X}) = 0$.
Thus the first exact sequence of this paragraph implies that
$H^1(C, \sigma^* h^* T_{X/C}) = 0$. Thus $(h \circ \sigma)^*
T_{X/C}$ is globally generated with trivial $H^1$ and 
we conclude that $h \circ \sigma$ is $1$-free, i.e., free.
\end{proof}

\noindent
The following innocuous looking lemma is why we introduce
twisting surfaces. It will (eventually) show that any
point of the boundary is connected by a $\PP^1$ to a 
point in the interior.

\begin{lem}
\label{lem-whytwist}
\marpar{lem-whytwist}
In Situation \ref{hyp-defn-porc}, let $(R, \sigma)$
be $m$-twisting with $m \geq 1$. Let $t\in C(k)$.
The stable map $\sigma(C) \cup R_t \to X$ defines a
nonstacky unobstructed point of $\Sigma(X/C/k)$ which is
connected by a rational curve in $\Sigma(X/C/k)$
to a free section of $X \to C$.
\end{lem}

\begin{proof}
The fact that the point is nonstacky
comes from the fact that sections and lines have no
automorphisms. The fact that the point is unobstructed
follows from Lemma \ref{lem-twistimplies} above.
Consider the linear system $|\sigma + R_t|$.
Since $H^1(C, \OO_R(\sigma)) = 0$ the map
$H^0(R, \OO_R(\sigma + R_t)) \to H^0(R_t , \OO_{R_t}(1))$
is surjective. This implies that $|\sigma + R_t|$
is base point free. A general member of $|\sigma + R_t|$
is a section $\sigma' : C \to R$. It is trivial to
show that $(R, \sigma')$ is $(m-1)$-twisting. 
Hence by Lemma \ref{lem-twistimplies} we see that
$\sigma'$ is free. The result is clear now by considering
the pencil of curves on $R$ connecting $\sigma + R_t$
to $\sigma'$.
\end{proof}

\noindent
It is useful to have a criterion that guarantees the
existence of a twisting surface. In particular, we would
like a condition formulated in terms of the map
$g : C \to \Kgnb{0,1}(X/C, 1)$. We do not know
a good way to do this unless $g(C) = 0$.

\begin{lem}
\label{lem-twistexist}
\marpar{lem-twistexist}
In Situation \ref{hyp-defn-porc} assume $C = \PP^1$.
Let $g : \PP^1 \to \Kgnb{0, 1}(X/\PP^1, 1)$
be a section. Let $m \geq 1$. Assume we have:
\begin{enumerate}
\item The pull back $g^* T_{\Phi}$ has degree $\geq 0$.
\item The image of $g$ is contained in the unobstructed
locus of $\text{ev}$.
\item The cohomology group $H^1(\PP^1, g^*T_{\text{ev}}(-m))$
is zero.
\item The composition $\text{ev} \circ g : \PP^1 \to X$
is a free section of $X \to \PP^1$.
\end{enumerate}
Then the associated pair $(R, \sigma)$ is
a $m$-twisting ruled surface in $X$.
\end{lem}

\begin{proof}
We will use the identifications of $g^*T_{\Phi} = \sigma^* \OO_R(\sigma)$
and $g^*T_{\text{ev}} = \pi_*N_{R/X}(-\sigma)$ obtained in the proof of
Lemma \ref{lem-twistimplies}.
Note that $R$ is a Hirzebruch surface, and $H^1(R, \OO_R) = H^2(R, \OO_R) =0$.
Combined with the fact that $\sigma^2 = \deg(g^*T_{\Phi}) \geq 0$
we see that $|\sigma |$ is base point free. Also, during the course
of the proof we may assume that $\sigma$ is general in its linear
system on $R$. In particular this means that the section
$1 \in \Gamma(R, \OO_R(\sigma))$ is regular for the coherent
sheaf $N_{R/X}$. (Note that this is automatic in the case, which
always holds in practice, that $R \to X$ is a closed immersion.)

\medskip\noindent
The fact that $\text{ev} \circ g : \PP^1 \to X$
is $1$-free means that $g^* \text{ev}^* T_{X/\PP^1}$ is
globally generated. The fact that $H^1(\PP^1, g^*T_{\text{ev}}) = 0$
means that any infinitesimal deformation of the morphism
$\text{ev} \circ g : \PP^1 \to X$ can be followed
by an infinitesimal deformation of $g : \PP^1 \to
\Kgnb{0,1}(X/C, 1)$. In terms of the pair $(R, \sigma)$
this means that the image of
$\alpha : H^0(R, N_{R/X}) \to H^0(\PP^1, \sigma^*N_{R/X})$
contains the image of $\beta : H^0(\PP^1, (\text{ev} \circ g)^*
T_{X/\PP^1}) \to H^0(\PP^1, \sigma^*N_{R/X})$. In this
way we conclude that $N_{R/X}$ is at least globally generated over
the image of $\sigma$.

\medskip\noindent
This weak global generation result in particular implies that
$R^1\pi_* N_{R/X}(-\sigma) = 0$, and $R^1\pi_* N_{R/X} = 0$;
we can for example see this by computing the cohomology on
the fibres. Thus we see that
$H^1(R, N_{R/X}(-\sigma - \pi^*A)) = H^1(\PP^1, g^*T_{\text{ev}}(-A))$.
This gives us the vanishing of the cohomology group
$H^1(R, N_{R/X}(-\sigma - \pi^*A))$
for any divisor $A$ of degree $\leq m$ on $\PP^1$.
We also get
$H^1(R, N_{R/X}) = H^1(\PP^1, \pi_*N_{R/X})$, and
an exact sequence
$0 \to \pi_* N_{R/X}(-\sigma) \to \pi_* N_{R/X} \to \sigma^*N_{R/X} \to 0$.
The first sheaf being identified with
$g^*T_{\text{ev}}$ and the second being globally generated
we conclude that $H^1(R, N_{R/X}) = 0$.
Note that, with $m > 0$ this argument actually also implies that
$H^1(R, N_{R/X}(-R_t)) = 0$ for any $t \in C(k)$.
At this point, what is left, is to show that $N_{R/X}$ is
globally generated. It is easy to show that a coherent
sheaf on $\PP^1$ which is globally generated at a point is
globally generated. Since $m \geq 1$ the map
$H^0(R, N_{R/X}) \to H^0(R_t, N_{R/X}|_{R_t})$ is
surjective by the vanishing of cohomology we just
established. Combined these imply that $N_{R/X}$
is globally generated.
\end{proof}

\noindent
Here is the definition we promised above.

\begin{defn}
\label{defn-twistabs}
\marpar{defn-twistabs}
Let $k$ be an algebraically closed field of characteristic $0$.
Let $Y$ be projective smooth over $k$, and let $\mc{L}$ be an
ample invertible sheaf on $Y$. A \emph{scroll} on $Y$
is given by morphisms $\PP^1 \leftarrow R \rightarrow Y$
such that (a) $R$ is a smooth projective surface, (b)
all fibres $R_t$ of $R \to \PP^1$ are nonsingular rational
curves, and (c) the induced maps $R_t \to Y$ are lines in $Y$.
Let $R$ be a scroll in $Y$ and suppose $D \subset R$ is a Cartier divisor.
For every nonnegative integer $m$, we say $(R,D)$ is an
\emph{$m$-twisting scroll} in $Y$ if the diagram
$$
\xymatrix{
R \ar[r] \ar[d] & \PP^1 \times Y \ar[d]^{\text{pr}_1} \\
\PP^1 \ar@{=}[r] & \PP^1
}
$$
and the divisor $D$ form
an $m$-twisting ruled surface with respect to the invertible
sheaf $\text{pr}_2^*\mc{L}$.
If $m$ is at least $2$, an $m$-twisting scroll will also be called a
\emph{very twisting scroll}.
\end{defn}

\noindent
The main observation of this section is that if
$X \to C$ is as in Situation \ref{hyp-defn-porc},
and Hypothesis \ref{hyp-peace} holds over
an open part of $C$,  and if the geometric generic fibre
has a very twisting surface, then $X \to C$ contains very
twisting ruled surfaces.

\begin{lem}
\label{lem-strongtwist}
\marpar{lem-strongtwist}
In Situation \ref{hyp-peaceY}, assume that $Y$ has a very
twisting scroll. Then there exist $m$-twisting
scrolls $(\pi : R \to \PP^1, h : R \to Y, \sigma)$
such that $\pi_* \OO_R(D)$ and $\pi_* N_{R/\PP^1\times Y}$
are ample locally free sheaves on $\PP^1$ with
$m$ arbitrarily large. We can
further arrange it so that $h \circ\sigma$ corresponds
to an unobstructed point of one of the irreducible
components $Z_e$ defined in Section \ref{sec-varlines}
\end{lem}

\begin{proof}
Let $(R, D)$ be an $m$-twisting scroll on $Y$.
Choose a section $\sigma$ of $R$ in $|D|$ and
think of the associated morphism $g : \PP^1 \to
\Kgnb{0,1}(Y, 1)$ as in Lemmas \ref{lem-twistimplies}
and \ref{lem-twistexist}. In fact Lemma \ref{lem-twistimplies}
implies that the morphism $g$ satisfies
the assumptions of Lemma \ref{lem-twistexist}.
It is clear that the assumptions of Lemma \ref{lem-twistexist}
hold on an open subspace of $\text{Mor}(\PP^1, \Kgnb{0,1}(Y, 1)$.
Also, a morphism $g$ satisfying those assumptions
is free. Given two morphisms $g_i : \PP^1
\to \Kgnb{0,1}(Y, 1)$, $i=1,2$ corresponding to $m_i$-twisting
surfaces with $g_1(\infty) = g_2(0)$ there exists a smoothing
(see \cite[II Definition 1.10]{K}) of
$g_1 \cup g_2 : \PP^1 \cup_{\infty \sim 0} \PP^1 \to \Kgnb{0,1}(Y,1)$
whose general fibre is a morphism $g : \PP^1 \to \Kgnb{0,1}(Y,1)$
which corresponds to an $(m_1 + m_2 - 1)$-twisting scroll.

\medskip\noindent
We conclude that $3$-twisting surfaces exist; by glueing
$2$ $2$-twisting surfaces which exist by the assumption
on $Y$ and smoothing as above. If $(R \to \PP^1, R \to Y, D)$
is $m$-twisting with $m \geq 3$, then $(R \to \PP^1, R \to Y, D + R_0)$
is $(m-1)$-twisting and has the property that $\pi_* \OO_R(D + R_0)$
is ample.

\medskip\noindent
Fix $m \geq 3$ such that an $m$-twisting surface exist.
Consider a large integer $N$ and consider maps
$g_i : \PP^1 \to \Kgnb{0,1}(Y, 1)$, $i=1,\ldots, N$
each corresponding to an $m$-twisting surface
$(\pi_i : R_i \to \PP^1, h_i : R_i \to X, \sigma_i)$, and such that
$g_i(\infty) = g_{i+1}(0)$ for $i=1,\ldots,N-1$.
We may assume that all the free rational curves
$h_i \circ \sigma_i$ lie in the same irreducible component
$Z'$ of $\Kgnb{0,0}(Y, e_0)$, for example by taking the same
twisting scroll for each $i$ (with coordinate on $\PP^1$
reversed for odd indices $i$). We may also assume that
$N \geq E - e_0$ where $E$ is as in assertion (\ref{samecomp})
in Section \ref{sec-varlines}.
Since each $g_i$ is free we can find a smoothing
of the stable map $g_1\cup \ldots \cup g_N :
\PP^1 \cup_{\infty \sim 0} \ldots \cup_{\infty \sim 0} \PP^1
\to \Kgnb{0,1}(Y, 1)$. In fact, we may moreover assume
the smoothing of
$\PP^1 \cup_{\infty \sim 0} \ldots \cup_{\infty \sim 0} \PP^1$
has $N$ sections $z_i$ with $z_i$ limiting to a point $t_i'$ on
the $i$th component of the initial chain.
Let $(\pi : R \to \PP^1, h : R \to Y, \sigma)$ correspond to
a general point of the smoothing, and let $t_1,\ldots, t_N
\in \PP^1(k)$ be the values of the sections $z_i$. We claim that
$(R \to \PP^1, h : R \to Y, \sigma + \sum R_{t_i})$ has at
least two of the three desired properties. 

\medskip\noindent
First, since $(R, \sigma)$ is $Nm$-twisting we see
immediately that $(R, \sigma + \sum R_{t_i})$ is $N(m-1)$
twisting.

\medskip\noindent
Second, the Cartier divisor $\sigma + \sum R_{t_i}$ is a deformation
of a divisor on the surface
$$
R_1 \bigcup_{R_{1,\infty} \cong R_{2, 0}} \ldots
\bigcup_{R_{N-1,\infty} \cong R_{N, 0}} R_N
$$
which restricts to the divisor class of
$\sigma_i + R_{i,t_i'}$ on every $R_i$. By our discussion above
we conclude that $\pi_* \OO_R(\sigma + \sum R_{t_i})$ is a deformation
of a locally free sheaf on a chain of $\PP^1$'s which is ample
on each link, hence ample.

\medskip\noindent
Third, we show that a general element of $|\sigma + \sum R_{t_i}|$
corresponds to a point of $Z_e$ with $e = Ne_0 + N$.
Consider the comb consisting of
$h : \sigma(\PP^1) \cup \bigcup R_{t_i} \to Y$.
It suffices to show that this defines a unobstructed
point of $\Kgnb{0,0}(Y, e)$ which is in $Z_e$.
It is unobstructed because $(R, \sigma)$ is a twisting surface.
By construction our comb is a ``partial smoothing'' of the
tree of rational curves
$$
\xymatrix{
(\PP^1 \cup R_{1, t_i'})
\cup_{\infty \sim 0}
\ldots
\cup_{\infty \sim 0}
(\PP^1 \cup R_{N, t'_N})
\ar[rrr]^-{h_i \circ \sigma_i \cup h_i|_{R_{i, t'_i}}}
&
&
&
Y.
}
$$
Again since all fibres of twisting surfaces are free lines,
and since each $\sigma_i$ is free this stable map is
unobstructed. At this point we may apply Lemma \ref{lem-Zdabs}
to conclude that this stable map is in $Z_e$.

\medskip\noindent
At this point it is not yet clear that $\pi_*N_{R/X}$ is ample.
Let $\sigma'$ be a section of $R$ representing a general
element of $\sigma + \sum R_{t_i}$. We just proved that
$\sigma'$ is a point of $Z_e$. Let $g' : \PP^1 \to
\Kgnb{0,1}(Y, 1)$ be the morphism corresponding
to the twisting surface $(R, \sigma')$. We saw above
that $g'$ corresponds to a twisting surface.
Since $(g')^*T_{\text{ev}}$
has no $H^1$ there are no obstructions to lifting a
given deformation of $\text{ev} \circ g' = \sigma'$ to
a deformation of $g'$. Hence we may assume that
$\sigma'$ is a general point of $Z_e$ and in particular
we may assume that $\sigma'$ is very free, see
Section \ref{sec-varlines} (\ref{better}).
At this point consider the exact sequence of
locally free sheaves on $\PP^1$
$$
0
\to
\pi_* N_{R/\PP^1 \times Y}(-\sigma') 
\to
\pi_* N_{R/\PP^1 \times Y}
\to
\sigma^* N_{R/\PP^1 \times Y}
\to
0.
$$
The sheaf on the left hand side is ample because
$(R,\sigma')$ is $N(m-1)$-twisting. Combined with
the surjective map $(\sigma')^*T_Y \to \sigma^* N_{R/\PP^1\times Y}$
and the ampleness of $(\sigma')^*T_Y$
this proves the result.
\end{proof}

\noindent
In order to formulate the next lemma, let us
call a very twisting scroll (as in the lemma above)
$(\pi : R \to \PP^1, h : R \to Y, \sigma)$
\emph{wonderful} if $\pi_* N_{R/\PP^1\times Y}$
is ample, $\pi_* \OO_R(\sigma)$ is ample, and
$h \circ \sigma$ belongs to the irreducible component
$Z_e$.

\begin{lem}
\label{lem-twistlinesopen}
\marpar{lem-twistlinesopen}
In Situation \ref{hyp-defn-porc}, assume that
Hypothesis \ref{hyp-peace} holds over a nonempty
open $S \subset C$. Assume that for some
$t \in C(k)$ the fibre $X_t$ has a very twisting scroll.
Then there exist:
\begin{enumerate}
\item a smooth variety $B$ over $k$,
\item a flat morphism $\underline{t} : B \to C$,
\item a smooth projective family of surfaces $\mathcal{R} \to B$,
\item a morphism $\pi : \mathcal{R} \to \PP^1$,
\item a morphism $h : \mathcal{R} \to X$ such that $f \circ h =
\underline{t}$, and
\item a morphism $\sigma : \PP^1 \times B \to \mathcal{R}$ over $B$
such that $\pi \circ \sigma = \text{pr}_1$. 
\end{enumerate}
These data satisfy:
\begin{enumerate}
\item for each $b\in B(k)$ the fibre $(\pi_b : \mathcal{R}_b \to \PP^1, 
h_b : \mathcal{R}_b \to X_{\underline{t}(b)}, \sigma_b)$ is
a wonderful very twisting scroll in $X_{\underline{t}(b)}$, and
\item the image of the map
$$
\text{fib}_0 : B \longrightarrow \Kgnb{0,1}(X/C, 1),
$$
which assigns to $b\in B(k)$ the $1$-pointed free line
$\sigma_b(0) \in \pi_b^{-1}(0) \to X_{\underline{t}(b)}$
contains a nonempty open $V \subset \Kgnb{0,1}(X/C, 1)$.
\end{enumerate}
\end{lem}

\begin{rmk}
\label{rmk-demystify}
\marpar{rmk-demystify}
The first part of this lemma is just one possible
formulation of what it means to have a family
of wonderful very twisting surfaces. Really the
lemma just claims that once there is one there are many.
\end{rmk}

\begin{proof}
Let $(\PP^1 \leftarrow R \to X_t, D)$ be an $m$-twisting scroll
in a fibre. There are no obstructions to deforming the
morphism $R \to \PP^1 \times X_t$ because
$H^1(R, N_{R/\PP^1\times X_t}) = 0$, see Remark \ref{rmk-technical}. 
Thus there are also no obstructions to deforming the morphism
$R \to \PP^1 \times X$ since this just adds a trivial summand
to the normal bundle and $H^1(R, \OO_R) = 0$ (as $R$ is a 
rational surface). Since the normal bundle of $R \to \PP^1 \times X$
is globally generated we see that we may deform $R \to \PP^1\times X_t$
to a morphism $R' \to \PP^1 \times X_{t'}$ with $t'$ general and
$R' \to X_{t'}$ passing through a general point of $X_{t'}$.
In particular there exist very twisting surfaces in fibres
over $R$ to which Lemma \ref{lem-strongtwist} applies.
Thus we may now assume that $R$ is a wonderful twisting
surface.

\medskip\noindent
Moreover, let $\sigma : \PP^1 \to R$ be a section of
$R \to \PP^1$ such that $D \sim \sigma$ are rationally
equivalent. Pick a point $p \in \PP^1(k)$. Since $m > 0$ the map
$H^0(R, N_{R/\PP^1\times X_t}(-\sigma)) \to
H^0(R_p, N_{R_p/X_t}(-\sigma(p)))$ is surjective.
We see that given any infinitesimal deformation
of the pointed map $(R_p, \sigma(p)) \to (X_t, h(\sigma(p)))$
(from a pointed line to $X_t$ pointed by the image of the point)
in $X_t$ we can find an (unobstructed) infinitesimal deformation
of $R\to \PP^1 \times X_t$ that induces it. 
Combined with the fact, proven above, that we can pass
a deformation of $R$ through a general point of $X$ this
shows that we can deform $R \to X_t$ such that a given
fibre of $R \to \PP^1$ is a general line in $X/C$.

\medskip\noindent
Finally, what is left is to show that in a family of
morphisms of surfaces $\mathcal{R} \to \PP^1 \times X$
the locus where the surface is a wonderful $m$-twisting scroll
is open. This follows from semi-continuity of cohomology
and can safely be left to the reader, although a very
similar and more difficult case is handled in Lemma
\ref{lem-deformtwist} below.
\end{proof}

\noindent
In Situation \ref{hyp-defn-porc}, let $R \to X$ be a ruled surface
all of whose fibres are free lines in $X/C$. Let $D$ be a Cartier
divisor on $R$ of degree $1$ on the fibres of $R \to C$. Let
$t_1,\ldots,t_\delta \in C(k)$ be pairwise distinct points.
Let $\PP^1 \leftarrow S_i \rightarrow X_{t_i}$ be
scrolls. Let $D_i$ be a Cartier divisor on $S_i$ of degree
$1$ on the fibres of $S_i \to \PP^1$. Assume given isomorphisms
$R_{t_i} \cong S_{i, 0}$ of the fibre of
$R$ over $t_i$ with the fibre of $S_i$ over $0$ compatible
with the maps into $X_{t_i}$.
Let $C' = C \cup \bigcup_i \PP^1$ be a copy of $C$ with
$\delta$ copies of $\PP^1$ glued by identifying $0$ in
the $i$th copy with $t_i \in C(k)$. Let
$R' = R \cup \bigcup_i S_i \to C'$
be the ruled surface over $C'$ gotten by gluing $R_{t_i}$
to $S_{i,0}$ using the given isomorphisms.
Let $h' : R' \to C' \times_CX$ be the obvious morphism.
Note that there is a Cartier divisor $D'$ on $R'$ which
restricts to $D$ on $R$ and to $D_i$ on $S_i$.
This Cartier divisor is unique up to rational equivalence.

\medskip\noindent
In the following we are interested in smoothings of
situations as above. This means that we have an irreducible
variety $T$ over $k$, and a commutative diagram of varieties
$$
\xymatrix{
R' = R \cup \bigcup S_i \ar[r] \ar[d]^{\pi_0} &
\mathcal{R} \ar[r] \ar[d]^\pi &
\mathcal{C} \times_{C \times T} (X \times T)  \ar[r] \ar[d] &
X \times T \ar[ld] \ar[d]
\\
C' = C \cup \bigcup \PP^1 \ar[r] & \mathcal{C} \ar[r] & C \times T \ar[r] & T
}
$$
satisfying the following conditions: 
(1) $\mathcal{C}\to T$ and $\mathcal{R} \to T$ are flat
and proper, (2) every fibre $\mathcal{C}_t$ of $\mathcal{C} \to T$
over $t\in T(k)$ is a nodal curve of genus $g(C)$ and $\mathcal{C}_t
\to C$ has degree $1$ (see discussion in Section \ref{sec-ssections}),
(3) the morphism $\mathcal{R} \to \mathcal{C}$ is
smooth, (4) every fibre $\mathcal{R}_t \to \mathcal{C}_t \times_C X$
is a ruled surface over every irreducible component of $\mathcal{C}_t$,
(5) for some point $0 \in T(k)$ the fibre 
$\mathcal{R}_0 \to \mathcal{C}_0 \times_C X$ is isomorphic to our map
$h' : R' \to C' \times_CX$ above, and finally (6) for some point
$t \in T(k)$ the fibre $\mathcal{C}_t = C$.
In addition we assume given a Cartier divisor
$\mathcal{D}$ on $\mathcal{R}$ restricting to the divisor
$D'$ on $R' = \mathcal{R}_0$.

\begin{lem}
\label{lem-deformtwist}
\marpar{lem-deformtwist}
In the situation above. 
\begin{enumerate}
\item If $(R, D)$ is $m$-twisting and
$(S_i, D_i)_{i=1\ldots\delta}$ are $m_i$-twisting, with $m_i \geq 2$.
Then for $t \in T(k)$ general the ruled
surface $(\mathcal{R}_t, \mathcal{D}_t)$ in $X$
is $(m + \delta)$-twisting.
\item Let $n_i$, $i=1,2,3$ be integers such that
for all Cartier divisors $A$ of degree on $C$ we have
\begin{enumerate}
\item $\deg(A) \geq n_1 \Rightarrow H^1(C, (\pi_* N_{R/X})(A)) = 0$,
\item $\deg(A) \geq n_2 \Rightarrow H^1(C, (\pi_* N_{R/X}(-D))(A)) = 0$, and
\item $\deg(A) \geq n_3 \Rightarrow H^1(C, (\pi_* \OO_R(D))(A)) = 0$.
\end{enumerate}
If $(S_i, D_i)_{i=1\ldots\delta}$ are wonderful very twisting scrolls
and $\delta \geq \max\{n_1 + 1, n_2 + 1, n_3\}$, then for $t \in T(k)$ general
the ruled surface $(\mathcal{R}_t, \mathcal{D}_t)$ in $X$
is $(\delta - n_3)$-twisting.
\end{enumerate}
\end{lem}

\begin{proof}
We will not prove the first statement. It follows by a combination of the
methods of the proof of the second statement, and those
in the proof of Lemma \ref{lem-strongtwist}. We just remark that
the increase in the twisting comes from the fact that the sheaves
$(S_i \to \PP^1)_*\big( N_{S_i/\PP^1\times X_{t_i}}(-D_i)\big)$ are ample
vector bundles on $\PP^1$
because $m_i \geq 2$ combined with \cite[II Lemma 7.10.1]{K}.

\medskip\noindent
We begin the proof of the second statement.
Because the morphism $\mathcal{R} \to \mathcal{C}$ is smooth
we can define $N_{\mathcal{R}/X}$ by the short
exact sequence
$$
0 \to T_{\mathcal{R}/\mathcal{C}} \to
(\mathcal{R} \to X)^*T_{X/C} \to
N_{\mathcal{R}/X} \to
0.
$$
As before this is a sheaf on $\mathcal{R}$ which is flat over
$\mathcal{C}$. For every point $c \in \mathcal{C}(k)$ lying over
$t \in C(k)$ the
restriction of $N_{\mathcal{R}/X}$ to the fibre $\mathcal{R}_c$
is the normal bundle of the line $\mathcal{R}_c \to X_t$ (compare
with Remark \ref{rmk-technical}). For all points $c \in \mathcal{C}(k)$
lying over $0$ the sheaf $N_{\mathcal{R}/X}|_{\mathcal{R}_c}$
is globally generated. Hence after replacing $T$ by an open
subset containing $0$ we may assume this is true for all points
$c$.

\medskip\noindent
Consider the coherent sheaves of $\OO_\mc{C}$-modules
$\mathcal{E}_1 = \pi_* \OO_R(\mathcal{D})$,
$\mathcal{E}_2 = \pi_* N_{\mathcal{R}/X}$ and 
$\mathcal{E}_3 = \pi_* N_{\mathcal{R}/X}(-\mathcal{D})$.
By the above the corresponding higher direct images are zero
(as $H^1$ of a globally generated sheaf on $\PP^1$ is zero).
By cohomology and base change this implies $\mathcal{E}_i$ is
a locally free sheaf on $\mathcal{C}$.
Let $U \subset T$ be a nonempty open such that
$\mathcal{C}_t = C$ for all $t \in U(k)$ (this exists
because we started with a smoothing). For $t \in U(k)$
we can and do think of $\mathcal{E}_{i,t}$ as a locally
free sheaf on $C$. Let $A$ be a Cartier divisor of
degree $\leq \delta - n_3$ on $C$, and let $p\in C(k)$ be any point.
We would like to show that
$H^1(C, \mathcal{E}_{1, t}(-p)) = 0$,
$H^1(C, \mathcal{E}_{2, t}(-p)) = 0$, and
$H^1(C, \mathcal{E}_{3, t}(-A)) = 0$, for $t$ general in $U(k)$.
This now follows from \cite[II Lemma 7.10.1]{K}, the definition
of $n_i$ in the lemma and the fact that $\delta \geq n$.

\medskip\noindent
Namely, these vanishings imply that both $\mathcal{E}_{1,t}$,
and $\mathcal{E}_{2,t}$ are globally generated and have
no $H^1$, which implies that $\OO_{\mathcal{R}_t}(\mathcal{D}_t)$
and $N_{\mathcal{R}_t/X}$ are globally generated and have no $H^1$.
The statement on $\mathcal{E}_{3,t}$ gives the desired
vanishing of
$H^1(\mathcal{R}_t, N_{\mathcal{R}_t/X}(-\mathcal{D}_t -\pi_t^*A))$.
\end{proof}

\begin{prop}
\label{prop-twistexist}
\marpar{prop-twistexist}
In Situation \ref{hyp-defn-porc},
assume Hypothesis \ref{hyp-peace} holds over
an open part $S$ of $C$. Suppose a geometric fibre
of $X \to C$ has a very twisting scroll.
Then there exists an irreducible component
$Z \subset \Sigma(X/C/k)$ containing a 
free section, such that for all $e \gg 0$ 
the irreducible component $Z_e$ defined
in Lemma \ref{lem-Zd} contains a point $[s]$
such that $s = h \circ \sigma$ for some
very twisting ruled surface $(h : R \to X, \sigma)$ in $X$.
\end{prop}

\begin{proof}
The proof of this proposition is in two stages.
In the first stage we will find some sufficiently
twisting surface $(R, \sigma)$. The irreducible component
$Z$ will be the unique one containing $[h \circ \sigma]$.
After that, in the second stage we will find
the others by glueing lines and wonderful very
twisting surfaces.

\medskip\noindent
Let $B$, $\underline{t} : B \to C$, $\mathcal{R} \to B$,
$\pi : \mathcal{R} \to \PP^1$, $h : \mathcal{R} \to X$,
$\sigma : \PP^1 \times B \to \mathcal{R}$ be the family
of wonderful very twisting scrolls in fibres of $X/C$
we found in Lemma \ref{lem-twistlinesopen}.
Let $V \subset \Kgnb{0,1}(X/C, 1)$ be the nonempty open
swept out by $\text{fib}_0$ as in that lemma.
In addition, let $e_0$ be the degree of the
rational curves $h_b \circ \sigma_b : \PP^1 \to X_{\underline{t}(b)}$.

\medskip\noindent
By Lemma \ref{lem-existDfree} there exists a free section
$s$ of $X \to C$ meeting $X_{f,\pax}$. We may deform $s$
such that it also meets the open set $\text{ev}(V)$.
We may also deform $s$ such that $s$ defines a point of the open
$U$ of Lemma \ref{lem-freeruled}. In the proof of Lemma
\ref{lem-freeruled} we used the main result of \cite{GHS} 
to establish the existence of a ruled surface $R$ penning
$s$ all of whose fibres are free lines. Hence we may
actually choose $R$ such that for some $t\in C(k)$ general
the one pointed line $s(t) \in R_t$ defines a point of the open $V$.
Here we use that the moduli space of lines in $X_S/S$ meeting
$X_{f,\pax}$ is irreducible by Hypothesis \ref{hyp-peace}. Thus we
may assume that $s$ is penned by a ruled surface $R$
all of whose fibres are free lines in $X/C$, by (b) of 
Lemma \ref{lem-freeruled}, and such that there exists an
open $W \subset C$ with the property that $s(t) \in R_t \to X_t$ defines
a point of $V$ for all $t \in W(k)$.

\medskip\noindent
In other words, at every point of $W$ we may attach
a wonderful very twisting surface out of the family
$\mathcal{R}/B$ to $R$. We reinterpret this in terms
of the corresponding map $g = g_{(R,\sigma)} :
C \to \Kgnb{0,1}(X/C, 1)$. In terms of this it means
that we can find arbitrarly many free maps
$g_\alpha : \PP^1 \to \Kgnb{0,1}(X/C, 1)$ corresponding to wonderful
very twisting surfaces in fibres that we may attach
to the morphism $g$ to get combs $C \cup \bigcup_\alpha \PP^1
\to \Kgnb{0,1}(X/C, 1)$. By \cite[II Theorem 7.9]{K}
an arbitrarily large subcomb will smooth. Now we reinterpret
this back into a smoothing of the corresponding glueing
of ruled surfaces. It says, via part (2) of
Lemma \ref{lem-deformtwist}, that a general point of the base
of this smoothing will correspond to a $(e_0 + 1)$-twisting
surface $(\pi : R \to C, h : R \to X , \sigma)$.
This finishes the proof of the first stage.

\medskip\noindent
Let $e_1$ be the degree of $h \circ \sigma$, and let
$Z \subset \Sigma^{e_1}(X/C/k)$ (as promised) be the
unique irreducible component containing the free section
$h \circ \sigma$. Fix an integer $i \in \{0, 1, \ldots, e_0-1\}$. Pick
some effective Cartier divisor $\Delta_i \subset C$ of
degree $i$, and pick a section $\sigma_i$ of $R \to \PP^1$
which is a member of the (base point free) linear system
$|\sigma + \pi^*\Delta_i|$. Then
$(R, \sigma_i)$ is $(e_0 + 1 - i)$-twisting. Thus
the associated morphism $g_i : C \to \Kgnb{0,1}(X/C, 1)$
is free, see Lemma \ref{lem-twistimplies}. In particular
after deforming $g_i$ a bit we may assume that $g_i$ meets
the open dense subset $V$ above.

\medskip\noindent
For any $\delta > 0$ consider pairwise distinct point
$t_1, \ldots, t_\delta \in C(k)$ such that $g_i(t_j) \in V$
for all $j =1, \ldots, \delta$. In addition let $g_{ij} :
\PP^1 \to \Kgnb{0,1}(X_{t_j},  1)$ be morphisms corresponding
to wonderful very twisting surfaces out of our family
$\mathcal{R}/B$ above with the property that $g_{ij}(0) =
g_i(t_j)$ (this is possible). In other words, now we have
a comb $\tilde g_i : C \cup \bigcup_j \PP^1 \to \Kgnb{0,1}(X/C, 1)$.
Note that the resulting stable section $\text{ev} \circ \tilde g_i$
of $X/C$ has degree $e_1 + i + \delta e_0$. Since both
$g_i$ and all the $g_{ij}$ are free there exists a smoothing
of this comb (for all $\delta \geq 0$). By part (1) of Lemma
\ref{lem-deformtwist} a general point of the base of this smoothing
corresponds to a $(e_0 + 1 - i + \delta)$-twisting surface. 

\medskip\noindent
It remains to show that (for every $i$) the stable section
$\text{ev} \circ \tilde g_i$ corresponds to a point of
$Z_{e_1 + i + \delta e_0}$ (since then the same will be
true after smoothing the surface). This follows from 
Lemma \ref{lem-Zdgluemore}.
\end{proof}

\begin{cor}
\label{cor-porcline}
\marpar{cor-porcline}
Assumptions and notations as in the Proposition
\ref{prop-twistexist} above, with $Z$, $Z_e$ as in
that proposition. Let $P_{e,1}$ be the moduli
space of porcupines $(s : C \to X, q \in C(k),
r \in L(k), L \to X_q)$ with one quill
and body $s$ whose moduli point is in $Z_e$.
There exists a dense open $U_{e,1} \subset P_{e,1}$
such that if $[(s : C \to X, q \in C(k),
r \in L(k), L \to X_q)]$ lies in $U_{e,1}$
then there exists a ruled surface
$R$ in $X$ penning $s(C) \cup L$ (see Definition \ref{defn-pen})
such that $(R, s)$ is very twisting.
\end{cor}

\begin{proof}
Although the formulation of this corollary is cumbersome
the proof is not. Namely, we already know that for all
$e \gg 0$ there exists a very twisting ruled surface
$(\pi : R \to C, h : R \to X, \sigma)$ such that
$h \circ \sigma$ lies in $Z_e$. Pick a general $q \in C(k)$
and consider the datum $(q \in C(k), \pi : R \to C, h : R \to X, \sigma)$.
After moving the ruled surface $(R, \sigma)$ a bit we
may assume $h(\sigma(q))$ is a point of $X_{f,\pax}$ such
that $(h \circ \sigma)(C) \cup h(R_q)$ is a porcupine.
(This step is not strictly necessary for the proof.)
At this point it suffices to show that any deformation
of the porcupine can be followed by a deformation of
the datum $(q \in C(k), \pi : R \to C, h : R \to X, \sigma)$.
This is true because $H^1(R, N_{R/X}(-\sigma - R_q)) = 0$
as $(R,\sigma)$ was assumed very twisting.
\end{proof}

%%%%%%%%%%%%%%%%%%%%%%%%%%%%%%%%%%%%%%%%%%%%%%%%%%%%%%%%%%%%%%%
%%
%% Main theorems
%% 
%%%%%%%%%%%%%%%%%%%%%%%%%%%%%%%%%%%%%%%%%%%%%%%%%%%%%%%%%%%%%%%

\section{Main theorem}
\label{sec-proofs} 
\marpar{sec-proofs}

\noindent
Here is the main theorem of this paper.

\begin{thm}
\label{thm-main}
\marpar{thm-main}
In Situation \ref{hyp-defn-porc}, assume Hypothesis \ref{hyp-peace}
holds over an open part $S$ of $C$. Suppose a geometric fibre
of $X \to C$ has a very twisting scroll. In this case there exists
a sequence $(Z_e)_{e\gg 0}$ of irreducible components $Z_e$ of
$\Sigma^e(X/C/\kappa)$ as in Lemma \ref{lem-Zd} such that 
\begin{enumerate}
\item a general point of $Z_e$ parametrizes a free section of $f$,
\item each Abel map
$$
\alpha_{\mc{L}}|_{Z_e}:Z_e \rightarrow \text{Pic}^e_{C/\kappa}
$$
has nonempty rationally connected geometric
generic fibre, and
\item for every free section $s$ of $f$ there exists an
$E = E(s) > 0$ such that for all $e\geq E$, 
$Z_e$ is the unique irreducible component of
$\Sigma^e(X/C/\kappa)$ containing every comb whose handle is $s$
and whose teeth are free lines in fibres of $f$.
\end{enumerate}
\end{thm}

\begin{proof}
We will show the sequence of irreducible components
we found in Proposition \ref{prop-twistexist} satisfies
the conclusions of the theorem. By Lemma \ref{lem-Zdgenpt}
the irreducible components satisfy the first condition.
As a second step we remark that $\{Z_e\}$ satisfies the
third part by Corollary \ref{cor-Zdsame}. Also, the geometric
generic fibres of $\alpha_{\mc{L}}|_{Z_e}$ are irreducible by
Lemma \ref{lem-Abel-irred} (for large enough $e$).

\medskip\noindent
So what is left is to show that these fibres are
rationally connected. Since they are proper it is the
same as showing they are birationally rationally connected.
We will prove this for a general fibre
rather than the geometric generic fibre, which is anyway the same
thing since our field is uncountable. To show that a variety
$W / k$ is birationally
rationally connected it suffices to show that for a general pair
of points $(p,q) \in W \times W$ there exist open subsets
$V_i \subset \PP^1$, morphisms
$f_i : V_i \to W$, for $i=1,\ldots, N$ such that
(1) each $V_i$ nonempty open, (2) $p \in f_1(V_1)$, (3) $q\in f_N(V_N)$
and (4) $f(V_i)\cap f_{i+1}(V_{i+1})$ contains a {\it smooth} point of
$W$ for $i=1,\ldots,N-1$. See \cite[Theorem IV.3.10.3]{K}. This is
the criterion we will use. During the proof of the theorem we
will call such a chain of rational curves a ``useful chain''.
We will show that there are sufficiently
many rational curves available to do the job.
(More technically precise and specific would
be to say that we will use this criterion
for the coarse moduli scheme for the stack $Z_e$.)

\medskip\noindent
A first point is to note that since the target of
$\alpha_{\mc{L}}$ is an abelian variety every rational curve
in $Z_e$ is automatically contracted to a point under
$\alpha_{\mc{L}}$. Hence we do not have to worry about our
curves ``leaving'' the fibre.

\medskip\noindent
Second, what does Corollary \ref{cor-porcline} say? Let $E$ be the
implied (large) constant of Corollary \ref{cor-porcline}. It says that
for all $e - 1 \geq E$ there is a nonempty open
$U_{e-1 ,1} \subset Z_e$ of the space of porcupines of total degree 1 and
one quill which are pinned by a very twisting surface.
By Lemma \ref{lem-whytwist} all points of $U_{e-1, 1}$ are connected
by a rational curve to a point in the interior of $Z_e$. Now since
the points of $U_{e - 1, 1}$ are unobstructed, we conclude:
(A) a general point of $Z_e$ is connected by a rational point to
a general point of $U_{e-1, 1}$. 

\medskip\noindent
Thirdly, we repeat the previous argument several times, as follows.
Consider the forgetful map $U_{e - 1, 1} \to Z_{e - 1} \times C$
which omits the quill but remembers the attachment point.
It has rationally connected smooth projective fibres at least over
a suitable open of $Z_{e-1} \times C$, by Hypothesis \ref{hyp-peace}
which holds over $S \subset C$. By (A), if $e - 2 \geq E$, then we
can find a rational curve connecting a general point of $Z_{e - 1}$
to a general point of $U_{e - 2, 1}$, and by the previous remark
we can find a rational curve in $U_{e - 1, 1}$ connecting a general
point of $U_{e - 1, 1}$ to a point of $Z_e$ which corresponds to 
a general porcupine with $2$ quills. In other words, if $e - 2 \geq E$,
we can find a useful chain of rational curves connecting a general point of
$Z_e$ to a general porcupine with two quills.

\medskip\noindent
Continuing in this manner, we see that if $e - \delta \geq E$, then
we can connect a general point of $Z_e$ by a useful chain of rational
curves to a general porcupine with $\delta$ quills. It is mathematically
more correct to say that we proved that useful chains emanating
from points corresponding to porcupines whose body is in
$Z_{e - \delta}$ and with $\delta$ quills sweep out an open in $Z_e$.

\medskip\noindent
Fix $e_0$ an integer such that (a) $e_0 > E$, and (b) the general
point of $Z_{e_0}$ corresponds to a $2g(C)+1$-free section, see
Lemma \ref{lem-Zdbetter}. Let $P \subset Z_{e_0}$ be
the set of points corresponding to porcupines without quills.
Recall (see just above Proposition \ref{prop-pentogether})
that $P_{e- e_0}$ indicates porcupines with $e-e_0$ quills
and body in $P$, in other words porcupines with body in $Z_{e_0}$
and $e-e_0$ quills. We apply Proposition \ref{prop-pentogether} to the pair
$(P, P' = P)$; let $E_1$ be the constant found in that
propostion. The proposition says exactly that a general
point of
$$
P_{e - e_0}
\times_{\alpha_{\mc{L}}, \underline{\text{Pic}}^e_{C/k},\alpha_{\mc{L}}}
P_{e - e_0}
$$
can be connected by useful chains in $Z_e$. Combined with the assertion
above that a general point of $Z_e$ may be connected by useful chains
to a point of $P_{e - e_0}$ we win.
\end{proof}

\noindent
The impetus of a lot of the research in this paper comes from
the following application of the main theorem.

\begin{cor}
\label{cor-main}
\marpar{cor-main}
Let $k$ be an algebraically closed field of characteristic $0$.
Let $S$ be a smooth, irreducible, projective surface over $k$.
Let
$$
f:X\rightarrow S
$$ 
be a flat proper morphism. Assume there exists a
Zariski open subset $U$ of $S$ whose complement has codimension $2$
such that the total space $f^{-1}(U)$ is smooth and such that
$f : f^{-1}(U) \to U$ has geometrically irreducible fibres.
Let $\mc{L}$ be an $f$-ample invertible sheaf on $f^{-1}(U)$.  
Assume in addition that the geometric generic fibre
$X_{\overline{\eta}} \to \overline{\eta}$
satisfies Hypothesis \ref{hyp-peace} and has a very twisting surface.
Then there exists a rational section of $f$.  
\end{cor}

\begin{proof}
Right away we may assume that $k$ is an uncountable algebraically
closed field. Also, if $X_{\overline{\eta}} \to \overline{\eta}$
satisfies \ref{hyp-peace} and has a very twisting surface,
then for some nonempty open $V \subset U$ the morphism $X_v \to V$
satisfies \ref{hyp-peace} and the fibres contain very twisting
surfaces. In addition we may blow up $S$ (at points of $U$) and
assume we have a morphism
$$
b : S \longrightarrow \PP^1
$$
whose general fibre is a smooth irreducible projective curve.
Let $W \subset \PP^1$ be a nonempty Zariski open such that
all fibres of $b$ over $W$ are smooth, and meet the open $V$. 
Consider the space and moduli map
$$
\Sigma^e(X_W / S_W / W)
\longrightarrow
\underline{\text{Pic}}^e_{S_W/W}
$$
introduced in Definition \ref{defn-relstablemap} and Lemma
\ref{lem-Abel}.

\medskip\noindent
Let $\xi \in W$ be the generic point so its residue field
$\kappa(\xi)$ is equal to the function field $k(W)$ of $W$.
Let $\overline{\xi} = \text{Spec}(\overline{k(W)}) \to W$ be
a geometric point over $\xi$ corresponding to an algebraic
closure of $k(W)$.
Note that by our choice of $W$ and the assumption of the
corollary the hypotheses of the main theorem (\ref{thm-main})
are satisfied for the morphisms $X_w \to S_w$ for every geometric
point $w$ of $W$. Let us apply this to the geometric generic
point $\overline{\xi}$ of $W$. This gives a sequence of
irreducible components
$Z_e  \subset \Sigma^e(X_W / S_W / W)_{\overline{\xi}}$
say for all $e \geq e_0$. Note that $Z_{e_0}$ can be
defined over a finite Galois extension $\kappa(\xi) \subset L$.
By construction of the sequence $\{Z_e\}_{e \geq e_0}$,
see Lemma \ref{lem-Zd}, we see that each $Z_e$ is defined over $L$. 
Next, suppose that $\sigma \in \text{Gal}(L/\kappa(\xi))$.
Then we similarly have a sequence $\{Z_e^\sigma\}_{e \geq e_0}$
deduced from the sequence by applying $\sigma$ to the coefficients
of the equations of the original $Z_e$. On the other hand,
by Corollary \ref{cor-Zdsame} we have $Z_e = Z_e^\sigma$ for
all $e \gg 0$ (this also follows directly from the third
assertion of the main theorem if you think about it right).
Since $\text{Gal}(L/\kappa(\xi))$ is finite, we conclude that
for $e \gg 0$ the irreducible components $Z_e \subset 
\Sigma^e(X_W / S_W / W)_{\overline{\xi}}$ are the geometric
fibres of irreducible components $Z_{W, e} \subset 
\Sigma^e(X_W / S_W / W)$!

\medskip\noindent
The conclusion of the above is that we know that for all
$e \gg 0$ there exist irreducible components
$Z_{W, e} \subset \Sigma^e(X_W / S_W / W)$ such that
the induced morphisms
$$
Z_{W, e}
\longrightarrow
\underline{\text{Pic}}^e_{S_W/W}
$$
have birationally rationally connected nonempty geometric
generic fibre. But note that since $\Sigma^e(X_W / S_W /W)$
is actually proper over $W$, so is $Z_{W,e}$ and hence its
fibres over $\underline{\text{Pic}}^e_{S_W/W}$ are
actually rationally connected (more precisely the underlying
coarse moduli spaces of the fibres are rationally connected).

\medskip\noindent
At this point we are done by the lemmas below. Namely,
Lemma \ref{lem-ratsec} says it is enough to construct 
a rational section of $Z_{W, e} \to W$, and even such
a rational section to its coarse moduli space
$Z_{W, e}^{\text{coarse}}$. Since the original surface
$S$ was projective, we can find for arbitrarily large
$e$ divisors $D$ on $S$ which give rise to sections
of $\underline{\text{Pic}}^e_{S_W/W} \to W$. On the
other hand, since $\underline{\text{Pic}}^e_{S_W/W}$
is smooth, Lemma \ref{lem-curveup} guarantees, because
of the result on the geometric generic fibre above,
that we can lift this up to a section of
$Z_{W, e}^{\text{coarse}} \to W$. Thus the proof of
the corollary is finished.
\end{proof}

\begin{lem}
\label{lem-ratsec}
\marpar{lem-ratsec}
Let $k$ be an algebraically closed field of characterstic zero.
Let $S$ be a variety over $k$. Let $C \to S$ be
a family of smooth projective curves. Let $X \to C$
be a proper flat morphism. Let $\mc{L}$
be an invertible sheaf on $X$ ample w.r.t.\ $X \to C$.
If $\Sigma^e(X/C/S) \to S$ (see Definition \ref{defn-relstablemap})
has a rational section for some $e$, then $X \to C$ has a rational
section. The same conclusion holds if we replace $\Sigma^e(X/C/S)$
by its coarse moduli space.
\end{lem}

\begin{proof}
After shrinking $S$ we have to show that if there
is a section, then $X \to C$ has a rational section.
Let $\tau : S \to \Sigma(X/C/S)$ be a section.
This corresponds to a proper flat family of
curves $\mc{C} \to S$, and a morphism
$\mc{C} \to X$ such that $\mc{C} \to C$
is a prestable map of degree $1$. See Section
\ref{sec-ssections} and the initial discussion
in that section. In particular there is
a nonempty Zariski open $\mc{U} \subset \mc{C}$
such that $\mc{U} \to \mc{C} \to X \to C$ is
an open immersion. (For example, $\mc{U}$ is the
open part where the morphism $\mc{C} \to C$ is flat,
or it is the part you get by excising the vertical
pieces of the stable maps.)
This proves the first assertion
of the lemma.

\medskip\noindent
To prove the final assertion, after shrinking $S$ as before,
let $t : S \to \Sigma(X/C/S)^{\text{coarse}}$ be a map
to the coarse moduli space. After shrinking $S$ some more
we may assume there is a surjective finite etale Galois
morphism $\nu : S' \to S$ such that $t \circ \nu$
corresponds to a morphism $\tau' : S' \to \Sigma(X/C/S)_S'
= \Sigma(X_{S'}/C_{S'}/S')$. By the arguments in the
first part of the proof this corresponds to a section
of $X_{S'} \to C_{S'}$. To see that it descends to a rational
section of $X \to C$ it is enough to see that it is
Galois invariant. This is clear because the only stackiness
in the spaces of stable maps comes from the contracted
componets.
\end{proof}

\begin{lem}
\label{lem-curveup}
\marpar{lem-curveup}
Let $\alpha : Z \to P$ be a proper morphism of varieties over an algebraically
closed field $k$ of characteristic zero. Assume the geometric
generic fibre is rationally connected, and $P$ is quasi-projective.
Let $C \subset P$ be any curve meeting the nonsingular locus of $P$.
Then there exists a rational section of $\alpha^{-1}(C) \to C$.
\end{lem}

\begin{proof}
(Sketch only, see also similar arguments in \cite{GHMS}.)
Curve means irreducible and reduced closed subscheme of
dimension 1. Thus $C$ is generically nonsingular and
meets the nonsingular locus of $P$. This means $C$ is locally 
at some point a complete intersection in $P$. This means
that globally $C$ is an irreducible component of multiplicity 1 of a 
global complete intersection curve on $P$ (here we use that
$P$ is quasi-projective). This means that there exists a
$1$-parameter flat family of complete intersections
$\mathcal{C} \subset P\times T$ parametrized
by a smooth irreducible curve $T/k$ such that
$C$ is an irreducible component of multiplicity 1 of
a fibre $\mathcal{C}_t$ for some $t \in T(k)$
and such that the general fibre
is a smooth irreducible curve passing through a general
point of $P$. In other words, by \cite{GHS} and the
assumption of the lemma, we can find a lift of
$\mathcal{C}_{\overline{k(T)}} \to P$ to a
morphism $\gamma : \mathcal{C}_{\overline{k(T)}} \to X$ into $X$.
Note that if $T' \to T$ is a nonconstant
morphism of nonsingular curves over $k$ and $t' \in T'(k)$ maps to $t \in T(k)$
then the pullback family $\mathcal{C}' = \mathcal{C}\times_T T'$
still has $C$ as a component of multiplicity one of
the fibre over $t'$. Using this, we may assume that $\gamma$
is defined over $k(T)$. In other words we obtain a rational
map $\mathcal{C} \supset \mathcal{W} \xrightarrow{\gamma} X$ lifting
the morphism $\mathcal{C} \to P$. Since $T$ is nonsingular,
and since $C$ is a component of multiplicity one of
the fibre of $\mathcal{C} \to T$ we see that the generic
point of $C$ corresponds to a codimension $1$ point $\xi$ of 
$\mathcal{C}$ whose local ring is a DVR. Because $X \to P$
is proper we see, by the valuative criterion of properness
that we may assume that $\gamma$ is defined at $\xi$.
This gives the desired rational map from $C$ into $X$.
\end{proof}

%%%%%%%%%%%%%%%%%%%%%%%%%%%%%%%%%%%%%%%%%%%%%%%%%%%%%%%%%%%%%%%
%%
%% Special rational curves on flag varieties
%% 
%%%%%%%%%%%%%%%%%%%%%%%%%%%%%%%%%%%%%%%%%%%%%%%%%%%%%%%%%%%%%%%

\section{Special rational curves on flag varieties} 
\label{sec-He}
\marpar{sec-He}

\noindent
In this section we show that flag varieties contain many
good rational curves. In characteristic $p > 0$ there
exist closed subgroup schemes $H$ of reductive groups
$G$ such that $G/H$ is projective, but $H$ not reduced.
To avoid dealing with these, we say $P \subset G$ is
a {\it standard parabolic subgroup} if $P$ is geometrically
reduced and $G/P$ is projective.

\begin{thm}
\label{thm-He}
\marpar{thm-He}
Let $k$ be an algebraically closed field.  Let $G$ be a connected,
reductive algebraic group over $k$. Let $P$ be a
standard parabolic subgroup
of $G$.  There exists a $k$-morphism $ f:\PP^1 \rightarrow G/P $
such that for every standard parabolic subgroup $Q$ of $G$ containing $P$,
denoting the projection by
$$
\pi: G/P \rightarrow G/Q,
$$ 
we have that the pull back of the relative tangent bundle
$f^* T_\pi$ is ample on $\PP^1$.
\end{thm}

\begin{rmk}
\label{rmk-He}
\marpar{rmk-He}
Note that for any \emph{fixed} $Q$, since the fibres of $\pi$ are
rationally connected there exists a morphism $ f:\PP^1 \rightarrow G/P $
such that $f^* T_\pi$ is ample.  The importance of Theorem \ref{thm-He}
is that $f$ has this property for \emph{all} parabolics $Q$
simultaneously. 
\end{rmk}

\newcommand{\cb}{\mathcal B}
\newcommand{\Lie}{\text {\rm Lie}}
\newcommand{\Ad}{\text {\rm Ad}}
\newcommand{\kk}{\mathbf k}

\def\ge{\geqslant}
\def\le{\leqslant}
\def\a{\alpha}
\def\b{\beta}
\def\c{\chi}
\def\g{\gamma}
\def\G{\Gamma}
\def\d{\delta}
\def\D{\Delta}
\def\L{\Lambda}
\def\e{\epsilon}
\def\et{\eta}
\def\io{\iota}
\def\o{\omega}
\def\p{\pi}
\def\ph{\phi}
\def\ps{\psi}
\def\r{\rho}
\def\s{\sigma}
\def\t{\tau}
\def\th{\theta}
\def\k{\kappa}
\def\l{\lambda}
\def\z{\zeta}
\def\v{\vartheta}
\def\x{\xi}
\def\i{^{-1}}

\noindent
The purpose of this section is to prove Theorem \ref{thm-He}. We use some
results in the theory of reductive linear algebraic groups. It turns out that
the proof in the case $X = G/B$, is somewhat easier. Thus we first
reduce the general case to this case by a simple geometric argument.

\begin{lem}
\label{lem-reduceBorel}
\marpar{lem-reduceBorel}
It suffices to prove Theorem \ref{thm-He} when $P$ is
a Borel subgroup.
\end{lem}

\begin{proof}
If $P$ is not a Borel, then pick any $B \subset P$ Borel
contained in $P$. Let $g : \PP^1 \to G/B$ be a morphism
as in Theorem \ref{thm-He}. We claim that the composition
$f$ of $g$ with the projection map $G/B \to G/P$ works.
Namely, if $P \subset Q$ is another parabolic, then we
have a commutative diagram
$$
\xymatrix{
G/B \ar[r] \ar[d]^{\tau} & G/P \ar[d]^\pi \\
G/Q \ar@{=}[r] & G/Q }
$$
which proves that $f^* T_{\pi}$ is a quotient
of $g^* T_{\tau}$, and hence we win.
\end{proof}

\noindent
We review some of the notation of \cite[Section 2]{MR}.
Let $G$ be a connected reductive algebraic group over an algebraically
closed field $k$. Choose a maximal torus $T \subset G$ and let
$T \subset B$ be a Borel subgroup of $G$ containing $T$. Let $B^-$ be the
opposite Borel subgroup, so that $T=B \cap B^-$. Let $U$ (resp. $U^-$) be
the unipotent radical of $B$ (resp. $B^-$).
Let $X=X^\ast(T)=\Hom(T,{\mathbf G}_m)$ and
$Y=X_\ast(T)=\Hom({\mathbf G}_m,T)$ be the
character and cocharacter groups of $T$. Let
$( , ): X \times Y \to \ZZ$ be the pairing defined by
$x(y(t))=t^{(x, y)}$ for $t \in {\mathbf G}_m$.
The choice of $T \subset B \subset G$ determines the set of roots
$\Phi \subset X$, the set of positive roots $\Phi^+$, as well
as its complement $\Phi^-$. For $\a \in \Phi$, let $U_{\a}$ be the
root subgroup corresponding to $\a$. Our normalization is that
$\a \in \Phi^+ \Leftrightarrow U_\a \subset B$. We denote the simple
roots $\Delta = \{ \a_i \mid i\in I\}$. So here $I$ is implicitly
defined as the index set for the simple roots.
For any $J \subset I$, let $\Phi_J$ be the set of roots spanned by simple
roots in $J$ and $\Phi_J^+=\Phi_J \cap \Phi^+$,
$\Phi_J^-=\Phi_J \cap \Phi^-$. Let $P_J$ be the standard parabolic subgroup
corresponding to $J$ which is characterized by
$U_\a \subset P_J \Leftrightarrow (\a \in \Phi_J\ \text{or}\ \a \in \Phi^+)$.
Let $X_J=G/P_J$ the corresponding flag variety (there will be no subscripts
to $X=X^*(T)$ and so we will avoid any possible confusion).
Note that $X_{\emptyset} = G/B$ classifies Borel subgroups of $G$ and
is sometimes written as $\cb$.

\medskip\noindent
For each $i \in I$ there is a morphism of algebraic groups
$$
h_i: SL_{2,k} \to G
$$
which maps the diagonal torus into $T$, such that $a \mapsto h_i \bigl(
\begin{smallmatrix} 1 & a \\ 0 & 1 \end{smallmatrix} \bigr)$
is an isomorphism $x_i: {\mathbf G}_a \to U_{\a_i}$, and such that
$y_i: a \mapsto h_i \bigl(
\begin{smallmatrix} 1 & 0 \\ a & 1 \end{smallmatrix} \bigr)$
is an isomorphism $y_i: {\mathbf G}_a \to U_{-\a_i}$. It follows that
$\a_i\bigl(
h_i\bigl(\begin{smallmatrix} t & 0 \\ 0 & t \i \end{smallmatrix} \bigr)\bigr)$
is equal to $t^2$. The cocharacter ${\mathbf G}_m \to T$,
$t \mapsto h_i \bigl( 
\begin{smallmatrix} t & 0 \\ 0 & t \i \end{smallmatrix} \bigr)$
is the coroot $\a_i^\vee$ dual to the simple root $\a_i$.
The datum $(B, T, x_i, y_i; i \in I)$ is called an {\it \'epinglage}
or a {\it pinning} of $G$.
For $i \in I$, let $s_i$ be the corresponding simple reflection in
the Weyl group $W = N_G(T)/T$. The element
$\dot s_i=h_i \bigl(
\begin{smallmatrix} 0 & -1 \\ 1 & 0 \end{smallmatrix} \bigr) \in N_G(T)$
is a representative of $s_i$. Then
\[\tag{a} y_i(1)=x_i(1) \dot s_i x_i(1).\]
For $w \in W$, we denote by $l(w)$ its length and we set
$\dot w=\dot s_{i_1} \dot s_{i_2} \cdots \dot s_{i_n}$,
where $s_{i_1} s_{i_2} \cdots s_{i_n}$ is a reduced expression of $w$.
This expression is independent of the choice of the reduced expression
for $w$.
For $J \subset I$, we denote by $W_J$ the standard parabolic subgroup
of $W$ generated by $s_j$, $j\in J$. We denote by $W^J$ (resp.\ ${}^J W$)
the set of minimal coset representatives in $W/W_J$ (resp.\ 
$W_J \backslash W$. We simply write $W^J \cap {}^K W$ as ${}^K W^J$
for $J, K \subset I$.

\medskip\noindent
For any subgroup $H$ of $G$, we denote by $\Lie(H)$ its Lie algebra.
(We will think of this as the spectrum of the symmetric algebra on
the cotangent space of $G$ at the identity.)
Let us recall the following result (see \cite[page 26, lemma 4]{Sl}).

\begin{lem}
\label{lem-action}
\marpar{lem-action}
Let $H$ be a closed subgroup scheme of $G$. Let $X$ be a scheme
of finite type over $k$ endowed with a left $G$-action. Let
$f: X \rightarrow G/H$ be a $G$-equivariant morphism from $X$ to the
homogeneous space $G/H$. Let $E \subset X$ be the scheme theoretic fibre
$f \i(H/H)$. Then $E$ inherets a left $H$-action and the map
$G \times_H E \rightarrow X$ sending $(g, e)$ to $g \cdot e$
defines a $G$-equivariant isomorphism of schemes.
\end{lem}

\noindent
Let $J \subset I$ and $e=P_J/P_J \in X_J=G/P_J$. We are going to apply
the lemma to the tangent bundle $T_{X_J}$ of $X_J$ and the $G$-equivariant
map $f : T_{X_J} \to X_J$. Note that $f\i(e)=T_{X_J, e}=\Lie(G)/\Lie(P_J)$
because $1\in G$ lifts $e\in X_J$. For an element $p\in P_J$ the left
multiplication $p : X_J \to X_J$ lifts to $\text{inn}_p : G \to G$,
$g \mapsto pgp\i$ because $pgp\i P_J = pgP_J$. Since $\text{inn}_p$
fixes $1\in G$ and acts by the adjoint action on $\Lie(G)$
it follows that the $P_J$-action on $f\i(e)=\Lie(G)/\Lie(P_J)$
is the adjoint action of $P_J$ on $\Lie(G)/\Lie(P_J)$. We define an action
of $P_J$ on $G \times \Lie(G)/\Lie(P_J)$ by
$p \cdot (g, k)=(g p \i, \Ad(p) \cdot k)$. 
Let $G \times_{P_J} \Lie(G)/\Lie(P_J)$ be the quotient space.
Then by the lemma, we have
$$
T_{X_J} \cong G \times_{P_J} \Lie(G)/\Lie(P_J).
$$
In particular, when $J = \emptyset$ we get that the
tangent bundle of $X_{\emptyset} = G/B$ is identified
with $G \times_B \Lie(G)/\Lie(B)$.

\medskip\noindent
What about the relative tangent bundles?
Consider a subset $J \subset I$ as before. Then $B \subset P_J$
and we obtain a canonical morphism $p_{J} : X_\emptyset \to X_J$.
Note that $\Lie(P_J)/\Lie(B)$ is a $B$-stable subspace of
$\Lie(G)/\Lie(B)$. Again by the above lemma, we see that the
vertical tangent bundle
$T_{p_J} \subset T_{X_\emptyset}$ of the projection map
$p_{J}$ is isomorphic to
$$
G \times_{B} \Lie(P_J)/\Lie(B) \subset G \times_{B} \Lie(G)/\Lie(B).
$$

\medskip\noindent
Consider the adjoint group $ad(G)$ of $G$ together with the image
$ad(T)$ of $T$. The character group $X_*(ad(T))$ is the
free ${\mathbf Z}$-module
$\oplus_{i\in I} {\mathbf Z}\a_i \subset X = X_\ast(T)$.
Let $\rho^\vee \in X_\ast(ad(T))$ be the cocharacter such that
$(\a_i, \rho^\vee)=1$ for all $i \in I$. In other words,
$$
\rho^\vee : {\mathbf G}_m \to ad(G)
$$
is a morphism into $ad(T)$ such that $\rho^\vee(t)$ acts via
multiplication by $t$ on $U_{\a_i}$ for all $i\in I$. From this we
conclude that, for any $u\in U$ we have
$$
\lim\nolimits_{t \to 0}\ \text{inn}_{\rho^\vee(t)}(u)  = 1
$$
and, for any $u^- \in U^-$ we have
$$
\lim\nolimits_{t \to \infty}\ \text{inn}_{\rho^\vee(t)}(u^-)  = 1
$$
Here, by abuse of notation, for an element $\bar g \in ad(G)$ 
we denote $\text{inn}_{\bar g}$ the associated inner automorphism
of $G$; in other words, $\text{inn}_{\bar g} = \text{inn}_g$
for any choice of $g\in G$ mapping to $\bar g \in ad(G)$.
Since each $P_J$ contains $T$ and hence the center of $G$ we see
that ${\mathbf G}_m$ acts on $X_J = G/P_J$ via $\rho^\vee$.
This action induces a ${\mathbf G}_m$-action on the bundles
$T_{G/B}$ and $T_{p_J}$. Under the isomorphism
$T_{p_J} = G \times_{B} \Lie(P_J)/\Lie(B)$ this action is given by
$t \cdot (g, k)=(\text{inn}_{\rho^\vee(t)}(g), Ad_{\rho^\vee(t)}(k))$.
We leave the verification to the reader.

\medskip\noindent
Let $w$ be the maximal element in $W$ and let 
$w = s_{i_1} s_{i_2} \cdots s_{i_n}$ be a reduced expression of $w$.
Set 
$u=x_{i_1}(1) \dot s_{i_1} x_{i_2}(1) \dot s_{i_2} \cdots x_{i_n}(1) s_{i_n}$.
By \cite[Proposition 5.2]{MR},
$u \cdot B/B \in \big(B \dot w \cdot B/B\big) \cap
\big(B^- \cdot B/B\big)$.
In particular, $u \cdot B/B = u_1 \dot w \cdot B/B = u_2 \cdot B/B$
for some $u_1 \in U$ and $u_2 \in U^-$. This is the only property of $u$
that we will use. Another way to find an element with this property
is as follows: Since $w$ is the longest element we have
$\dot w B\dot w^{-1} = B^-$. Thus finding a $u$ as above
amounts to showing that $B\dot w B \cap B^-B$ is nonempty, which is clear
since both are open in $G$ (because $BwB=wB^-B$).

\medskip\noindent
Let $C$ be the closure of
$\{\text{inn}_{\rho^\vee(t)}(u) \cdot B/B ; t \in \kk^*\}$
in $X_\emptyset$. In other words, $C$ is the closure of the
${\mathbf G}_m$-orbit of the point $u_\emptyset := u \cdot B/B$
of $X_\emptyset$ with ${\mathbf G}_m$ acting via $\rho^\vee$. Then 
\begin{gather*}
\lim_{t \to \infty} t \cdot u_\emptyset = 
\lim_{t \to \infty} \text{inn}_{\rho^\vee(t)}(u_2) \cdot B/B = B/B, \\
\lim_{t \to 0} t \cdot u_\emptyset =
\lim_{t \to 0} \text{inn}_{\rho^\vee(t)}(u_1) \dot w \cdot B/B =
\dot w \cdot B/B.
\end{gather*} 
Then $C=\{\text{inn}_{\rho^\vee(t)}(u_\emptyset); t \in \kk^*\}
\sqcup [0] \sqcup [\infty]$, where $[0]=\dot w \cdot P_J/P_J$ and
$[\infty]=P_J/P_J$. Let 
$$
h : {\mathbf P}^1 \rightarrow C \subset X_J
$$
be the ${\mathbf G}_m$-equivariant morphism with $h(t) = t \cdot u$,
$h(0) = [0]$, and $h(\infty)=[\infty]$. (Warning: $h$ needs not be
birational.)

\medskip\noindent
With this notation $h^\ast T_{p_J}$ is a ${\mathbf G}_m$-equivariant
vector bundle over $\mathbf P^1$. We have that
$h^\ast T_{p_J} = \oplus L_l$ for some ${\mathbf G}_m$-equivariant
line bundles $L_l$ over $\mathbf P^1$. We would like to compute
the degrees of the line bundles $L_l$.
For any ${\mathbf G}_m$-equivariant line bundle $L$
on ${\mathbf P}^1$ we define two integers $n(L,0)$, and $n(L,\infty)$
as the weight of the ${\mathbf G}_m$-action on the fibre
at $0$, respectively $\infty$. We leave it to the reader to prove
the formula
$$
\deg L = n(L,0) - n(L,\infty),
$$
see also Remark \ref{rmk-doesnotmatter} below.
In order to compute the degrees of the $L_l$ we compute the
integers $n(L_l, 0)$ and $n(L_l,\infty)$. And for this in turn
we compute the fibre of $T_{p_J}$ at the points $[0]$ and $[\infty]$
as ${\mathbf G}_m$-representations. In fact both points are $T$-invariant
and hence we can think of $T_{p_J}|_{[0]}$ and $T_{p_J}|_{[\infty]}$
as representations of $T$.

\medskip\noindent
The fibre of $T_{p_J}$ at the point $[\infty] = B/B$
is just the vector space $\text{Lie}(P_J)/\text{Lie}(B)$
with $T$ acting via the adjoint action. We have
$\Lie(P_J) = \Lie(B) \oplus \oplus_{\a \in \Phi^-_J} \Lie(U_{\a})$.
Thus
$$
\Lie(P_J)/\Lie(B) \cong \oplus_{\a \in \Phi^-_J} \Lie(U_{\a}).
$$
How does ${\mathbf G}_m$ act on this? It acts via the coroot $\rho^\vee$.
For each $\a \in \Phi^-_J$ we have $(\a, \rho^\vee)<0$, hence all
the weights are negative, in other words all the $n(L_l,\infty)$
are negative.

\medskip\noindent
The fibre of $T_{p_J}$ at the point $[0] = \dot w \cdot B/B$
is the space $ \{\dot w \} \times  \text{Lie}(P_J)/\text{Lie}(B)$
inside $G\times_{B} \text{Lie}(P_J)/\text{Lie}(B)$. For an
element $t \in T$ and an element $k \in \text{Lie}(P_J)/\text{Lie}(B)$
we have (see above)
\begin{align*}
t \cdot (\dot w, k) & = (\text{inn}_t(\dot w), Ad_t(k))\cr
& = (t \dot w t^{-1}, Ad_t(k)) \cr
& = (\dot wt^wt^{-1}, Ad_t(k)) \cr
& = (\dot w, Ad_{t^w}(Ad_{t^{-1}}(Ad_t(k)))) \cr
& = (\dot w, Ad_{t^w}(k)).\cr
\end{align*}
Thus the weights for $T$ on $T_{p_J}|_{[0]}$ are the $w(\a)$, where
$\a \in \Phi^-_J$. How does ${\mathbf G}_m$ act on this?
It acts via the coroot $\rho^\vee$. For each $\a \in \Phi^-_J$
we have $(w(\a), \rho^\vee) > 0$, because the longest element $w$
exchanges positive and negative roots. This follows directly from
$\dot wB\dot w^{-1} = B^-$.) Hence all the weights are positive,
in other words all the $n(L_l,0)$ are positive.

\medskip\noindent
From this we conclude that all the line bundles $L_l$ have positive
degree. In other words we have show that $h^\ast T_{p_J}$ is a positive
vector bundle for all $J \subset I$. This proves Theorem \ref{thm-He}.

\begin{rmk}
\label{rmk-doesnotmatter}
\marpar{rmk-doesnotmatter}
The sign of the formula for the degree of
an equivariant line bundle on $\PP^1$ may
appear to be wrong due to the fact that we
are working with line bundles and not
invertible sheaves (which is dual).
The actual sign of the formula for an equivariant line bundle on $\PP^1$
does not matter for the argument however. Namely, in either case
the rest of the arguments show that all the line bundles
$L_l$ have the same sign. And since we know that $G/B$ is Fano
we know that the sum of all the degrees in of the line bundles
in the pull back of $T_{G/B}$ has to be positive.
\end{rmk}

%%%%%%%%%%%%%%%%%%%%%%%%%%%%%%%%%%%%%%%%%%%%%%%%%%%%%%%%%%%%%%%
%%
%% Rational simple connectedness of homogeneous spaces
%% 
%%%%%%%%%%%%%%%%%%%%%%%%%%%%%%%%%%%%%%%%%%%%%%%%%%%%%%%%%%%%%%%

\section{Rational simple connectedness of homogeneous spaces} 
\label{sec-GPR1C}
\marpar{sec-GPR1C}

\noindent
Let $k$ be an algebraically closed field.
Let $X$ and $Y$ be smooth, connected, quasi-projective $k$-schemes.
Let $\mc{L}$ be an ample invertible sheaf on $X$.  
Let $u: Y \rightarrow \Kgnb{0,0}(X,1)$ be a morphism.
This corresponds to the $X, \mathcal{C}, Y, p, q$
in the following diagram of $k$-schemes
$$
\xymatrix{
\PP^1 \ar[dr]_\tau \ar[r]^\sigma & \mc{C} \ar[r]^p \ar[d]^q & X \\
& Y.}
$$
Automatically $q$ is a smooth, projective morphism whose geometric
fibres are isomorphic to $\PP^1$ which map to lines in $X$ via $p$.
There is an associated $1$-morphism
$$
v: \mc{C} \rightarrow \Kgnb{0,1}(X,1).
$$

\begin{prop}
\label{prop-diagram}
\marpar{prop-diagram}
In the situation above with $\sigma : \PP^1 \to \mathcal{C}$ as in
the displayed diagram. Assume that
\begin{enumerate}
\item all lines parametrized by $Y$ are free,
\item $p$ is smooth,
\item $u$ is unramified,
\item either $\tau$ is a free rational curve on $Y$,
or $p \circ \sigma : \PP^1 \to X$ is free, and
\item $\sigma^* T_p$ and $\sigma^* T_q$ are ample.
\end{enumerate}
Then $X$ has a very twisting scroll.
\end{prop}

\begin{proof}
We will show the pair $(R, s)$ where the scroll $R$ is defined as
$R = \PP^1 \times_{q\circ \sigma, Y, q} \mc{C}$ and
with section $s : \PP^1 \rightarrow R$ given by $\sigma$
satisfies all the properties of Definition \ref{defn-twistabs}.
Product will mean product over $\text{Spec}(k)$.

\medskip\noindent
We will use the following remarks frequently below.
Denote the projection to $\PP^1$ by $q_R: R \rightarrow \PP^1$
and the projection to $\mathcal{C}$ by $\sigma_R : R \to \mathcal{C}$.
For every coherent sheaf
$\mc{F}$ on $\mc{C}$, cohomology and base change implies
the natural map
$$
\tau^* R^1 q_* \mc{F} \rightarrow
R^1 q_{R,*}( \sigma_R^*\mc{F})
$$
is an isomorphism. If the sheaves 
$R^1q_{R}(\mc{F})$ are zero (along the image of $\tau$) then
$$
\tau^* q_* \mc{F} \rightarrow q_{R,*} ( \sigma_R^* \mc{F})
$$
is an isomorphism.

\medskip\noindent
Note that $N_{R/\PP^1\times X}$ equals the pullback of
$N_{\mc{C}/Y\times X}$. Because $T_{X\times Y} = \text{pr}_1^* T_X 
\oplus \text{pr}_2^* T_Y$, and because $q$ is smooth
there is a short exact sequence of locally free sheaves
$0 \to T_p \to q^* T_Y \to N_{\mc{C}/Y\times X} \to 0$
on $\mathcal{C}$.
There is an associated long exact sequence
$$
0 \to q_* T_p \to q_* q^* T_Y \xrightarrow{\alpha}
q_* N_{\mc{C}/Y\times X} \to R^1 q_* T_p \to 0
$$
using the fact that $R^1 q_* q^* T_Y$ is zero.
We can say more: (a) $q_*q^* T_Y$ is canonically isomorphic
to $T_Y$, (b) $q_* N_{\mc{C}/Y\times X}$ is canonically
isomorphic to $u^* T_{\Kgnb{0,0}(X,1)}$ (see Remark
\ref{rmk-technical}), and (c) the sheaf
homomorphism $\alpha$ above is canonically isomorphic
to $\text{d}u$. By hypothesis $u$ is unramified, in other
words $\alpha$ is injective. Thus $q_*T_p = 0$,
and we get an isomorphism $R^1q_* T_p = N_{Y/\Kgnb{0,0}(X,1)}$.
We conclude that $R^1q_* T_p$ is locally free, and formation
of $R^1q_* T_p$ and $q_* T_p$ commutes with arbitrary change
of base. By relative duality, we get that $R^1 q_* T_p$
is dual to $q_*(\Omega_p \otimes_{\OO_\mc{C}} \Omega_q)$,
and that $R^1q_*(\Omega_p \otimes_{\OO_\mc{C}}\Omega_q) = 0$.

\begin{claim}
\label{claim-vanishing}
\marpar{claim-vanishing}
The sheaf $\tau^* q_* (\Omega_p \otimes_{\OO_\mc{C}}\Omega_q)$
is anti-ample, and thus $\tau^*N_{Y/\Kgnb{0,0}(X,1)}$ is ample.
\end{claim}

\noindent
The surface $R$ is a Hirzebruch surface.
Denote by $\mc{E}$ the pullback of $\Omega_p\otimes_{\OO_\mc{C}}
\Omega_q$ to $R$. By our remarks above the claim is equivalent
to the assertion that $(q_R)_*\mc{E}$ is anti-ample.
The normal bundle $N_{s(\PP^1)/R}$ is canonically isomorphic
to $\sigma^* T_q$, which is ample by hypothesis.
Therefore the divisor $s(\PP^1)$ on $R$ moves
in a basepoint free linear system. This implies the section $s$
deforms to a family $\{s_t\}_{t\in \Pi}$ of sections whose images
cover a dense open subset of $R$.
Of course the pullback of $\Omega_p \otimes_{\OO_\mc{C}}
\Omega_q$ to $R$ is a locally free sheaf $\mc{E}$ whose dual
$\mc{E}^\vee$ is canonically isomorphic to the pullback of
$T_p\otimes_{\OO_\mc{C}} T_q$.  And $s^* \mc{E}^\vee$ equals $\sigma^*
T_p\otimes_{\OO_{\PP^1}} \sigma^* T_q$.  By hypothesis, each of
$\sigma^* T_p$ and $\sigma^* T_q$ is ample on $\PP^1$.  Thus $s^*
\mc{E}^\vee$ is ample on $\PP^1$.  Since ampleness is an open
condition, after shrinking $\Pi$ we may assume $s_t^* \mc{E}^\vee$ is
ample for $t \in \Pi$.  Therefore, $s_t^* \mc{E}$ is
anti-ample.

\mni
For every $t$ there is an evaluation morphism
$$
e_t: (q_R)_* \mc{E} \rightarrow s_t^* \mc{E}.
$$
Of course, since the curves $s_t(\PP^1)$ cover a dense open subset of
$R$, the only local section of $(q_R)_*\mc{E}$ in the kernel of
\emph{every} evaluation morphism $e_t$ is the zero section.  Since
$(q_R)_* \mc{E}$ is a coherent sheaf, in fact for $N \gg 0$ and for
$t_1,\dots, t_N$ a general collection of closed points of $\Pi$, the
morphism
$$
(e_{t_1},\dots, e_{t_n}):
(q_R)_* \mc{E}
\longrightarrow
\bigoplus\nolimits_{i=1}^N s_{t_i}^* \mc{E}
$$
is injective.  Since every summand $s_{t_i}^* \mc{E}$ is anti-ample,
the direct sum is anti-ample.  And a locally free sheaf admitting
an injective sheaf homomorphism to an anti-ample sheaf is itself
anti-ample.  Therefore $(q_R)_* \mc{E}$ is anti-ample, proving Claim
~\ref{claim-vanishing}.

\medskip\noindent
As usual, we denote by $\text{ev}: \Kgnb{0,1}(X,1) \rightarrow X$
the evaluation morphism. Thus $\text{ev} \circ v = p$ as morphisms
$\mc{C} \to X$. Consider the diagram
$$
\xymatrix{
\mc{C} \ar[d]^q \ar[r]^-v &
\Kgnb{0,1}(X,1) \ar[d]^{\text{forget}} \ar[r]^-{\text{ev}} &
X \\
Y \ar[r]^-u & \Kgnb{0,0}(X,1)
}
$$
whose square is cartesian. Since we assumed that all lines
parametrized by $Y$ are free it follows that $\Kgnb{0,1}(X,1)$
is nonsingular at every point of $v(\mc{C})$. As $p$ is assumed
smooth, it follows that $\text{ev} : \Kgnb{0,1}(X,1) \to X$
is smooth at all points of $v(\mc{C})$ (by Jacobian criterion
for example). Also, the diagram shows $v$ is unramified and
$N_{\mc{C}/\Kgnb{0,1}(X,1)}$ is canonically isomorphic to
$q^* N_{Y/\Kgnb{0,0}(X,1)}$. Smoothness of $\text{ev}$ and $p$
implies we have a short exact sequence
$$
0 \to T_p \to v^* T_\text{ev} \to N_{\mc{C}/\Kgnb{0,1}(X,1)} \to 0
$$
By hypothesis, $\sigma^* T_p$ is ample. Claim \ref{claim-vanishing}
implies $\tau^* N_{Y/\Kgnb{0,0}(X,1)} = 
\sigma^* N_{\mc{C}/\Kgnb{0,1}(X,1)}$ is ample.  Therefore
$(v\circ \sigma)^* T_\text{ev}$ is ample.

\medskip\noindent
At this point we can apply Lemma \ref{lem-twistexist} to the surface
$R \to \PP^1 \times X$ and the section $s$, which combine
to give the morphism $g = v \circ \sigma :
\PP^1 \to \Kgnb{0,1}(\PP^1 \times X/\PP^1, 1)$.
Condition (1) corresponds to the fact that $s$ moves on $R$. Condition
(2) holds because the lines parametrized by $Y$ are free.
Condition (3) with $m = 2$ holds because, as we just saw, the sheaf
$(v\circ \sigma)^* T_\text{ev}$ is ample. It remains to see that
the morphism $p \circ \sigma : \PP^1 \to X$ is free. This follows
either by assumption, or if $\tau$ is free we argue as follows:
The pull back
$(p \circ \sigma)^* T_X$ is a quotient of $\sigma^* T_{\mc{C}}$.
The sheaf $\sigma^* T_{\mc{C}}$ is sandwiched between
$\sigma^*T_q$ and $\sigma^* q^* T_Y$ and by assumption both
are globally generated. Hence $\sigma^* T_{\mc{C}}$ and also
$(p \circ \sigma)^* T_X$ is globally generated.
\end{proof}

\begin{rmk}
\label{rmk-jason}
\marpar{rmk-jason}
In the ``abstract'' situation of Proposition \ref{prop-diagram} we
can often deduce that the fibres of $\text{ev} : 
\Kgnb{0,1}(X, 1) \to X$ are rationally connected using
\cite[Proposition 3.6]{Sr1c} (the proof of this
uses a characteristic $0$ hypothesis). We will avoid
this below to keep the paper more self contained; instead
we will use a little more group theory.
\end{rmk}

\noindent
The new result in the following corollary is the statement
about very twisting scrolls.

\begin{cor}
\label{cor-GPverytwisting}
\marpar{cor-GPverytwisting}
Let $k$ be an algebraically closed field.
Let $G$ be a connected semi-simple group over $k$ and
let $P \subset G$ be a standard maximal parabolic
subgroup. Let $Z = G/P$. Let $\mathcal{L}$ be
an ample generator for $\text{Pic}(Z) \cong \ZZ$.
Then the evaluation morphism
$$
\text{ev}:\Kgnb{0,1}(Z,1) \rightarrow Z
$$
is surjective, smooth and projective. In addition, there exists a
very twisting scroll in $Z$.
\end{cor}

\begin{proof}
The action of $G$ on $Z$ determines a sheaf homomorphism
$$
T_e G\otimes_k \OO_Z \rightarrow T_Z
$$
which is surjective because the action is separable and $X$ is
homogeneous.  Thus $T_Z$ is globally generated.  Therefore every
smooth, rational curve in $Z$ is free.  This implies
$$
\text{ev}:\Kgnb{0,1}(Z,1) \rightarrow Z
$$
is smooth. It is always projective, allthough at this point
the space of lines on $Z$ could be empty.

\medskip\noindent
Let $B$ be a Borel subgroup of $G$ contained in $Z$, and let $T$ be
a maximal torus in $B$. The data $(G, B, T)$ determines a root system
$(X, \Phi, Y , \Phi^\wedge)$ as in Section \ref{sec-He}; we will use
the notation introduced there.
As $P$ is a standard maximal parabolic subgroup, $P$ equals the
parabolic
subgroup $P_{J}$ where $J = I - \{j\}$ for an element $j$ of $I$.
Thus $Z = G/P_{J} = X_J$ in the notation of Section \ref{sec-He}.
It is proved in \cite[\S 4.20]{Cohen} and \cite[Lemma 3.1]{CohenCooperstein}
that the subvariety
$$
L := P_{\{j\}} \cdot P_{J}/P_{J}
$$
is a line in $G/P_{J}$ with respect to the ample sheaf $\mc{L}$.

\medskip\noindent
We quickly explain how to see this in terms of the \emph{pinning}
we chose in Section \ref{sec-He}. Namely, after replacing $G$ by
its simply connected covering we may assume $\mc{L}$ corresponds
to a $G$-equivariant invertible sheaf. This in turn corresponds
to a character $\chi : P_J \to \mathbf{G}_m$ which generates the
character group $X^*(P_J)$. Since $G$ is simply connected the
lattice $Y$ is freely generated by the coroots $\alpha_i^\vee$
and hence there is an element $\beta \in X = X^*(T)$ such that
$(\beta, \alpha_i^\vee) = 1$ and $(\beta, \alpha_j^\wedge) = 0$
for $i \in I, i\not=j$. This element clearly generates the character
group of $P_J$ and hence the restriction of $\chi$ to the maximal
torus $T$ corresponds to the element $\beta$ (up to sign). Next,
observe that $h_j : SL_2 \to G$ gives a closed immersion
$\overline{h}_j : SL_2/B_2 \to G/P_J$, where $B_2 \subset SL_2$
is the group of upper triangular matrices. We claim that
$\text{Im}(\overline{h}_j) = L = \PP^1$ is the desired line.
(Note that $h_j$ maps into $P_{\{j\}}$ and hence we actually
get the same line as above.)
This is true because the pull back of $\mc{L}$ to the image
corresponds to the standard character $B_2 \to \mathbf{G}_m$
by what we said about $\chi, \beta$ above (as well as the properties
of $h_j$).

\medskip\noindent
The action of $G$ on $Z$ induces an action of $G$ on
$\Kgnb{0,0}(Z,1)$ and $\Kgnb{0,1}(Z,1)$. Let $p = P_J/P_J \in G/P_J$.
This is a point on our line $L$. The stabilizer of $(L,p)$ in
$\Kgnb{0,1}(Z,1)$ contains the Borel subgroup $B$, and thus is of the
form $P_K$ for a subset $K \subset J = I - \{j\}$. Since $P_K$ is
parabolic, the orbit $\mc{C}$ of $(L,p)$ is a projective (hence
closed) $G$-orbit. The image $Y$ of $\mc{C}$ in $\Kgnb{0,0}(Z,1)$ is
also a projective $G$-orbit. Observe that $P_j$ acts transitively on
the subset $\{(L,q)|q\in L\}$ of $\Kgnb{0,1}(Z,1)$ by construction
of $L$. Thus the fibre of the forgetful morphism
$\Kgnb{0,1}(Z,1) \rightarrow \Kgnb{0,0}(Z,1)$ over $[L]$ equals the
fibre of $\mc{C} \rightarrow Y$ over $[L]$.  By
homogeneity it follows that
$$
\mc{C} = \Kgnb{0,1}(Z,1) \times_{\Kgnb{0,0}(Z,1)} Y.
$$
So the diagram
$$
\xymatrix{
\mc{C} \ar[r]^p \ar[d]_q& Z \\
Y}
$$
is a diagram of smooth, projective morphisms where every geometric
fibre of $q$ is a smooth, rational curve.  

\medskip\noindent
By Theorem ~\ref{thm-He}, there exists a morphism,
$$
\sigma : \PP^1 \rightarrow \mc{C}
$$
such that $\sigma^* T_p$ and $\sigma^* T_q$ are both ample sheaves.
The morphism $p$ is smooth by homogeneity.  And the morphism
$u : Y \to \Kgnb{0,0}(Z,1)$ is a closed immersion hence unramified.
Thus, by Proposition \ref{prop-diagram}, there is a
very twisting scroll in $Z$.  
\end{proof}

\begin{rmk}
\label{rmk-veryample}
\marpar{rmk-veryample}
The variety $G/P$ can be constructed as the orbit of
the highest weight vector in $\PP(V_{\omega_j})$ where
$\omega_j$ is the $j$th fundamental weight (this is what
we called $\beta$ in the proof above). Under this
embedding $\mc{L}$ is the pull back of $\OO(1)$ and
lines are lines in $\PP(V_{\omega_j})$. In particular
$\mc{L}$ is very ample.
\end{rmk}

\begin{lem}
\label{lem-fibreRC}
\marpar{lem-fibreRC}
Let $k$ be an algebraically closed field of characteristic $0$.
Let $G$ be a linear algebraic group over $k$.
Let $X$, $Y$ be proper $G$-varieties over $k$.
Let $Y \to X$ be a $G$-equivariant morphism.
Suppose that $Y$ is rationally connected and
that $X$ has an open orbit whose stabilizer is
a connected linear algebraic group.
Then the geometric generic fibre of $Y \to X$ is
rationally connected.
\end{lem}

\begin{proof}
Let $x \in X(k)$ be a point in the open orbit.
Let $H$ be the stabilizer of $x$ in $G$, which is
a connected linear algebraic group by assumption.
Let $F = Y_x$ be the fibre of $Y\to X$ over $x$.
Consider the morphism $\Psi : G \times F \to Y$,
$(g, f) \mapsto g \cdot f$. This is a
dominant morphism of varieties over $k$. By construction
the fibres of $\Psi$ are isomorphic to the 
birationally rationally connected variety $H$.
It follows from the main result of \cite{GHS}
that $G \times F$ is birationally rationally connected.
Hence $F$ is rationally connected.
\end{proof}

\noindent
Note that the Bruhat decomposition in particular implies
that the variety $G/P \times G/P$ has finitely many $G$-orbits.
In addition all stabilizers of points of $G/P \times G/P$
are connected, for example by \cite[Proposition 14.22]{Borel}.

\begin{lem}
\label{lem-GPevRC}
\marpar{lem-GPevRC}
Let $k$ be an algebraically closed field of characteristic $0$.
Let $Z = G/P$, $\mc{L}$ be as in Corollary \ref{cor-GPverytwisting} above.
Then every geometric fibre of
$$
\text{ev}: \Kgnb{0,1}(Z,1) \rightarrow Z
$$
is nonempty and rationally connected.
\end{lem}

\begin{proof}
Because $Z$ is homogeneous, all fibres are isomorphic and
it suffices to prove one fibre is rationally connected.
By Lemma \ref{lem-fibreRC}, it suffices to show that $\Kgnb{0,1}(Z, 1)$
is rationally connected. Clearly, it suffices to show that
$\Kgnb{0,0}(Z, 1)$ is rationally connected. Since a line is
determined by any pair of points it passes through, this follows
from the remarks about the Bruhat decomposition immediately
preceding the lemma.
\end{proof}

\begin{lem}
\label{lem-GPchainR1C}
\marpar{lem-GPchainR1C}
Let $k$ be an algebraically closed field of characteristic $0$.
Let $Z = G/P$, $\mc{L}$ be as in Corollary \ref{cor-GPverytwisting}
above. There exists a positive integer $n$ and a nonempty open
subset $V$ of $Z\times Z$ such that the geometric fibres of
the evaluation morphism
$$
\text{ev}_{1, n+1} :
\FCh_2(X/k, n)
\longrightarrow
X\times X
$$
over $V$ are nonempty, irreducible and birationally rationally connected.
(Notation as in Section \ref{sec-peace}.)
\end{lem}

\begin{proof}
We will use without mention that every line on $Z$ is free.

\medskip\noindent
We rephrase a few arguments of Campana and Koll\'ar-Miyaoka-Mori
in this particularly nice setting. Consider the diagram
$$
\Kgnb{0,0}(Z, 1) \xleftarrow{\Phi} \Kgnb{0,1}(Z, 1) \xrightarrow{\text{ev}} Z
$$
For any closed subset $T \subset Z$, the closed
set $T' = \text{ev}(\Phi^{-1}(\Phi(\text{ev}^{-1}(T))))$
is the set of points which are connected by a line passing through
a point in $T$. Note that if $T$ is irreducible, then so is $T'$ since
by Corollary \ref{cor-GPverytwisting} the morphism $\text{ev}$
is smooth with irreducible fibres according to Lemma \ref{lem-GPevRC}.
For any point $z \in Z(k)$ consider the increasing sequence of
closed subvarieties
$$
T_0(z) = \{z\} \subset T_1(z) = T_0(z)' \subset T_2(z) = T_1(z)'
\subset \ldots
$$
The first is the variety of points that lie on a line passing through
$z$. The second is the set of points which lie on a chain of lines
of length $2$ passing through $z$, etc. Let $n \geq 0$ be the first
integer such that the dimension of $T_n(z)$ is maximal. Then obviously
$T_{n+1}(z) = T_n(z)$. Hence $T_{2n} = T_n(z)$. We conclude that
for every point $z' \in T_n(z)$ we have $T_n(z') = T_n(z)$.
Consider the map $Z \to \text{Hilb}_{Z/k}$, $z \mapsto [T_n(z)]$.
By what we just saw the fibres of this map are exactly
the varieties $T_n(z)$. Since $\text{Pic}(Z) = \ZZ$ we conclude
that either $T_n(z) = \{z\}$, or $T_n(z) = Z$. The first possibility
is clearly excluded and hence $T_n(z) = Z$, in other words every
pair of points on $Z$ can be connected by a chain of lines
of length $n$.

\medskip\noindent
Thus the morphism
$$
\text{ev}: \FCh_2(X/k,n) \rightarrow X\times_k X
$$
is surjective.  This is a $G$-equivariant morphism for the evident
actions of $G$ on the domain and target.  
There exists an open orbit in $Z\times Z$, and all stabilizers are
connected, see the remarks preceding Lemma \ref{lem-GPevRC}.
By Lemma \ref{lem-fibreRC} it suffices to prove that $\FCh_2(X/k,n)$
is rationally connected. For $n = 1$ this was shown in
the proof of Lemma \ref{lem-GPevRC}. For larger $n$ it follows by
induction on $n$ from Lemma \ref{lem-GPevRC} in exactly the same way
as the proof of the corresponding statement of Lemma \ref{lem-peacechain}.
\end{proof}

%%%%%%%%%%%%%%%%%%%%%%%%%%%%%%%%%%%%%%%%%%%%%%%%%%%%%%%%%%%%%%%
%%
%% Serre's conjecture II
%% 
%%%%%%%%%%%%%%%%%%%%%%%%%%%%%%%%%%%%%%%%%%%%%%%%%%%%%%%%%%%%%%%

\section{Families of projective homogeneous varieties}
\label{sec-famGP}
\marpar{sec-famGP}

\noindent
In this section $k$ is an algebraically closed field of any characteristic.
A \emph{projective homogeneous space} over $k$ or an algebraically
closed extension of $k$ is a variety isomorphic to $G/P$ where $G$ is
a connected, semi-simple, simply connected linear algebraic group
and $P \subset G$ is a standard parabolic subgroup. (In
other words we do not allow exotic, nonreduced, parabolic subgroup
schemes.) Let $S$ be a variety over $k$, and let $\ol{\eta} \to S$ be a
geometric generic point. A \emph{split family of homogeneous spaces}
over $S$ will be a proper smooth morphism
$X \to S$ such that $X_{\ol{\eta}}$ is a projective homogeneous space and
that $\text{Pic}(X) \to \text{Pic}(X_{\ol{\eta}})$ is surjective.

\begin{lem}
\label{lem-inductive}
\marpar{lem-inductive}
Let $k$ be an algebraically closed field of any characteristic.
Let $X \to S$ be a split family of homogeneous spaces over $S$.
Then, after possibly shrinking $S$, there exists a factorization
$$
X \to Y \to S
$$
such that $X \to Y$ is a split family of homogeneous spaces
over $Y$, such that $Y \to S$ is a split family of homogeneous spaces
over $S$ and such that $\text{Pic}(Y_{\ol{\eta}}) = \ZZ$.
\end{lem}

\begin{proof}
Choose an isomorphism $X_{\ol{\eta}} = G/P$ for some
$G$, $P$ over $\ol{\eta}$. Choose a maximal parabolic
$P \subset P'$ containing $P$. Let $\mc{L}_0$ be an
ample generator of $\text{Pic}(G/P')$, see Corollary
\ref{cor-GPverytwisting}, its proof and Remark \ref{rmk-veryample}.
Let $\mc{L}$ be an invertible sheaf on $X$ whose restriction
to $X_{\ol{\eta}}$ is isomorphic to the pull back
of $\mc{L}_0$ to $X_{\ol{\eta}} = G/P$ via the morphism
$\pi : G/P \to G/P'$. Since $\mc{L}_0$ is ample, some power
$\mc{L}_0^N$ is very ample and globally generated.
Also, $H^0(G/P, \pi^*\mc{L}_0^N) = H^0(G/P', \mc{L}_0^N)$
as the fibres of $\pi$ are connected projective varieties.
We conclude, by cohomology and base change, that after
shrinking $S$ we may assume that $\mc{L}^N$ is globally
generated on $X$. 

\medskip\noindent
Choose a finite collection of sections $s_0, \ldots, s_M$
which generate the $\kappa(\ol{\eta})$-vector space
$H^0(G/P, \pi^*\mc{L}_0^N) = H^0(G/P', \mc{L}_0^N)$.
Consider the morphism $X \to \PP^M_S$ over $S$.
Note that on the geometric generic fibre the image
is the projective embedding of $G/P'$ via the
sections $s_0,\ldots,s_M$. Hence, after possibly shrinking
$S$ and using Stein factorization, we get a factorization
$X  \to Y \to S$ such that $Y_{\ol{\eta}} \cong G/P'$
with $X_{\ol{\eta}} \to Y_{\ol{\eta}}$ equal to $\pi$.
The family $Y \to S$ is split because the invertible
sheaf $\mathcal{L}$ descends to $Y$ due to the fact that
it does so on the geometric generic fibre. (Hint: just
take the pushforward of $\mc{L}$ under $X \to Y$.)

\medskip\noindent
The geometric generic fibre of $X \to Y$ is isomorphic to $P/P'$
which is a projective homogeneous variety as well (for a possibly
different group). To show that $X \to Y$ is split it suffices to
show that $\text{Pic}(G/P) \to \text{Pic}(P'/P)$ is surjective,
as $X \to S$ is split. This surjectivity can be proved using
the fact that the derived group of the Levi group of $P'$ is
simply connected (since $G$ is simply connected), combined
with for example the results of \cite{popov}. We prefer to
prove it by considering the Leray spectral sequence
for $f : G/P \to G/P'$ and the sheaf $\mathbf{G}_m$. Namely,
$R^1f_*\mathbf{G}_m$ is
a constant sheaf with value $\text{Pic}(P/P')$
(due to the fact that $G/P'$ has trivial fundamental group).
Hence the obstruction to surjectivity is an element
in the Brauer group $Br(G/P')$. Since $k$ is algebraically
closed and $G/P'$ is rational (by Bruhat decomposition) we
see $Br(G/P') = 0$ and we win.
\end{proof}

\noindent
Let $k$ be an algebraically closed field of positive characteristic.
Let $R$ be a Cohen ring for $k$ (for example the Witt ring of $k$).
Let $G$ be a connected, semi-simple, simply connected linear
algebraic group over $k$, and let $P \subset G$ be a standard parabolic
subgroup. According to Chevalley (\cite{Chevalley}, see
also \cite{Demazure} and \cite{Springer}) there exists a
reductive group scheme $G_R$ over $R$ and a parabolic subgroup scheme
$P_R \subset G_R$ whose fibre over $k$ recovers the pair
$(G, P)$. Denote $Z_R = G_R/P_R$ the smooth projective $G_R$-scheme
over $R$ whose special fibre is $Z = G/P$. Since $\text{Pic}(Z_R)$
is isomorphic to the character group of $P_R$, and since by
construction the character group of $P_R$ equals that of
$P$, we see that $\text{Pic}(Z_R) = \text{Pic}(Z)$. Choose
an ample invertible sheaf $\OO(1)$ on $Z_R$.

\begin{lem}
\label{lem-flat}
\marpar{lem-flat}
The automorphism scheme $\text{Aut}_R((Z_R, \OO(1)))$ is smooth over $R$.
\end{lem}

\begin{proof}
The scheme of automorphisms $\text{Aut}_R(Z_R)$ is
smooth over $R$ by \cite[Proposition 4]{Demazure-Aut}.
Thus it suffices to show that the morphism
$\text{Aut}_R((Z_R, \OO(1))) \to \text{Aut}_R(Z_R)$
is smooth. By deformation theory it suffices to show that
$H^1(Z, \OO_Z) = 0$. This follows from \cite[Section 6 Theorem 1]{Kempf-GP}.
\end{proof}

\noindent
The following lemma follows from a general ``lifting''
argument that was explained to us by Ofer Gabber,
Jean-Louis Colliot-Th\'el\`ene and Max Lieblich. For a brief
explanation, see Remark \ref{rmk-lift} below.

\begin{lem}
\label{lem-lift}
\marpar{lem-lift}
Let $k, R, G, P, G_R, P_R$ as above. Assume that $P$ is maximal
parabolic in $G$. Let $\Omega = \overline{K}$ be an algebraic closure of
the field of fractions of $R$. Suppose that every split family
of homogeneous spaces over a surface over $\Omega$
whose general fibre is isomorphic to $G_{\Omega}/P_{\Omega}$
has a section. Then the same is true over $k$.
\end{lem}

\begin{proof}
Pick $\OO(1)$ on $Z_R$ an ample generator of $\text{Pic}(Z_R)
= \text{Pic}(Z) = \ZZ$. Denote $H_R = \text{Aut}_R((Z_R, \OO(1)))$.
By Lemma \ref{lem-flat} we see that $H_R$ is flat over $R$.
At this point \cite[Corollary 2.3.3]{dJS8} applies with our
$H_R$ playing the role of the group named $G_R$ in 
ibid., and with our $Z_R$ playing the
role of the scheme called $V_R$ in ibid.
\end{proof}

\begin{rmk}
\label{rmk-lift}
\marpar{rmk-lift}
What is involved in the proof of \cite[Corollary 2.3.3]{dJS8}?
Consider a split family $X \to S / k$ of homogeneous spaces whose
general fibre is isomorphic to $Z = G/P$. Pick $\mc{L}$ on $X$
restricting to a generator of the picard group of a general fibre.
Then $\text{Isom}((X,\mc{L}), (Z, \OO_Z(1)))$ is a $H_k$-torsor over
a nonempty affine open $U \subset S$ of the surface $S$. After
shrinking $U$ we may assume $U$ is smooth over $k$.
After this one lifts $U$ to $W \to \text{Spec}(R)$ smooth
using \cite{Elkik}. The ``lifting'' result metioned above
is the statement that a $H_k$-torsor over $U$ is the restriction
of a $H_R$-torsor over $W$ after possibly replacing
$W$ by $W' \to W$ etale and trivial over a nonempty open $U$.
This torsor in turn defines a split family of homogeneous
space $\mathcal{X} \to W$ restricting to $X_U \to U$ over $U$.
It is split because the geometric generic fibre has Picard
group $\ZZ$ and the ample generator exists by virtue of
the fact that $H_R$ acts on $\OO(1)$ on $Z_R$.
At this point, a standard specialization argument 
implies the result. For details see \cite{dJS8}.
\end{rmk}

\begin{thm}
\label{thm-main2}
\marpar{thm-main2}
Let $k$ be an algebraically closed field of any characteristic.
Let $S$ be a quasi-projective surface over $k$. Let
$X \to S$ be a split family of projective homogeneous spaces.
Then $X \to S$ has a rational section.
\end{thm}

\begin{proof}
By Lemma \ref{lem-inductive} it suffices to prove the 
theorem in case the general fibre has Picard number $\rho = 1$,
i.e., the parabolic group is maximal. By Lemma \ref{lem-lift}
it suffices to prove the result in the $\rho = 1$
case in characteristic zero. The case $\rho = 1$, $S$ is a
\emph{projective} smooth surface and $X \to S$ is smooth and
projective is Corollary \ref{cor-B} which was proved in the
introduction (as a consequence of the more general
Corollary \ref{cor-main} using the results of Section \ref{sec-GPR1C}).

\medskip\noindent
Thus it remains to deduce the general $\rho = 1$ case in characteristic
$0$ from the case where the base is projective. This is done
by the method of \emph{discriminant avoidance}. Precise references
are \cite[Theorem 1.3]{dJS7}, or \cite[Theorem 1.0.1]{dJS8} 
(either reference contains a complete argument; the first giving
a more down to earth approach than the second).
Please see Remark \ref{rmk-discavoid} below for a brief
synopsis of the argument.
\end{proof}

\begin{rmk}
\label{rmk-discavoid}
\marpar{rmk-discavoid}
Let $k$ be an algebraically closed field (of any characteristic).
As before let $Z = G/P$ with $P$ maximal parabolic and let
$\OO(1)$ be an ample generator of $\text{Pic}(Z)$. Also, as
before let $H = \text{Aut}((Z, \OO(1)))$ be the automorphism
scheme of the pair. According to proof of Lemma \ref{lem-flat}
and the references therein, the group scheme $H$ is reductive.
Let $S$ be a quasi-projective surface over $k$.
Suppose that $X \to S$ is a split family of projective
homogeneous spaces whose general fibre is isomorphic to
$Z$. We will explain the method of \cite{dJS8} reducing
the problem of finding a rational section of $X\to S$
to the problem in a case where the base surface is 
projective.

\medskip\noindent
Choose an invertible sheaf $\mc{L}$ on $X$ whose
restriction to a general fibre is an ample generator.
As in Remark \ref{rmk-lift} there exists an open
$U$ and an $H$-torsor $\mathcal{T}$ over $U$ such
that $(X_U, \mc{L}_U)$ is isomorphic over $U$ to
$(\mc{T}\times_H Z, \OO \boxtensor_H \OO(1))$.
The main result of \cite{dJS8} implies that there
exists a flat morphism $W \to \text{Spec}(k[[t]])$,
an open immersion $j : W_k \to U$, a $H$-torsor $\mc{T}'$ over $W$ 
such that (1) $\mc{T}'_k \cong j^* \mc{T}$ and
(2) the generic fibre $W_{k((t))}$ is projective.
We may use the torsor $\mc{T}'$ to obtain an extension
of the pull back family $j^*X$ to a split family
$\mathcal{X} \to W$
of projective homogeneous spaces over all of $W$.
At this point a standard specialization argument
implies that the existence of a rational section of
$\mathcal{X}_{\overline{k((t))}} \to W_{\overline{k((t))}}$
implies the existence of a rational section of the original
$X \to S$.

\medskip\noindent
The main result of \cite{dJS8} mentioned above is the following:
If $H$ is a reductive algebraic group, and $c \geq 0$ an integer
then the algebraic stack $B(H) = [\text{Spec}(k)/H]$ allows a
morphism $X \to B(H)$ surjective on field points such that
$X$ has a projective compactification $X \subset \overline{X}$
whose boundary has codimension $\geq c$ in $\overline{X}$. This
result easily implies the existence of the family of torsors
as above. For details see the references above.
\end{rmk}

\section{Application to families of complete intersections}
\label{sec-hypersurface}
\marpar{sec-hypersurface}

\noindent
Let $K \supset k$ be the function field of a surface over
an algebraically closed field $k$. Tsen's theorem states
that any complete intersection of type $(d_1,\ldots,d_c)$
in $\PP^n_K$ has a $K$-rational point if $\sum d_i^2 \leq n$.
See \cite{Lang52}. An addendum is that if the inequality is
violated then there are function fields $K$ and complete
intersections without rational points.

\medskip\noindent
In this section we sketch a proof of Tsen's theorem
in characteristic $0$ for families of \emph{hypersurfaces}
using the main results of this paper. As explained in the
introduction this is silly, but hopefully provides
a useful illustration of our methods.

\medskip\noindent
First of all we may replace $k$ by an uncountable algebraically
closed overfield of characteristic zero (by standard limit techniques).
Second we may assume $d_i \geq 2$ for all $i$.
The moduli space $M$ of all complete intersections $X$ of type
$(d_1,\ldots,d_c)$ in $\PP^n$ is an open subset of a product of
projective spaces $M \subset \PP^{n_1}\times \ldots \times \PP^{n_c}$.
Note that we do not require complete intersections to be nonsingular;
it is then easy to see that the complement of $M$ in
$\PP^{n_1}\times \ldots \times \PP^{n_c}$ has codimension $\geq 2$.
It is also straightforward to see that the total space of
the universal family $\mathcal{X} \to M$ is nonsingular.
It follows by standard specialization
techniques that to prove Tsen's theorem for all $K$ and
all complete intersections, it suffices to prove it for
those $X/K$ such that $K$ corresponds to the function field
of a surface $S \subset M$ with the following properties:
\begin{enumerate}
\item $S$ passes through a very general point of $M$,
\item the closure $\overline{S} \subset 
\PP^{n_1}\times \ldots \times \PP^{n_c}$ differs from
$S$ by finitely many points, and
\item the restriction of $\mathcal{X}$ to $S$ has smooth
total space.
\end{enumerate}
At this point Corollary \ref{cor-A} says that it is
enough to show that a very general complete intersection $Y$
of type $(d_1,\ldots,d_c)$ in $\PP^n$ is
rationally simply connected by chains of free
lines, and contains a very twisting surface.

\medskip\noindent
To check that $Y$ is rationally simply connected by chains of free
lines, i.e., that $Y \to \text{Spec}(k)$ satisfies Hypothesis
\ref{hyp-peace} we point out the following:
\begin{enumerate}
\item The space of lines through a general point $y \in Y(k)$
is a nonsingular complete intersection of type
$(1,2,\ldots,d_1,1,2,\ldots,d_2,\ldots,1,2,\ldots,d_c)$
in a $\PP^{n-1}$. Because
$1 + 2 + \ldots + d_1 + 1 + 2 + \ldots
+ d_2 + \ldots + 1 + 2 + \ldots + d_c =
\sum d_i(d_i + 1)/2 \leq \sum (d_i^2 - 1) \leq n - 1$
it is a smooth projective Fano
and hence rationally connected.
\item For a general pair of points $x,y \in Y(k)$
the space of pairs of lines connecting $x$ and $y$
is parametrized by a nonsingular complete intersection
of type
$(1,2,\ldots,d_1-1, 1,2,\ldots,d_1-1, d_1,
\ldots,1,2,\ldots,d_c-1,1,2,\ldots,d_c-1,d_c)$
in $\PP^n$. (To see this consider the equations 
for the locus of the intersection points of the two lines.)
The sum of the degrees is equal to
$\sum d_i^2 \leq n$ and again we see that this is
a smooth projective Fano variety and hence rationally connected.
\end{enumerate}
To prove this argue as in \cite[Chapter V Section 4]{K},
especially the exercises. See also \cite{dJS10}.
This proves Hypothesis \ref{hyp-peace} for $Y/k$.

\medskip\noindent
As usual the hardest thing to check is the existence
of a very twisting surface. A simple parameter count shows
that a very twisting surface should exist on a general
complete intersection with $\sum d_i^2 \leq n$. We
do not know how to prove the existence unless $c = 1$. Namely,
in \cite{HS2} it was show very twisting scrolls exist
if $d_1^2 + d_1 + 1 \leq n$. In \cite[Propositions 5.7 and 10.1]{Sr1c}
it is shown that very twisting scrolls exist in a general
hypersurface if $d_1^2 \leq n$, thereby concluding our sketch
of the proof of Tsen's theorem for familie of hypersurfaces.

\medskip\noindent
Final remark: In the preprint \cite{dJS10}
very twisting surfaces are constructed in any smooth complete
intersection in $\PP^n$ under the assumption that
$n + 1 \geq \sum (2d_i^2 - d_i)$. But note that these surfaces
are not ruled by lines in every case; sometimes they are ruled
by conics -- to which the methods of this paper do not apply.

\bibliography{my}
\bibliographystyle{alpha}

\end{document}